\documentclass[graybox,envcountsame,envcountresetchap]{svmult}

\usepackage{mathptmx}       
\usepackage{helvet}         
\usepackage{courier}        
\usepackage{type1cm}        
\usepackage{makeidx}         
\usepackage{graphicx}        
\usepackage{multicol}        
\usepackage[bottom]{footmisc}

\usepackage{amssymb}
\usepackage{amsmath}
\usepackage{amsfonts}
\usepackage{comment}

\usepackage{amssymb}
\usepackage{amsmath}
\usepackage{amsfonts}
\usepackage{comment}

\DeclareMathOperator{\PSL}{PSL}
\DeclareMathOperator{\SL}{SL}
\DeclareMathOperator{\GL}{GL}
\DeclareMathOperator{\asin}{asin}
\DeclareMathOperator{\acosh}{acosh}
\DeclareMathOperator{\disc}{disc}
\newcommand{\fp}{\qed}
\newcommand{\sump}{\sideset{}{'}\sum}
\newcommand{\psmm}[4]{\left(\begin{smallmatrix}{#1}&{#2}\\{#3}&{#4}\end{smallmatrix}\right)}

\renewcommand{\a}{\mathfrak a}
\newcommand{\ov}[1]{\overline{#1}}
\newcommand{\Q}{{\mathbb Q}}
\newcommand{\Z}{{\mathbb Z}}
\renewcommand{\P}{{\mathbb P}}
\newcommand{\F}{{\mathbb F}}
\newcommand{\R}{{\mathbb R}}
\newcommand{\C}{{\mathbb C}}

\renewcommand{\L}{\Lambda}
\newcommand{\oL}{\overline{\Lambda}}
\newcommand{\G}{\Gamma}
\renewcommand{\gg}{{\mathfrak g}}
\newcommand{\g}{{\mathfrak g}}
\newcommand{\p}{{\mathfrak p}}
\newcommand{\GP}{{\mathfrak P}}
\newcommand{\om}{\omega}

\newcommand{\e}{\mathbf e}
\newcommand{\al}{\alpha}
\newcommand{\be}{\beta}
\newcommand{\ga}{\gamma}

\newcommand{\la}{\lambda}
\renewcommand{\th}{\theta}
\newcommand{\Th}{\Theta}
\newcommand{\z}{\zeta}
\newcommand{\eps}{\varepsilon}
\renewcommand{\pmod}[1]{\allowbreak\ ({\rm{mod}}\,\,#1)}
\newcommand{\lgs}[2]{\mbox{$\left(\frac{#1}{#2}\right)$}}
\newcommand{\leg}[2]{\mbox{$\left(\dfrac{#1}{#2}\right)$}}
\newcommand{\isom}{\simeq}
\DeclareMathOperator{\Tr}{Tr}
\DeclareMathOperator{\Res}{Res}
\def\cal{\mathcal}
\DeclareMathOperator{\N}{{\cal N}}
\DeclareMathOperator{\M}{{\cal M}}
\DeclareMathOperator{\Os}{\widetilde {\mathit O}}

\makeatletter
\def\renewtheorem#1{%
  \expandafter\let\csname#1\endcsname\relax
  \expandafter\let\csname c@#1\endcsname\relax
  \gdef\renewtheorem@envname{#1}
  \renewtheorem@secpar
}
\def\renewtheorem@secpar{\@ifnextchar[{\renewtheorem@numberedlike}{\renewtheorem@nonumberedlike}}
  \def\renewtheorem@numberedlike[#1]#2{\newtheorem{\renewtheorem@envname}[#1]{#2}}
  \def\renewtheorem@nonumberedlike#1{
    \def\renewtheorem@caption{#1}
    \edef\renewtheorem@nowithin{\noexpand\newtheorem{\renewtheorem@envname}{\renewtheorem@caption}}
    \renewtheorem@thirdpar
  }
  \def\renewtheorem@thirdpar{\@ifnextchar[{\renewtheorem@within}{\renewtheorem@nowithin}}
    \def\renewtheorem@within[#1]{\renewtheorem@nowithin[#1]}
    \makeatother

\renewtheorem{theorem}{Theorem}[section]
\renewtheorem{remarks}[theorem]{{\it Remarks}}
\smartqed

\makeindex             

\begin{document}

\title*{Computational Number Theory in Relation with $L$-Functions}
\author{Henri Cohen}
\institute{Universit\'e de Bordeaux, CNRS, INRIA, IMB, UMR 5251, F-33400 Talence, France,
\email{Henri.Cohen@math.u-bordeaux.fr}}

%
%
\maketitle

\abstract{We give a number of theoretical and practical methods related to the 
computation of $L$-functions, both in the local case (counting points
on varieties over finite fields, involving in particular a detailed
study of Gauss and Jacobi sums), and in the global case (for instance
Dirichlet $L$-functions, involving in particular the study of inverse
Mellin transforms); we also give a number of little-known but very
useful numerical methods, usually but not always related to the computation
of $L$-functions.}

\section{$L$-Functions}\label{sec:one}

This course is divided into five parts. In the first part (Sections 1 and 2),
we introduce the notion of $L$-function, give a number of results and
conjectures concerning them, and explain some of the computational problems
in this theory. In the second part (Sections 3 to 6), we give a number of
computational methods for obtaining the Dirichlet series coefficients of the
$L$-function, so is \emph{arithmetic} in nature. In the third part
(Section 7), we give a number of \emph{analytic} tools necessary for working
with $L$-functions. In the fourth part (Sections 8 and 9), we give a number of
very useful numerical methods which are not sufficiently well-known, most of
which being also related to the computation of $L$-functions. The fifth part
(Sections 10 and 11) gives the {\tt Pari/GP} commands corresponding to most of
the algorithms and examples given in the course. A final Section 12 gives
as an appendix some basic definitions and results used in the course which
may be less familiar to the reader.

\subsection{Introduction}

The theory of $L$-functions is one of the most exciting subjects in number
theory. It includes for instance two of the crowning achievements of
twentieth century mathematics, first the proof of the Weil conjectures
and of the Ramanujan conjecture by Deligne in the early 1970's, using the
extensive development of modern algebraic geometry initiated by Weil himself
and pursued by Grothendieck and followers in the famous EGA and SGA treatises,
and second the proof of the Shimura--Taniyama--Weil conjecture by Wiles et
al., implying among other things the proof of Fermat's last theorem. It
also includes two of the seven 1 million dollar Clay problems for the
twenty-first century, first the Riemann hypothesis, and second the
Birch--Swinnerton-Dyer conjecture which in my opinion is the most beautiful,
if not the most important, conjecture in number theory, or even in the whole
of mathematics, together with similar conjectures such as the
Beilinson--Bloch conjecture.

There are two kinds of $L$-functions: local $L$-functions and global
$L$-functions. Since the proof of the Weil conjectures, local $L$-functions
are rather well understood from a theoretical standpoint, but somewhat less
from a computational standpoint. Much less is known on global $L$-functions,
even theoretically, so here the computational standpoint is much more
important since it may give some insight on the theoretical side.

Before giving a definition of $L$-functions, we look in some detail at a
large number of special cases of global $L$-functions.

\subsection{The Prototype: the Riemann Zeta Function $\z(s)$}

The simplest of all (global) $L$-function is the Riemann zeta function $\z(s)$
defined by
$$\z(s)=\sum_{n\ge1}\dfrac{1}{n^s}\;.$$
This is an example of a \emph{Dirichlet series} (more generally
$\sum_{n\ge1}a(n)/n^s$, or even more generally $\sum_{n\ge1}1/\la_n^s$, but
we will not consider the latter). As such, it has a half-plane of absolute
convergence, here $\Re(s)>1$.

The properties of this function, studied initially by Bernoulli
and Euler, are as follows, given historically:

\begin{enumerate}\item (Bernoulli, Euler): it has \emph{special values}.
When $s=2$, $4$,... is a strictly positive even integer, $\z(s)$ is equal
to $\pi^s$ times a \emph{rational number}. $\pi$ is here a \emph{period},
and is of course the usual $\pi$ used for measuring circles. These rational
numbers have elementary \emph{generating functions}, and are equal up to easy
terms to the so-called \emph{Bernoulli numbers}. For example
$\z(2)=\pi^2/6$, $\z(4)=\pi^4/90$, etc. This was conjectured by Bernoulli
and proved by Euler. Note that the proof in 1735 of the so-called
\emph{Basel problem}: 
$$\z(2)=1+\dfrac{1}{2^2}+\dfrac{1}{3^2}+\dfrac{1}{4^2}+\cdots=\dfrac{\pi^2}{6}$$
is one of the crowning achievements of mathematics of that time.
\item (Euler): it has an \emph{Euler product}: for $\Re(s)>1$ one has the
identity
$$\z(s)=\prod_{p\in P}\dfrac{1}{1-1/p^s}\;,$$
where $P$ is the set of prime numbers. This is exactly equivalent to the
so-called fundamental theorem of arithmetic. Note in passing (this does not
seem interesting here but will be important later) that if we consider
$1-1/p^s$ as a polynomial in $1/p^s=T$, its reciprocal roots all have the
same modulus, here $1$, this being of course trivial.
\item (Riemann, but already ``guessed'' by Euler in special cases): it has
an \emph{analytic continuation} to a meromorphic function in the whole complex
plane, with a single pole, at $s=1$, with residue $1$, and a \emph{functional
equation} $\Lambda(1-s)=\Lambda(s)$, where $\Lambda(s)=\G_{\R}(s)\z(s)$,
with $\G_{\R}(s)=\pi^{-s/2}\G(s/2)$, and $\G$ is the gamma function
(see appendix).
\item As a consequence of the functional equation, we have $\z(s)=0$
when $s=-2$, $-4$,..., $\z(0)=-1/2$, but we also have \emph{special
values} at $s=-1$, $s=-3$,... which are symmetrical to those at $s=2$, $4$,...
(for instance $\z(-1)=-1/12$, $\z(-3)=1/120$, etc.). This is the
part which was guessed by Euler.
\end{enumerate}

Roughly speaking, one can say that a global $L$-function is a function
having properties similar to \emph{all} the above. We will of course be
completely precise below. Two things should be added immediately: first, the
existence of special values will not be part of the definition but, at
least conjecturally, a consequence. Second, all the global $L$-functions
that we will consider should \emph{conjecturally} satisfy a Riemann hypothesis:
when suitably normalized, and excluding ``trivial'' zeros, all the zeros
of the function should be on the line $\Re(s)=1/2$, axis of symmetry of the
functional equation. Note that even for the simplest $L$-function, $\z(s)$,
this is not proved.

\subsection{Dedekind Zeta Functions}

The Riemann zeta function is perhaps too simple an example to get the correct
feeling about global $L$-functions, so we generalize:

Let $K$ be a number field (a finite extension of $\Q$) of degree $d$. We
can define its \emph{Dedekind zeta function} $\z_K(s)$ for $\Re(s)>1$ by
$$\z_K(s)=\sum_{\a}\dfrac{1}{\N(\a)^s}=\sum_{n\ge1}\dfrac{i(n)}{n^s}\;,$$
where $\a$ ranges over all (nonzero) integral ideals of the ring of integers
$\Z_K$ of $K$, $\N(\a)=[\Z_K:\a]$ is the norm of $\a$, and $i(n)$ denotes
the number of integral ideals of norm $n$.

This function has very similar properties to those of $\z(s)$ (which is the
special case $K=\Q$). We give them in a more logical order:

\begin{enumerate}\item It can be analytically continued to the whole complex 
plane into a meromorphic function having a single pole, at $s=1$, with known 
residue, and it has a functional equation $\Lambda_K(1-s)=\Lambda_K(s)$, where
$$\Lambda_K(s)=|D_K|^{s/2}\G_{\R}(s)^{r_1+r_2}\G_{\R}(s+1)^{r_2}\;,$$
where $(r_1,2r_2)$ are the number of real and complex embeddings of $K$
and $D_K$ its discriminant.
\item It has an Euler product $\z_K(s)=\prod_{\p}1/(1-1/\N(\p)^s)$, where
the product is over all prime ideals of $\Z_K$. Note that this can also be
written
$$\z_K(s)=\prod_{p\in P}\prod_{\p\mid p}\dfrac{1}{1-1/p^{f(\p/p)s}}\;,$$
where $f(\p/p)=[\Z_K/\p:\Z/p\Z]$ is the so-called \emph{residual index}
of $\p$ above $p$. Once again, note that if we set as usual $1/p^s=T$,
the reciprocal roots of $1-T^{f(\p/p)}$ all have modulus $1$.
\item It has \emph{special values}, but only when $K$ is a \emph{totally real}
number field ($r_2=0$, $r_1=d$): in that case $\z_K(s)$ is a \emph{rational
number} if $s$ is a negative odd integer, or equivalently by the functional
equation, it is a rational multiple of $\sqrt{|D_K|}\pi^{ds}$ if $s$ is a
positive even integer.\end{enumerate}

An important new phenomenon occurs: recall that 
$\sum_{\p\mid p}e(\p/p)f(\p/p)=d$, where $e(\p/p)$ is the so-called
\emph{ramification index}, which is equivalent to the
defining equality $p\Z_K=\prod_{\p\mid p}\p^{e(\p/p)}$. In particular
$\sum_{\p\mid p}f(\p/p)=d$ if and only if $e(\p/p)=1$ for all $\p$, which
means that $p$ is \emph{unramified} in $K/\Q$; one can prove that this is
equivalent to $p\nmid D_K$. Thus, the \emph{local $L$-function}
$L_{K,p}(T)=\prod_{\p\mid p}(1-T^{f(\p/p)})$ has degree in $T$ exactly
equal to $d$ for all but a finite number of primes $p$, which are exactly
those which divide the discriminant $D_K$, and for those ``bad'' primes
the degree is strictly less than $d$. In addition, note that the
number of $\G_{\R}$ factors in the \emph{completed} function $\Lambda_K(s)$
is equal to $r_1+2r_2$, hence once again equal to $d$.

\smallskip

{\bf Examples:} 

\begin{enumerate}\item Let $D$ be the discriminant of a quadratic field, and
let $K=\Q(\sqrt{D})$. In that case, $\z_K(s)$ \emph{factors} as
$\z_K(s)=\z(s)L(\chi_D,s)$, where $\chi_D=\lgs{D}{.}$ is the 
Legendre--Kronecker symbol, and $L(\chi_D,s)=\sum_{n\ge 1}\chi_D(n)/n^s$.
Thus, the local $L$-function at a prime $p$ is given by
$$L_{K,p}(T)=(1-T)(1-\chi_D(p)T)=1-a_pT+\chi_D(p)T^2\;,$$
with $a_p=1+\chi_D(p)$. Note that $a_p$ is equal to the number of solutions
in $\F_p$ of the equation $x^2=D$.
\item Let us consider two special cases of (1): first $K=\Q(\sqrt{5})$.
Since it is a real quadratic field, it has special values, for instance
$$\z_K(-1)=\dfrac{1}{30}\;,\quad \z_K(-3)=\dfrac{1}{60}\;,\quad \z_K(2)=\dfrac{2\sqrt{5}\pi^4}{375}\;,\quad \z_K(4)=\dfrac{4\sqrt{5}\pi^8}{84375}\;.$$
In addition, note that its \emph{gamma factor} is $5^{s/2}\G_{\R}(s)^2$.

Second, consider $K=\Q(\sqrt{-23})$. Since it is not a totally real field,
$\z_K(s)$ does not have special values. However, because of the factorization
$\z_K(s)=\z(s)L(\chi_D,s)$, we can look \emph{separately} at the special values
of $\z(s)$, which we have already seen (negative odd integers and positive
even integers), and of $L(\chi_D,s)$. It is easy to prove that the special
values of this latter function occurs at negative \emph{even} integers
and positive \emph{odd} integers, which have empty intersection which those
of $\z(s)$ and explains why $\z_K(s)$ itself has none. For instance,
$$L(\chi_D,-2)=-48\;,\quad L(\chi_D,-4)=6816\;,\quad L(\chi_D,3)=\dfrac{96\sqrt{23}\pi^3}{12167}\;.$$
In addition, note that its gamma factor is 
$$23^{s/2}\G_{\R}(s)\G_{\R}(s+1)=23^{s/2}\G_{\C}(s)\;,$$
where we set by definition 
$$\G_{\C}(s)=\G_{\R}(s)\G_{\R}(s+1)=2\cdot(2\pi)^{-s}\G(s)$$
by the duplication formula for the gamma function.
\item Let $K$ be the unique cubic field up to isomorphism of discriminant
$-23$, defined for instance by a root of the equation $x^3-x-1=0$. We
have $(r_1,2r_2)=(1,2)$ and $D_K=-23$. Here, one
can prove (it is less trivial) that $\z_K(s)=\z(s)L(\rho,s)$, where
$L(\rho,s)$ is a holomorphic function. Using both properties of $\z_K$ and
$\z$, this $L$-function has the following properties:
\begin{itemize}\item It extends to an entire function on $\C$ with a functional
equation $\Lambda(\rho,1-s)=\Lambda(\rho,s)$, with
$$\Lambda(\rho,s)=23^{s/2}\G_{\R}(s)\G_{\R}(s+1)L(\rho,s)=23^{s/2}\G_{\C}(s)L(\rho,s)\;.$$
Note that this is the \emph{same} gamma factor as for $\Q(\sqrt{-23})$.
However the functions are fundamentally different, since 
$\z_{\Q(\sqrt{-23})}(s)$ has a pole at $s=1$, while $L(\rho,s)$ is an
entire function.
\item It is immediate to show that if we let $L_{\rho,p}(T)=L_{K,p}(T)/(1-T)$
be the local $L$ function for $L(\rho,s)$, we have
$L_{\rho,p}(T)=1-a_pT+\chi_{-23}(p)T^2$, with
$a_p=1$ if $p=23$, $a_p=0$ if $\lgs{-23}{p}=-1$, and
$a_p=1$ or $2$ if $\lgs{-23}{p}=1$.
\end{itemize}
\end{enumerate}

\begin{remark} In all of the above examples, the function $\z_K(s)$
is \emph{divisible} by the Riemann zeta function $\z(s)$, i.e., the function
$\z_K(s)/\z(s)$ is an \emph{entire function}. This is known for some number
fields $K$, but is \emph{not} known in general, even in degree $d=5$ for
instance: it is a consequence of the more precise \emph{Artin conjecture} on
the holomorphy of Artin $L$-functions.\end{remark}

\subsection{Further Examples in Weight $0$}

It is now time to give examples not coming from number fields.
Define $a_1(n)$ by the formal equality
$$q\prod_{n\ge1}(1-q^n)(1-q^{23n})=\sum_{n\ge1}a_1(n)q^n=q-q^2-q^3+q^6+q^8-\cdots\;,$$
and set $L_1(s)=\sum_{n\ge1}a_1(n)/n^s$. The theory of modular forms
(here of the Dedekind eta function) tells us that $L_1(s)$ will satisfy 
exactly the same properties as $L(\rho,s)$ with $\rho$ as above.

Define $a_2(n)$ by the formal equality
$$\dfrac{1}{2}\left(\sum_{(m,n)\in\Z\times\Z}q^{m^2+mn+6n^2}-q^{2m^2+mn+3n^2}\right)=\sum_{n\ge1}a_2(n)q^n\;,$$
and set $L_2(s)=\sum_{n\ge1}a_2(n)/n^s$. The theory of modular forms
(here of theta functions) tells us that $L_2(s)$ will satisfy 
exactly the same properties as $L(\rho,s)$. 

And indeed, it is an interesting \emph{theorem} that
$$L_1(s)=L_2(s)=L(\rho,s)\;:$$
The ``moral'' of this story is the following, which can be made mathematically
precise: if two $L$-functions are holomorphic, have the same gamma factor 
(including in this case the $23^{s/2}$), then (conjecturally in general) they
belong to a finite-dimensional vector space. Thus in particular if this vector
space is $1$-dimensional and the $L$-functions are suitably normalized 
(usually with $a(1)=1$), this implies as here that they are equal.

\subsection{Examples in Weight $1$}

Although we have not yet defined the notion of weight, let me give two
further examples.

Define $a_3(n)$ by the formal equality
$$q\prod_{n\ge1}(1-q^n)^2(1-q^{11n})^2=\sum_{n\ge1}a_3(n)q^n=q-2q^2-q^3+2q^4+\cdots\;,$$
and set $L_3(s)=\sum_{n\ge1}a_3(n)/n^s$. The theory of modular forms
(again of the Dedekind eta function) tells us that $L_3(s)$ will satisfy
the following properties, analogous but more general than those satisfied by
$L_1(s)=L_2(s)=L(\rho,s)$:
\begin{itemize}\item It has an analytic continuation to the whole complex
plane, and if we set 
$$\Lambda_3(s)=11^{s/2}\G_{\R}(s)\G_{\R}(s+1)L_3(s)=11^{s/2}\G_{\C}(s)L_3(s)\;,$$
we have the functional equation $\Lambda_3(2-s)=\Lambda_3(s)$. Note the crucial
difference that here $1-s$ is replaced by $2-s$. 
\item There exists an Euler product $L_3(s)=\prod_{p\in P}1/L_{3,p}(1/p^s)$
similar to the preceding ones in that $L_{3,p}(T)$ is for all but a finite
number of $p$ a second degree polynomial in $T$. More precisely,
if $p=11$ we have $L_{3,p}(T)=1-T$, while for $p\ne11$ we have
$L_{3,p}(T)=1-a_pT+pT^2$, for some $a_p$ such that $|a_p|<2\sqrt{p}$.
This is expressed more vividly by saying that for $p\ne11$ we have
$L_{3,p}(T)=(1-\al_pT)(1-\be_pT)$, where the reciprocal roots $\al_p$ and
$\be_p$ have modulus exactly equal to $p^{1/2}$. Note again the crucial
difference with ``weight $0$'' in that the coefficient of $T^2$ is equal to
$p$ instead of $\pm1$, hence that $|\al_p|=|\be_p|=p^{1/2}$ instead of $1$.
\end{itemize}

\smallskip

As a second example, consider the equation $y^2+y=x^3-x^2-10x-20$ (an
elliptic curve $E$), and denote by $N_q(E)$ the number of projective points
of this curve over the finite field $\F_q$ (it is clear that there is a unique
point at infinity, so if you want $N_q(E)$ is one plus the number of affine
points). There is a universal recipe to construct an $L$-function out of
a variety which we will recall below, but here let us simplify: for $p$
prime, set $a_p=p+1-N_p(E)$ and 
$$L_4(s)=\prod_{p\in P}1/(1-a_pp^{-s}+\chi(p)p^{1-2s})\;,$$
where $\chi(p)=1$ for $p\ne11$ and $\chi(11)=0$. It is not difficult to show
that $L_4(s)$ satisfies exactly the same properties as $L_3(s)$ (using for
instance the elementary theory of modular curves), so by the moral explained
above, it should not come as a surprise that in fact $L_3(s)=L_4(s)$.

\subsection{Definition of a Global $L$-Function}

With all these examples at hand, it is quite natural to give the following
definition of an $L$-function, which is not the most general but will be
sufficient for us.

\begin{definition} Let $d$ be a nonnegative integer. We say that a
Dirichlet series $L(s)=\sum_{n\ge1}a(n)n^{-s}$ with $a(1)=1$ is an 
$L$-function of \emph{degree $d$} and \emph{weight $0$} if the following 
conditions are satisfied:
\begin{enumerate}\item (Ramanujan bound): we have $a(n)=O(n^\eps)$ for
all $\eps>0$, so that in particular the Dirichlet series converges
absolutely and uniformly in any half plane $\Re(s)\ge\sigma>1$.
\item (Meromorphy and Functional equation): The function $L(s)$ can be
extended to $\C$ to a meromorphic function of order $1$ (see appendix) having
a finite number of poles; furthermore there exist complex numbers $\la_i$ with 
nonnegative real part and an integer $N$ called the \emph{conductor} such
that if we set 
$$\ga(s)=N^{s/2}\prod_{1\le i\le d}\G_{\R}(s+\la_i)\text{\quad and\quad}\Lambda(s)=\ga(s)L(s)\;,$$
we have the \emph{functional equation}
$$\Lambda(s)=\om\ov{\Lambda(1-\ov{s})}$$
for some complex number $\om$, called the \emph{root number}, which will 
necessarily be of modulus~$1$.
\item (Euler Product): For $\Re(s)>1$ we have an Euler product
$$L(s)=\prod_{p\in P}1/L_p(1/p^s)\text{\quad with\quad}L_p(T)=\prod_{1\le j\le d}(1-\al_{p,j}T)\;,$$
and the reciprocal roots $\al_{p,j}$ are called the \emph{Satake parameters}.
\item (Local Riemann hypothesis): for $p\nmid N$ we have $|\al_{p,j}|=1$,
and for $p\mid N$ we have either $\al_{p,j}=0$ or $|\al_{p,j}|=p^{-m/2}$
for some $m$ such that $1\le m\le d$.
\end{enumerate}\end{definition}

\begin{remarks}{\rm \begin{enumerate}
\item More generally Selberg has defined a more general class of $L$-functions
which first allows $\G(\mu_i s+\la_i)$ with $\mu_i$ positive real in the gamma factors and second allows weaker assumptions on $N$ and the Satake parameters.
\item Note that $d$ is \emph{both} the number of $\G_{\R}$ factors,
\emph{and} the degree in $T$ of the Euler factors $L_p(T)$, at
least for $p\nmid N$, while the degree decreases for the ``bad'' primes $p$
which divide $N$.
\item The Ramanujan bound (1) is easily seen to be a consequence of the 
conditions that we have imposed on the Satake parameters: in Selberg's more 
general definition this is not the case.
\end{enumerate}}
\end{remarks}

It is important to generalize this definition in the following trivial way:

\begin{definition} Let $w$ be a nonnegative integer. A function $L(s)$ is said
to be an $L$-function of degree $d$ and \emph{motivic weight} $w$ if
$L(s+w/2)$ is an $L$-function of degree $d$ and weight $0$ as above
(with the slight additional technical condition that the nonzero Satake
parameters $\al_{p,j}$ for $p\mid N$ satisfy $|\al_{p,j}|=p^{-m/2}$ with
$1\le m\le w$).
\end{definition}

For an $L$-function of weight $w$, it is clear that the functional equation
is $\Lambda(s)=\om\ov{\Lambda(k-\ov{s})}$ with $k=w+1$,
and that the Satake parameters will satisfy $|\al_{p,j}|=p^{w/2}$ for
$p\nmid N$, and for $p\mid N$ we have either $\al_{p,j}=0$ or
$|\al_{p,j}|=p^{(w-m)/2}$ for some integer $m$ such that $1\le m\le w$.

Thus, the first examples that we have given are all of weight $0$, and
the last two (which are in fact equal) are of weight $1$. For those who
know the theory of modular forms, note that the motivic weight (that we
denote by $w$) is one less than the weight $k$ of the modular form.

\section{Origins of $L$-Functions}

As can already be seen in the above examples, it is possible to construct
$L$-functions in many different ways. In the present section, we look at three
different ways for constructing $L$-functions: the first is by the theory of
modular forms or more generally of \emph{automorphic forms} (of which we have 
seen a few examples above), the second is by using Weil's construction of
local $L$-functions attached to varieties, and more generally
to \emph{motives}, and third, as a special but much simpler case of this,
by the theory of \emph{hypergeometric motives}.

\subsection{$L$-Functions coming from Modular Forms}

The basic notion that we need here is that of \emph{Mellin transform}:
if $f(t)$ is a nice function tending to zero exponentially fast at infinity,
we can define its Mellin transform $\Lambda(f;s)=\int_0^\infty t^sf(t)\,dt/t$,
the integral being written in this way because $dt/t$ is the invariant Haar
measure on the locally compact group $\R_{>0}$. If we set $g(t)=t^{-k}f(1/t)$
and assume that $g$ also tends to zero exponentially fast at infinity,
it is immediate to see by a change of variable that 
$\Lambda(g;s)=\Lambda(f;k-s)$. This is exactly the type of functional equation
needed for an $L$-function.

The other fundamental property of $L$-functions that we need is the existence
of an Euler product of a specific type. This will come from the theory of
\emph{Hecke operators}. 

\smallskip

{\bf A crash course in modular forms} (see for instance \cite{Coh-Str} for a
complete introduction): we use the notation $q=e^{2\pi i\tau}$,
for $\tau\in\C$ such that $\Im(\tau)>0$, so that $|q|<1$. A function
$f(\tau)=\sum_{n\ge1}a(n)q^n$ is said to be a modular cusp form of (positive,
even) weight $k$ if $f(-1/\tau)=\tau^kf(\tau)$ for all $\Im(\tau)>0$.
Note that because of the notation $q$ we also have $f(\tau+1)=f(\tau)$,
hence it is easy to deduce that $f((a\tau+b)/(c\tau+d))=(c\tau+d)^kf(\tau)$
if $\psmm{a}{b}{c}{d}$ is an integer matrix of determinant $1$.
We define the $L$-function attached to $f$ as $L(f;s)=\sum_{n\ge1}a(n)/n^s$,
and the Mellin transform $\Lambda(f;s)$ of the function $f(it)$ is on the
one hand equal to $(2\pi)^{-s}\G(s)L(f;s)=(1/2)\G_{\C}(s)L(f;s)$, and on the
other hand as we have seen above satisfies the functional equation 
$\Lambda(k-s)=(-1)^{k/2}\Lambda(s)$.

One can easily show the fundamental fact that the vector space of modular
forms of given weight $k$ is \emph{finite dimensional}, and compute its 
dimension explicitly.

If $f(\tau)=\sum_{n\ge1}a(n)q^n$ is a modular form and $p$ is a prime number,
one defines $T(p)(f)$ by $T(p)(f)=\sum_{n\ge1}b(n)q^n$ with
$b(n)=a(pn)+p^{k-1}a(n/p)$, where $a(n/p)$ is by convention $0$ when 
$p\nmid n$, or equivalently 
$$T(p)(f)(\tau)=p^{k-1}f(p\tau)+\dfrac{1}{p}\sum_{0\le j<p}f\left(\dfrac{\tau+j}{p}\right)\;.$$
Then $T(p)f$ is also a modular cusp form, so $T(p)$ is an operator on the 
space of modular forms, and it is easy to show that the $T(p)$ commute and
are diagonalizable, so they are simultaneously diagonalizable hence there
exists a basis of common \emph{eigenforms} for all the $T(p)$. Since one can
show that for such an eigenform one has $a(1)\ne0$, we can normalize them
by asking that $a(1)=1$, and we then obtain a canonical basis.

If $f(\tau)=\sum_{n\ge1}a(n)q^n$ is such a \emph{normalized eigenform}, it
follows that the corresponding $L$ function $\sum_{n\ge1}a(n)/n^s$ will indeed
have an Euler product, and using the elementary properties of the operators
$T(p)$ that it will in fact be of the form:
$$L(f;s)=\prod_{p\in P}\dfrac{1}{1-a(p)p^{-s}+p^{k-1-2s}}\;.$$
As a final remark, note that the analytic continuation and functional equation
of this $L$-function is an \emph{elementary consequence} of the definition of
a modular form. This is totally different from the motivic cases that we will 
see below, where this analytic continuation is in general completely
\emph{conjectural}.

\smallskip

The above describes briefly the theory of modular forms on the modular group
$\PSL_2(\Z)$. One can generalize (nontrivially) this theory to \emph{subgroups}
of the modular group, the most important being $\G_0(N)$ (matrices as above
with $N\mid c$), to other \emph{Fuchsian groups}, to forms in several
variables, and even more generally to \emph{reductive groups}.

\subsection{Local $L$-Functions of Algebraic Varieties}

The second very important source of $L$-functions comes from algebraic
geometry. Let $V$ be some algebraic object. In modern terms, $V$ may be a
\emph{motive}, whatever that may mean for the moment, but assume for instance
that $V$ is an algebraic variety, in other words that for each suitable field
$K$, $V(K)$ is the set of common zeros of a family of polynomials in several
variables. If $K$ is a \emph{finite} field $\F_q$ (recall that we must then
have $q=p^n$ for some prime $p$ and that $\F_q$ exists and is unique up to
isomorphism), then $V(\F_q)$ will also be finite.

After studying a number of special cases, such as elliptic curves
(due to Hasse), and quasi-diagonal hypersurfaces in $\P^d$, in 1949 Weil was 
led to make a number of more precise conjectures concerning the number of
\emph{projective} points $|V(\F_q)|$, assuming that $V$ is a
\emph{smooth projective} variety, and proved these conjectures in the special
case of curves (the proof is already quite deep).

\smallskip

The first \emph{Weil conjecture} says that (for $p$ fixed) the number
$|V(\F_{p^n})|$ of projective points of $V$ over the finite field $\F_{p^n}$
satisfies a (non-homogeneous) linear recurrence with 
constant coefficients. For instance, if $V$ is an \emph{elliptic curve} 
defined over $\Q$ (such as $y^2=x^3+x+1$) and if we set 
$a(p^n)=p^n+1-|V(\F_{p^n})|$, then
$$a(p^{n+1})=a(p)a(p^n)-\chi(p)pa(p^{n-1})\;,$$
where $\chi(p)=1$ unless $p$ divides the so-called \emph{conductor} of the
elliptic curve, in which case $\chi(p)=0$ (this is not quite true because
we must choose a suitable model for $V$, but it suffices for us).

\begin{exercise} Using the above recursion for $a(p^n)$, find the corresponding
recursion for $v_n=|V(\F_{p^n})|$.
\end{exercise}

\begin{exercise}\begin{enumerate}\item Given a prime $p$ and $n\ge1$, write a 
computer program which runs through all the elements of $\F_{p^n}$, 
represented in a suitable way.
\item For the elliptic curve $y^2=x^3+x+1$, compute (on a computer) $a(5)$ 
and $a(5^2)$, and check the recursion.
\item Similarly, compute $a(31)$ and $a(31^2)$, and check the recursion
(here $\chi(31)=0$).\end{enumerate}
\end{exercise}

This first Weil conjecture was proved by Dwork in the early 1960's. It is
better reformulated in terms of \emph{local $L$-functions} as follows:
define the Hasse--Weil zeta function of $V$ as the \emph{formal power series}
in $T$ given by the formula
$$Z_p(V;T)=\exp\Biggl(\sum_{n\ge1}\dfrac{|V(\F_{p^n})|}{n}T^n\Biggr)\;.$$
There should be no difficulty in understanding this: setting for simplicity
$v_n=|V(\F_{p^n})|$, we have
\begin{align*}
Z_p(V;T)&=\exp(v_1T+v_2T^2/2+v_3T^3/3+\cdots)\\
&=1+v_1T+(v_1^2+v_2)T^2/2+(v_1^3+3v_1v_2+2v_3)T^3/6+\cdots\end{align*}
For instance, if $V$ is projective $d$-space $\P^d$, we have
$|V(\F_q)|=q^d+q^{d-1}+\cdots+1$, and since 
$\sum_{n\ge1}p^{nj}T^n/n=-\log(1-p^jT)$, we deduce that
$Z_p(\P^d;T)=1/((1-T)(1-pT)\cdots(1-p^dT))$.

In terms of this language, the existence of the recurrence relation is
equivalent to the fact that $Z_p(V;T)$ is a \emph{rational function} of $T$,
and as already mentioned, this was proved by Dwork in 1960.

\smallskip

The second conjecture of Weil states that this rational function is of the form
$$Z_p(V;T)=\prod_{0\le i\le 2d}P_{i,p}(V;T)^{(-1)^{i+1}}=\dfrac{P_{1,p}(V;T)\cdots P_{2d-1,p}(V;T)}{P_{0,p}(V;T)P_{2,p}(V;T)\cdots P_{2d,p}(V;T)}\;,$$
where $d=\dim(V)$, and the $P_{i,p}$ are polynomials in $T$.
Furthermore, a basic result in algebraic geometry called Poincar\'e duality
implies that $Z_p(V;1/(p^dT))=\pm p^{de/2}T^eZ_p(V;T)$, where $e$ is the
degree of the rational function (called the Euler characteristic of $V$),
which means that there is a relation between $P_{i,p}$ and $P_{2d-i,p}$.
In addition the $P_{i,p}$ have integer coefficients, and $P_{0,p}(T)=1-T$, 
$P_{2d,p}(T)=1-p^dT$. For instance, for \emph{curves}, this means that
$Z_p(V;T)=P_1(V;T)/((1-T)(1-pT))$, the polynomial $P_1$ is of even degree
$2g$ ($g$ is the so-called \emph{genus} of the curve) and satisfies
$p^{dg}P_1(V;1/(p^dT))=\pm P_1(V;T)$.

For knowledgeable readers, in highbrow language, the polynomial $P_{i,p}$ is 
the reverse characteristic polynomial of the Frobenius endomorphism acting on
the $i$th $\ell$-adic cohomology group $H^i(V;\Q_{\ell})$ for any $\ell\ne p$.

\smallskip

The third, most important and most difficult of the Weil conjectures is the
local \emph{Riemann hypothesis}, which says that the reciprocal roots of
$P_{i,p}$ have modulus exactly equal to $p^{i/2}$, in other words that
$$P_{i,p}(V;T)=\prod_j(1-\al_{i,j}T)\text{\quad with\quad}|\al_{i,j}|=p^{i/2}\;.$$
This last is the most important in applications.

The Weil conjectures were completely proved by Deligne in the early 1970's
following a strategy already put forward by Weil, and is considered as one of
the two or three major accomplishments of mathematics of the second half of
the twentieth century.

\begin{exercise} (You need to know some algebraic number theory for this).
Let $P\in\Z[X]$ be a monic irreducible polynomial and $K=\Q(\th)$, where
$\th$ is a root of $P$ be the corresponding number field. Assume that
$p^2\nmid\disc(P)$. Show that the Hasse--Weil zeta function at $p$ of the
$0$-dimensional variety defined by $P=0$ is the Euler factor at $p$ of
the Dedekind zeta function $\z_K(s)$ attached to $K$, where $p^{-s}$ is
replaced by $T$.\end{exercise}

\subsection{Global $L$-Function Attached to a Variety}

We are now ready to ``globalize'' the above construction, and build
\emph{global} $L$-functions attached to a variety.

Let $V$ be an algebraic variety defined over $\Q$, say. We assume that
$V$ is ``nice'', meaning for instance that we choose $V$ to be projective,
smooth, and absolutely irreducible. For all but a finite number of primes $p$ 
we can consider $V$ as a smooth variety over $\F_p$, so for each $i$ we can 
set $L_i(V;s)=\prod_p 1/P_{i,p}(V;p^{-s})$, where the product
is over all the ``good'' primes, and the $P_{i,p}$ are as above. The factor
$1/P_{i,p}(V;p^{-s})$ is as usual called the Euler factor at $p$. These 
functions $L_i$ can be called the global $L$-functions attached to $V$.

This na\"\i ve definition is insufficient to construct interesting objects.
First and most importantly, we have omitted a finite number of
Euler factors at the so-called ``bad primes'', which include in particular
those for which $V$ is not smooth over $\F_p$, and although there do
exist cohomological recipes to define them, as far as the author is aware 
these recipes do not really give practical algorithms. (In highbrow language,
these recipes are based on the computation of $\ell$-adic cohomology groups,
for which the known algorithms are useless in practice; in the simplest case
of Artin $L$-functions, one must determine the action of Frobenius on the
vector space fixed by the inertia group, which can be done reasonably easily.)

\smallskip

Another much less important reason is the fact that most of the $L_i$ are
uninteresting or related. For instance in the case of elliptic curves seen
above, we have (up to a finite number of Euler factors)
$L_0(V;s)=\z(s)$ and $L_2(V;s)=\z(s-1)$, so the only interesting $L$-function,
called \emph{the} $L$-function of the elliptic curve, is the function
$L_1(V;s)=\prod_p(1-a(p)p^{-s}+\chi(p)p^{1-2s})^{-1}$ (if the model of
the curve is chosen to be \emph{minimal}, this happens to be the correct
definition, including for the ``bad'' primes). For varieties of higher
dimension $d$, as we have mentioned as part of the Weil conjecture
the functions $L_i$ and $L_{2d-i}$ are related by Poincar\'e duality, and 
$L_0$ and $L_{2d}$ are translates of the Riemann zeta function (as above), so 
only the $L_i$ for $1\le i\le d$ need to be studied.

\subsection{Hypergeometric Motives}

Still another way to construct $L$-functions is through the use of
\emph{hypergeometric motives}, due to Katz and Rodriguez-Villegas. Although
this construction is a special case of the construction of $L$-functions of
varieties studied above, the corresponding variety is \emph{hidden} (although
it can be recovered if desired), and the computations are in some sense much
simpler.

Let me give a short and unmotivated introduction to the subject: let 
$\ga=(\ga_n)_{n\ge1}$ be a finite sequence of (positive or negative) integers
satisfying the essential condition $\sum_nn\ga_n=0$.
For any finite field $\F_q$ with $q=p^f$ and any character $\chi$ of $\F_q^*$,
recall that the Gauss sum $\gg(\chi)$ is defined by
$$\gg(\chi)=\sum_{x\in\F_q^*}\chi(x)\exp(2\pi i\Tr_{\F_q/\F_p}(x)/p)\;,$$
see Section \ref{sec:gausssum} below. We set
$$Q_q(\ga;\chi)=\prod_{n\ge1}\gg(\chi^n)^{\ga_n}$$
and for any $t\in\F_q\setminus\{0,1\}$
$$a_q(\ga;t)=\dfrac{1}{1-q}\left(1+\sum_{\chi\ne\eps}\chi(Mt)Q_q(\ga;\chi)\right)\;,$$
where $\eps$ is the trivial character and $M=\prod_nn^{n\ga_n}$ is a
normalizing constant (this is not quite the exact formula but it will
suffice for our purposes). The theorem of Katz is that for $t\ne0,1$ the
quantity $a_q(\ga;t)$ is the \emph{trace of Frobenius} on some \emph{motive}
\emph{defined over $\Q$}.
In the language of $L$-functions this means the following:
define as usual the local $L$-function at $p$ by the formal power series
$$L_p(\ga;t;T)=\exp\left(\sum_{f\ge1}a_{p^f}(\ga;t)\dfrac{T^f}{f}\right)\;.$$
Then $L_p$ is a rational function of $T$, satisfies the local Riemann 
hypothesis, and if we set
$$L(\ga;t;s)=\prod_pL_p(\ga;t;p^{-s})^{-1}\;,$$
then $L$ once completed at the ``bad'' primes should be a global $L$-function
of the standard type described above.

\smallskip

Let me give one of the simplest examples of a hypergeometric motive, and show
how one can recover the underlying algebraic variety. We choose
$\ga_1=4$, $\ga_2=-2$, $\ga_n=0$ for $n>2$, which does satisfy the condition
$\sum_nn\ga_n=0$ (we could choose the simpler values $\ga_1=2$, $\ga_2=-1$,
but this would give a zero-dimensional variety, i.e., a number field, so
less representative of the general case). We thus have
$Q_q(\ga,\chi)=\gg(\chi)^4/\gg(\chi^2)^2$ and $M=1/4$. By the results on
Jacobi sums that we will see below (Proposition \ref{jacgaufq}), if $\chi^2$
is not the trivial character $\eps$ we have $Q_q(\ga,\chi)=J(\chi,\chi)^2$,
where $J(\chi,\chi)=\sum_{x\in\F_q\setminus\{0,1\}}\chi(x)\chi(1-x)$. As
mentioned above, we did not give the precise formula, here it simply
corresponds to setting $Q_q(\ga,\chi)=J(\chi,\chi)^2$, including when
$\chi^2=\eps$. Thus
$$a_q(\ga;t)=\dfrac{1}{1-q}\left(1+\sum_{\chi\ne\eps}\chi(t/4)J(\chi,\chi)^2\right)\;.$$
If by a temporary abuse of notation\footnote{The definition of $J$ given
below is a sum over all $x\in\F_q$, so that $J(\eps,\eps)=q^2$ and not
$(q-2)^2$.} we define $J(\eps,\eps)$ by the same formula as above, we have
$J(\eps,\eps)=(q-2)^2$ hence
$$a_q(\ga;t)=\dfrac{1}{1-q}\left(1-(q-2)^2+\sum_{\chi}\chi(t/4)J(\chi,\chi)^2\right)\;.$$
Now
$$\sum_{\chi}\chi(t/4)J(\chi,\chi)^2=\sum_{x,y\in\F_q\setminus\{0,1\}}\sum_{\chi}\chi(t/4)\chi(x)\chi(1-x)\chi(y)\chi(1-y)\;.$$
The point of writing it this way is that because of orthogonality of
characters (Exercise \ref{exoorth} below) the sum on $\chi$ vanishes unless
the argument is equal to $1$ in which case it is equal to $q-1$, so that
$$\sum_{\chi}\chi(t/4)J(\chi,\chi)^2=(q-1)N_q(t)\;,\text{\quad where\quad }N_q(t)=\sum_{\substack{x,y\in\F_q\setminus\{0,1\}\\(t/4)x(1-x)y(1-y)=1}}1$$
is the number of \emph{affine} points over $\F_q$ of the algebraic variety
defined by $(t/4)x(1-x)y(1-y)=1$ (which automatically implies $x$ and $y$
different from $0$ and $1$). We have thus shown that
$$a_q(\ga;t)=\dfrac{1}{1-q}(1-(q-2)^2+(q-1)N_q(t))=q-3-N_q(t)\;.$$

\begin{exercise} By making the change of variables $X=(4/t)(1-1/x)$,
$Y=(4/t)(y-1)(1-1/x)$, show that
$$a_q(\ga;t)=q+1-|E(\F_q)|\;,$$
where $|E(\F_q)|$ is the number of projective points over $\F_q$ of the
elliptic curve $Y^2+XY=X(X-4/t)^2$. Thus, the global $L$-function
attached to the hypergeometric motive defined by $\ga$ is equal to
the $L$-function attached to the elliptic curve $E$.
\end{exercise}

Since we will see below fast methods for computing expressions such as\newline
$\sum_{\chi}\chi(t/4)J(\chi,\chi)^2$, these will consequently give fast
methods for computing $|E(\F_q)|$ for an arbitrary elliptic curve $E$.

\begin{exercise}\begin{enumerate}
\item In a similar way, study the hypergeometric motive
corresponding to $\ga_1=3$, $\ga_3=-1$, and $\ga_n=0$ otherwise,
assuming that the correct formula for $Q_q$ corresponds as above to
the replacement of quotients of Gauss sums by Jacobi sums for all
characters $\chi$, not only those allowed by Proposition \ref{jacgaufq}.
To find the elliptic curve, use the change of variable $X=-xy$,
$Y=x^2y$.
\item Deduce that the global $L$-function of this hypergeometric motive
is equal to the $L$-function attached to the elliptic curve
$y^2=x^3+x^2+4x+4$ and to the $L$-function attached to the modular form
$q\prod_{n\ge1}(1-q^{2n})^2(1-q^{10n})^2$.
\end{enumerate}
\end{exercise}

\subsection{Other Sources of $L$-Functions}

There exist many other sources of $L$-functions in addition to those that we
have already mentioned, that we will not expand upon:

\begin{itemize}
\item Hecke $L$-functions, attached to Hecke Gr\"ossencharacters.
\item Artin $L$-functions, of which we have met a couple of examples in
Section \ref{sec:one}.
\item Functorial constructions of $L$-functions such as Rankin--Selberg
$L$-functions, symmetric squares and more generally symmetric powers.
\item $L$-functions attached to Galois representations.
\item General automorphic $L$-functions.
\end{itemize}

Of course these are not disjoint sets, and as already mentioned, when some
$L$-functions lies in an intersection, this usually corresponds to an
interesting arithmetic property. Probably the most general such correspondence
is the \emph{Langlands program}.

\subsection{Results and Conjectures on $L(V;s)$}

The problem with global $L$-functions is that most of their properties are only
\emph{conjectural}. We mention these conjectures in the case of global
$L$-functions attached to algebraic varieties:

\smallskip

\begin{enumerate}\item The function $L_i$ is only defined through its Euler 
product, and thanks to the last of Weil's conjectures, the local Riemann 
hypothesis, proved by Deligne, it converges absolutely for $\Re(s)>1+i/2$.
Note that, with the definitions introduced above, $L_i$ is an $L$-function
of degree $d_i$, the common degree of $P_{i,p}$ for all but a finite number of
$p$, and of motivic weight exactly $w=i$ since the Satake parameters satisfy
$|\al_{i,p}|=p^{i/2}$, again by the local Riemann hypothesis.
\item A first conjecture is that $L_i$ should have an
\emph{analytic continuation} to the whole complex plane with a 
\emph{finite number} of \emph{known} poles with \emph{known} polar part. 
\item A second conjecture, which can in fact be considered as part of the
first, is that this extended $L$-function should satisfy a \emph{functional
equation} when $s$ is changed into $i+1-s$. More precisely, when completed
with the Euler factors at the ``bad'' primes as mentioned (but not explained)
above, then if we set
$$\Lambda_i(V;s)=N^{s/2}\prod_{1\le j\le d_i}\G_{\R}(s+\mu_j)L_i(V;s)$$
then $\Lambda_i(V;i+1-s)=\om\ov{\Lambda_i(V^*;s)}$ for some variety $V^*$ in 
some sense ``dual'' to $V$ and a complex number $\om$ of modulus $1$. In the 
above, $N$ is some integer divisible exactly by all the ``bad'' primes, i.e., 
essentially (but not exactly) the primes for which $V$ reduced modulo $p$ is 
not smooth, and the $\mu_j$ are in this case (varieties) \emph{integers}
which can be computed in terms of the \emph{Hodge numbers} $h^{p,q}$ of
the variety thanks to a recipe due to Serre \cite{Ser}. The number $i$ is
called the \emph{motivic weight}, and it is important to note that the
``weight'' $k$ usually attached to an $L$-function with functional equation
$s\mapsto k-s$ is equal to $k=i+1$, i.e., to \emph{one more} than the motivic
weight.

In many cases the $L$-function is self-dual, in which case the functional
equation is simply of the form $\Lambda_i(V;i+1-s)=\pm\Lambda_i(V;s)$.
\item The function $\Lambda_i$ should satisfy the generalized Riemann 
hypothesis (GRH): all its zeros in $\C$ are on the vertical line
$\Re(s)=(i+1)/2$. Equivalently, the zeros of $L_i$ are on the one hand
real zeros at some integers coming from the poles of the gamma factors,
and all the others satisfy $\Re(s)=(i+1)/2$.
\item The function $\Lambda_i$ should have \emph{special values}: for the
integer values of $s$ (called special points) which are those for which
neither the gamma factor at $s$ nor at $i+1-s$ has a pole, it should be 
computable ``explicitly'': it should be equal to a \emph{period}
(integral of an algebraic function on an algebraic cycle) times an algebraic
number. This has been stated (conjecturally) in great detail by Deligne in the
1970's.\end{enumerate}

It is conjectured that \emph{all} $L$-functions of degree $d_i$ and weight
$i$ as defined at the beginning should satisfy all the above properties, not
only the $L$-functions coming from varieties.

\smallskip

I now give the status of these conjectures.

\begin{enumerate}\item The first conjecture (analytic continuation) is known 
only for a very restricted class of $L$-functions: first $L$-functions of
degree $1$, which can be shown to be Dirichlet $L$-functions, $L$-functions of
Hecke characters, $L$-functions attached to modular forms as shown above, and
more generally to \emph{automorphic forms}. For $L$-functions attached to
varieties, one knows this \emph{only} when one can prove that the
corresponding $L$-function comes from an automorphic form: this is how Wiles
proves the analytic continuation of the $L$-function attached to an elliptic
curve defined over $\Q$, a very deep and difficult 
result, with Deligne's proof of the Weil conjectures one of the most important
result of the end of the 20th century. More results of this type are known for
certain higher-dimensional varieties such as certain \emph{Calabi--Yau
manifolds}. Note however that for such simple objects as most
\emph{Artin $L$-functions} (degree $0$, in which case only \emph{meromorphic}
continuation is known) or abelian surfaces, this is not
known, although the work of Brumer--Kramer--Poor--Yuen, as well as more
recent work of G.~Boxer, F.~Calegari, T.~Gee, and V.~Pilloni on the
\emph{paramodular conjecture} may some day lead to a proof in this last case.
\item The second conjecture on the existence of a functional equation is
of course intimately linked to the first, and the work of Wiles et al.
also proves the existence of this functional equation. But in
addition, in the case of Artin $L$-functions for which only meromorphy
(possibly with infinitely many poles) is known thanks to a theorem of
Brauer, this same theorem implies the functional equation which is thus known
in this case. Also, as mentioned, the Euler factors which we must include
for the ``bad'' primes in order to have a clean functional equation are often
quite difficult to compute.
\item The (global) Riemann hypothesis is not known for \emph{any} global
$L$-function of the type mentioned above, not even for the simplest one, the
Riemann zeta function $\z(s)$. Note that it \emph{is} known for other kinds of
$L$-functions such as \emph{Selberg zeta functions}, but these are
functions of order $2$, so are not in the class considered above.
\item Concerning \emph{special values}: many cases are known, and many
conjectured. This is probably one of the most \emph{fun} conjectures since
everything can be computed explicitly to thousands of decimals if desired.
For instance, for modular forms it is a theorem of Manin, for symmetric
squares of modular forms it is a theorem of Rankin, and for higher symmetric
powers one has very precise conjectures of Deligne, which check perfectly
on a computer, but none of them are proved. For the Riemann zeta function
or Dirichlet $L$-functions, of course all these results such as $\z(2)=\pi^2/6$
date back essentially to Euler.

In the case of an elliptic curve $E$ over $\Q$, the only special point is 
$s=1$, and in this case the whole subject revolves around the \emph{Birch and
Swinnerton-Dyer conjecture} (BSD) which predicts the behavior of $L_1(E;s)$
around $s=1$. The only known results, already quite deep, due to Kolyvagin
and Gross--Zagier, deal with the case where the \emph{rank} of the elliptic
curve is $0$ or $1$.
\end{enumerate}

There exist a number of other very important conjectures linked to the behavior
of $L$-functions at integer points which are not necessarily special,
such as the Bloch, Beilinson, Kato, Lichtenbaum, or Zagier conjectures,
but it would carry us too far afield to describe them in general. However,
in the next subsections, we will give three completely explicit numerical
examples of these conjectures, so that the reader can convince himself both
that they are easy to check numerically, and that the results are spectacular.

\subsection{An Explicit Numerical Example of BSD}\label{sec:BSD}

Let us now be a little more precise. Even if this subsection involves notions
not introduced in these notes, we ask the reader to be patient since the
numerical work only involves standard notions.

Let $E$ be an elliptic curve defined over $\Q$. Elliptic curves have a
natural \emph{abelian group} structure, and it is a theorem of Mordell
that the group of rational points on $E$ is \emph{finitely generated}, i.e.,
$E(\Q)\isom\Z^r\oplus E_{\text{tors}}(\Q)$, where $E_{\text{tors}}(\Q)$ is
a finite group, and $r$ is called the \emph{rank} of the curve.

On the analytic side, we have mentioned that $E$ has an $L$-function $L(E,s)$
(denoted $L_1$ above), and the deep theorem of Wiles et al. says that it has
an analytic continuation to the whole of $\C$ into an entire function with
a functional equation linking $L(E,s)$ to $L(E,2-s)$. The only special point
in the above sense is $s=1$, and a weak form of the Birch and Swinnerton-Dyer
conjecture states that the order of vanishing $v$ of $L(E,s)$ at $s=1$ should
be equal to $r$.

This has been proved for $r=0$ (by Kolyvagin) and for $r=1$
(by Gross--Zagier--Kolyvagin), and nothing is known for $r\ge2$. However,
this is not quite true: if $r=2$ then we cannot have $v=0$ or $1$ by the
previous results, so $v\ge2$. On the other hand, for any given elliptic curve
it is easy to check numerically that $L''(E,1)\ne0$, so to check that $v=2$.
Similarly, if $r=3$ we again cannot have $v=0$ or $1$. But for any given
elliptic curve one can compute the \emph{sign} of the functional equation
linking $L(E,s)$ to $L(E,2-s)$, and this will show that if $r=3$ all
derivatives $L^{(k)}(E,s)$ for $k$ even will vanish. Thus we cannot have
$v=2$, and once again for any $E$ it is easy to check that $L'''(E,1)\ne0$,
hence to check that $v=3$.

Unfortunately, this argument does not work for $r\ge4$. Assume for instance
$r=4$. The same reasoning will show that $L(E,1)=0$ (by Kolyvagin), that
$L'(E,1)=L'''(E,1)=0$ (because the sign of the functional equation will be
$+$), and that $L''''(E,1)\ne0$ by direct computation. The BSD conjecture
tells us that $L''(E,1)=0$, but this is not known for a single curve.

\medskip

Let us give the simplest numerical example, based on an elliptic curve with
$r=4$. I emphasize that no knowledge of elliptic curves is needed for this.

For every prime $p$, consider the congruence
$$y^2+xy\equiv x^3-x^2-79x+289\pmod{p}\;,$$
and denote by $N(p)$ the number of pairs $(x,y)\in(\Z/p\Z)^2$ satisfying it.
We define an arithmetic function $a(n)$ in the following way:

\smallskip

\begin{enumerate}
\item $a(1)=1$.
\item If $p$ is prime, we set $a(p)=p-N(p)$.
\item For $k\ge2$ and $p$ is prime, we define $a(p^k)$ by induction:
$$a(p^k)=a(p)a(p^{k-1})-\chi(p)p\cdot a(p^{k-2})\;,$$
where $\chi(p)=1$ unless $p=2$ or $p=117223$, in which case $\chi(p)=0$.
\item For arbitrary $n$, we extend by multiplicativity: if $n=\prod_ip_i^{k_i}$
then $a(n)=\prod_ia(p_i^{k_1})$.
\end{enumerate}

\begin{remarks}{\rm \begin{itemize}
\item The number $117223$ is simply a prime factor
of the discriminant of the cubic equation obtained by completing the square
in the equation of the above elliptic curve.
\item Even though the definition of $a(n)$ looks complicated, it is \emph{very}
easy to compute (see below), for instance only a few seconds for a million
terms. In addition $a(n)$ is quite small: for $n=1,2,\dots$ we have
$$a(n)=1,-1,-3,1,-4,3,-5,-1,6,4,-6,-3,-6,5,\ldots$$
\end{itemize}}
\end{remarks}

On the analytic side, define a function $f(x)$ for $x>0$ by
$$f(x)=\int_1^\infty e^{-xt}\log(t)^2\,dt\;.$$
Note that it is very easy to compute this integral to thousands of digits if
desired and also note that $f$ tends to $0$ exponentially fast as $x\to\infty$
(more precisely $f(x)\sim 2e^{-x}/x^3$). 

In this specific situation, the BSD conjecture tells us that $S=0$, where
$$S=\sum_{n\ge1}a(n)f\left(\dfrac{2\pi n}{\sqrt{234446}}\right)\;.$$
It takes only a few seconds to compute \emph{thousands} of digits of $S$,
and we can indeed check that $S$ is extremely close to $0$, but as of now
nobody knows how to prove that $S=0$.

\subsection{An Explicit Numerical Example of Beilinson--Bloch}\label{sec:BB}

This subsection is entirely due to V.~Golyshev (personal communication)
whom I heartily thank.

Let $u>1$ be a real parameter. Consider the elliptic curve $E(u)$ with
affine equation
$$y^2=x(x+1)(x+u^2)\;.$$
As usual one can define its $L$-function $L(E(u),s)$ using a general recipe.
The BSD conjecture deals with the value of $L(E(u),s)$ (and its derivatives)
at $s=1$. The Beilinson--Bloch conjectures deal with values at other
integer values of $s$, in the present case we consider $L(E(u),2)$. Once
again it is very easy to compute thousands of decimals of this quantity if
desired.

On the other hand, for $u>1$ consider the function
$$g(u)=2\pi\int_0^1\dfrac{\asin(t)}{\sqrt{1-t^2/u^2}}\,\dfrac{dt}{t}+\pi^2\acosh(u)=\dfrac{\pi^2}{2}\left(2\log(4u)-\sum_{n\ge1}\dfrac{\binom{2n}{n}^2}{n}(4u)^{-2n}\right)\;.$$

The conjecture says that when $u$ is an integer, $L(E(u),2)/g(u)$ should be a
\emph{rational number}. In fact, if we let $N(u)$ be the \emph{conductor}
of $E(u)$ (notion that I have not defined), then it seems that when
$u\ne4$ and $u\ne8$ we even have $F(u)=N(u)L(E(u),2)/g(u)\in\Z$.

Once again, this is a conjecture which can immediately be tested on
modern computer algebra systems such as {\tt Pari/GP}. For instance, for
$u=2,3,\ldots$ we find \emph{numerically} to thousands of decimal digits
(remember that nothing is proved)
$$F(u)=1,2,4/11,8,32,8,4/3,8,32,64,8,96,256,48,16,16,192,\ldots$$

\begin{exercise} Check numerically that the conjecture seems still to be true
when $4u\in\Z$, i.e., if $u$ is a rational number with denominator $2$ or $4$.
On the other hand, it is definitely wrong for instance if $3u\in\Z$ (and
$u\notin\Z$), i.e., when the denominator is $3$. It is possible that there
is a replacement formula, but Bloch and Golyshev tell me that this is
unlikely.\end{exercise}

\subsection{An Explicit Numerical Example of Mahler Measures}

This example is entirely due to W.~Zudilin (personal communication)
whom I heartily thank. The reader does not need any knowledge of Mahler
measures since we are again going to give the example as an equality
between values of $L$-functions and integrals. Note that this can also be
considered an isolated example of the Bloch--Beilinson conjecture.

Consider the elliptic curve $E$ with equation $y^2=x^3-x^2-4x+4$, of conductor
$24$. Its associated $L$-function $L(E,s)$ can easily be shown to be equal
to the $L$-function associated to the modular form
$$q\prod_{n\ge1}(1-q^{2n})(1-q^{4n})(1-q^{6n})(1-q^{12n})$$
(we do not need this for this example, but this will give us two
ways to create the $L$-function in {\tt Pari/GP}). We have the conjectural
identity due to Zudilin:
$$L(E,3)=\dfrac{\pi^2}{36}\left(\pi G+\int_0^1\asin(x)\asin(1-x)\,\dfrac{dx}{x}\right)\;,$$
where $G=\sum_{n\ge0}(-1)^n/(2n+1)^2=0.91596559\cdots$ is Catalan's constant.
  
At the end of this course, the reader will find three complete {\tt Pari/GP}
scripts which implement the BSD, Beilinson--Bloch, and Mahler measure examples
that we have just given.

\subsection{Computational Goals}

Now that we have a handle on what $L$-functions are, we come to the
computational and algorithmic problems, which are the main focus of these
notes. This involves many different aspects, all interesting in their own 
right.

In a first type of situation, we assume that we are ``given'' 
the $L$-function, in other words that we are given a reasonably ``efficient''
algorithm to compute the coefficients $a(n)$ of the Dirichlet series
(or the Euler factors), and that we know the gamma factor $\ga(s)$.
The main computational goals are then the following:

\begin{enumerate}\item Compute $L(s)$ for ``reasonable'' values of $s$: 
  for example, compute $\z(3)$. More sophisticated, but much more interesting:
  check the Birch--Swinnerton-Dyer conjecture, the Beilinson--Bloch
  conjecture, and the conjectures of Deligne concerning special values of
  symmetric powers $L$-functions of modular forms.
\item Check the numerical validity of the functional equation, and
in passing, if unknown, compute the numerical value of the \emph{root 
number} $\om$ occurring in the functional equation.
\item Compute $L(s)$ for $s=1/2+it$ for rather large real values of $t$
(in the case of weight $0$, more generally for $s=(w+1)/2+it$),
and/or make a plot of the corresponding $Z$ function (see below).
\item Compute all the zeros of $L(s)$ on the critical line up to a given
height, and check the corresponding Riemann hypothesis.
\item Compute the residue of $L(s)$ at $s=1$ (typically): for instance
if $L$ is the Dedekind zeta function of a number field, this gives the
product $hR$.
\item Compute the \emph{order} of the zeros of $L(s)$ at integer points
(if it has one), and the leading term in the Taylor expansion: for instance 
for the $L$-function of an elliptic curve and $s=1$, this gives 
the \emph{analytic rank} of an elliptic curve, together with the
Birch and Swinnerton-Dyer data.
\end{enumerate}

\medskip

Unfortunately, we are not always given an $L$-function completely
explicitly. We can lack more or less partial information on the 
$L$-function:

\begin{enumerate}\item One of the most frequent situations is that
one knows the Euler factors for the ``good'' primes, as well as the
corresponding part of the conductor, and that one is lacking both
the Euler factors for the bad primes and the bad part of the conductor.
The goal is then to find numerically the missing factors and missing parts.
\item A more difficult but much more interesting problem is when
essentially nothing is known on the $L$-function except $\ga(s)$, in
other words the $\G_{\R}$ factors and the constant $N$, essentially equal
to the conductor. It is quite amazing that nonetheless one can quite often
tell whether an $L$-function with the given data can exist, and give
some of the initial Dirichlet coefficients (even when several $L$-functions
may be possible).
\item Even more difficult is when essentially nothing is known except
the degree $d$ and the constant $N$, and one looks for possible $\G_{\R}$
factors: this is the case in the search for Maass forms over $\SL_n(\Z)$,
which has been conducted very successfully for $n=2$, $3$, and $4$.
\end{enumerate}

We will not consider these more difficult problems.

\subsection{Available Software for $L$-Functions}

Many people working on the subject have their own software. I mention the 
available public data.

$\bullet$ M.~Rubinstein's {\tt C++} program {\tt lcalc}, which can compute
values of $L$-functions, make large tables of zeros, and so on.
The program uses {\tt C++} language {\tt double}, so is limited to 15 decimal
digits, but is highly optimized, hence very fast, and used in most
situations. Also optimized for large values of the imaginary part
using Riemann--Siegel. Available in {\tt Sage}.

\smallskip

$\bullet$ T.~Dokchitser's program {\tt computel}, initially written in 
{\tt GP/Pari}, rewritten for {\tt magma}, and also available in {\tt Sage}.
Similar to Rubinstein's, but allows arbitrary precision, hence slower,
and has no built-in zero finder, although this is not too difficult
to write. It is not optimized for large imaginary parts.

\smallskip

$\bullet$ Since June 2015, {\tt Pari/GP} has a complete package for computing
with $L$-functions, written by B.~Allombert, K.~Belabas, P.~Molin, and myself,
based on the ideas of T.~Dokchitser for the computation
of inverse Mellin transforms (see below) but put on a more solid footing,
and on the ideas of P.~Molin for computing the $L$-function values themselves,
which avoid computing generalized incomplete gamma functions (see also below).
Note the related complete {\tt Pari/GP} package for computing with modular
forms, available since July 2018.

\smallskip

$\bullet$ Last but not least, not a program but a huge \emph{database}
of $L$-functions, modular forms, number fields, etc., which is the
result of a collaborative effort of approximately 30 to 40 people headed
by D.~Farmer. This database can of course be queried in many different
ways, it is possible and useful to navigate between related pages, and
it also contains {\tt knowls}, bits of knowledge which give the main
definitions. In addition to the stored data, the site can compute
additional required information on the fly using the software mentioned above,
i.e., {\tt Pari}, {\tt Sage}, {\tt magma}, and {\tt lcalc})
Available at:

\centerline{\tt http://www.lmfdb.org}

\section{Arithmetic Methods: Computing $a(n)$}

We now come to the second part of this course: the computation of
the Dirichlet series coefficients $a(n)$ and/or of the Euler factors,
which is usually the same problem. Of course this depends entirely on how
the $L$-function is \emph{given}: in view of what we have seen, it can be
given for instance (but not only) as the $L$-function attached to a modular 
form, to a variety, or to a hypergeometric motive. Since there are so many
relations between these $L$-functions (we have seen several identities above),
we will not separate the way in which they are given, but treat everything
at once.

\smallskip

In view of the preceding section, an important computational problem is the
computation of $|V(\F_q)|$ for a variety $V$. This may of course be done by a 
na\"\i ve point count: if $V$ is defined by polynomials in $n$ variables, we can
range through the $q^n$ possibilities for the $n$ variables and count the 
number of common zeros. In other words, there always exists a trivial 
algorithm requiring $q^n$ steps. We of course want something better.

\subsection{General Elliptic Curves}

Let us first look at the special case of \emph{elliptic curves}, i.e.,
a projective curve $V$ with affine equation $y^2=x^3+ax+b$ such that
$p\nmid 6(4a^3+27b^2)$, which is almost the general equation for an
\emph{elliptic curve}. For simplicity assume that $q=p$, but it is immediate
to generalize. If you know the definition of the Legendre symbol, you know
that the number of solutions in $\F_p$ to the equation $y^2=n$ is equal to
$1+\lgs{n}{p}$. If you do not, since $\F_p$ is a field, it is clear that this
number is equal to $0$, $1$, or $2$, and so one can \emph{define} $\lgs{n}{p}$
as one less, so $-1$, $0$, or $1$. Thus, since it is immediate to see that
there is a single projective point at infinity, we have
\begin{align*}|V(\F_p)|&=1+\sum_{x\in\F_p}\left(1+\leg{x^3+ax+b}{p}\right)=p+1-a(p)\;,\quad\text{with}\\
a(p)&=-\sum_{0\le x\le p-1}\leg{x^3+ax+b}{p}\;.\end{align*}
Now a Legendre symbol can be computed very efficiently using the
\emph{quadratic reciprocity law}. Thus, considering that it can be computed
in constant time (which is not quite true but almost), this gives a $O(p)$
algorithm for computing $a(p)$, already much faster than the trivial $O(p^2)$
algorithm consisting in looking at all pairs $(x,y)$.

\smallskip

To do better, we have to use an additional and crucial property of an elliptic
curve: it is an \emph{abelian group}. Using this combined with the so-called
Hasse bounds $|a(p)|<2\sqrt{p}$ (a special case of the Weil conjectures), and
the so-called \emph{baby-step giant-step algorithm} due to Shanks, one can
obtain a $O(p^{1/4})$ algorithm, which is very fast for all practical 
purposes.

\smallskip

However a remarkable discovery due to Schoof in the early 1980's is that
there exists a practical algorithm for computing $a(p)$ which is
\emph{polynomial in $\log(p)$}, for instance $O(\log^6(p))$. The idea is to
compute $a(p)$ modulo $\ell$ for small primes $\ell$ using
\emph{$\ell$-division polynomials}, and then use the Chinese remainder theorem
and the bound $|a(p)|<2\sqrt{p}$ to recover $a(p)$. Several 
important improvements have been made on this basic algorithm, in particular
by Atkin and Elkies, and the resulting SEA algorithm (which is implemented
in many computer packages) is able to compute $a(p)$ for $p$ with several
thousand decimal digits. Note however that in practical ranges (say
$p<10^{12}$), the $O(p^{1/4})$ algorithm mentioned above is sufficient.

\subsection{Elliptic Curves with Complex Multiplication}

In certain special cases it is possible to compute $|V(\F_q)|$ for an elliptic
curve $V$ much faster than with any of the above methods: when the elliptic 
curve $V$ has \emph{complex multiplication}. Let us consider the special
cases $y^2=x^3-nx$ (the general case is more complicated but not 
really slower). By the general formula for $a(p)$, we have for 
$p\ge3$:
\begin{align*}a(p)&=-\sum_{-(p-1)/2\le x\le (p-1)/2}\leg{x(x^2-n)}{p}\\
&=-\sum_{1\le x\le (p-1)/2}\left(\leg{x(x^2-n)}{p}+\leg{-x(x^2-n)}{p}\right)\\
&=-\left(1+\leg{-1}{p}\right)\sum_{1\le x\le(p-1)/2}\leg{x(x^2-n)}{p}\end{align*}
by the multiplicative property of the Legendre symbol. This already
shows that if $\lgs{-1}{p}=-1$, in other words $p\equiv3\pmod4$, we
have $a(p)=0$. But we can also find a formula when $p\equiv1\pmod4$:
recall that in that case by a famous theorem due to Fermat, there
exist integers $u$ and $v$ such that $p=u^2+v^2$. If necessary by
exchanging $u$ and $v$, and/or changing the sign of $u$, we may
assume that $u\equiv-1\pmod4$, in which case the decomposition is
unique, up to the sign of $v$. It is then not difficult to
prove the following theorem (see Section 8.5.2 of \cite{Coh3} for the proof):

\begin{theorem} Assume that $p\equiv1\pmod4$ and $p=u^2+v^2$ with
$u\equiv-1\pmod4$. The number of projective points on the elliptic 
curve $y^2=x^3-nx$ (where $p\nmid n$) is equal to $p+1-a(p)$, where
$$a(p)=2\leg{2}{p}\begin{cases}
-u&\text{\quad if\quad $n^{(p-1)/4}\equiv1\pmod{p}$}\\
u&\text{\quad if\quad $n^{(p-1)/4}\equiv-1\pmod{p}$}\\
-v&\text{\quad if\quad $n^{(p-1)/4}\equiv-u/v\pmod{p}$}\\
v&\text{\quad if\quad $n^{(p-1)/4}\equiv u/v\pmod{p}$}\end{cases}$$
(note that one of these four cases must occur).
\end{theorem}

To apply this theorem from a computational standpoint we note the
following two \emph{facts}:

(1) The quantity $n^{(p-1)/4}\bmod p$ can be computed efficiently
by the \emph{binary powering algorithm} (in $O(\log^3(p))$ 
operations). It is however possible to compute it more efficiently
in $O(\log^2(p))$ operations using the \emph{quartic reciprocity law}.

(2) The numbers $u$ and $v$ such that $u^2+v^2=p$ can be computed
efficiently (in $O(\log^2(p))$ operations) using \emph{Cornacchia's
algorithm} which is very easy to describe but not so easy to prove.
It is a variant of Euclid's algorithm. It proceeds as follows:

\smallskip

$\bullet$ As a first step, we compute a square root of $-1$ modulo $p$,
i.e., an $x$ such that $x^2\equiv-1\pmod{p}$. This is done by choosing
randomly a $z\in[1,p-1]$ and computing the Legendre symbol $\lgs{z}{p}$
until it is equal to $-1$ (we can also simply try $z=2$, $3$, ...).
Note that this is a fast computation. When this is the case, we have
by definition $z^{(p-1)/2}\equiv-1\pmod{p}$, hence $x^2\equiv-1\pmod{p}$
for $x=z^{(p-1)/4}\bmod{p}$. Reducing $x$ modulo $p$ and possibly
changing $x$ into $p-x$, we normalize $x$ so that $p/2<x<p$.

\smallskip

$\bullet$ As a second step, we perform the Euclidean algorithm on the pair
$(p,x)$, writing $a_0=p$, $a_1=x$, and $a_{n-1}=q_na_n+a_{n+1}$
with $0\le a_{n+1}<a_n$, and we stop at the exact $n$ for which
$a_n^2<p$. It can be proved (this is the difficult part) that for
this specific $n$ we have $a_n^2+a_{n+1}^2=p$, so up to exchange of
$u$ and $v$ and/or change of signs, we can take $u=a_n$ and $v=a_{n+1}$.

\smallskip

Note that Cornacchia's algorithm can easily be generalized to solving
efficiently $u^2+dv^2=p$ or $u^2+dv^2=4p$ for any $d\ge1$, see Section 1.5.2
of\cite{Coh1} (incidentally one can also solve this for $d<0$, but it poses
completely different problems since there may be infinitely many solutions).

\smallskip

The above theorem is given for the special elliptic curves
$y^2=x^3-nx$ which have complex multiplication by the (ring of integers
of the) field $\Q(i)$, but a similar theorem is valid for all curves
with complex multiplication, see Section 8.5.2 of \cite{Coh3}.

\subsection{Using Modular Forms of Weight $2$}

By Wiles' celebrated theorem, the $L$-function of an elliptic curve is
equal to the $L$-function of a modular form of weight $2$ for $\G_0(N)$,
where $N$ is the conductor of the curve. We do not need to give the
precise definitions of these objects, but only a specific example.

Let $V$ be the elliptic curve with affine equation $y^2+y=x^3-x^2$.
It has conductor $11$. It can be shown using classical modular form methods
(i.e., without Wiles' theorem) that the global $L$-function
$L(V;s)=\sum_{n\ge1}a(n)/n^s$ is the same as that of the modular form
of weight $2$ over $\G_0(11)$ given by
$$f(\tau)=q\prod_{m\ge1}(1-q^m)^2(1-q^{11m})^2\;,$$
with $q=\exp(2\pi i\tau)$. Even with no knowledge of modular forms, this
simply means that if we formally expand the product on the right hand side
as
$$q\prod_{m\ge1}(1-q^m)^2(1-q^{11m})^2=\sum_{n\ge1}b(n)q^n\;,$$
we have $b(n)=a(n)$ for all $n$, and in particular for $n=p$ prime.
We have already seen this example above with a slightly different equation
for the elliptic curve (which makes no difference for its $L$-function outside
of the primes $2$ and $3$).

We see that this gives an alternate method for computing $a(p)$ by
expanding the infinite product. Indeed, the function
$$\eta(\tau)=q^{1/24}\prod_{m\ge1}(1-q^m)$$
is a modular form of weight $1/2$ with known expansion:
$$\eta(\tau)=\sum_{n\ge1}\leg{12}{n}q^{n^2/24}\;,$$
and so using Fast Fourier Transform techniques for formal power series
multiplication we can compute all the coefficients $a(n)$ simultaneously
(as opposed to one by one) for $n\le B$ in time $O(B\log^2(B))$. This
amounts to computing each individual $a(n)$ in time $O(\log^2(n))$, so
it seems to be competitive with the fast methods for elliptic curves
with complex multiplication, but this is an illusion since we must
store all $B$ coefficients, so it can be used only for $B\le 10^{12}$,
say, far smaller than what can be reached using Schoof's algorithm,
which is truly polynomial in $\log(p)$ for each fixed prime $p$.

\subsection{Higher Weight Modular Forms}

It is interesting to note that the dichotomy between elliptic curves
with or without complex multiplication is also valid for modular forms
of higher weight (again, whatever that means, you do not need to know
the definitions). For instance, consider
$$\Delta(\tau)=\Delta_{24}(\tau)=\eta^{24}(\tau)=q\prod_{m\ge1}(1-q^m)^{24}:=\sum_{n\ge1}\tau(n)q^n\;.$$
The function $\tau(n)$ is a famous function called the \emph{Ramanujan
$\tau$ function}, and has many important properties, analogous to those
of the $a(p)$ attached to an elliptic curve (i.e., to a modular
form of weight $2$).

There are several methods to compute $\tau(p)$ for $p$ prime, say. One
is to do as above, using FFT techniques. The running time is similar,
but again we are limited to $B\le 10^{12}$, say. A second more
sophisticated method is to use the \emph{Eichler--Selberg trace formula},
which enables the computation of an individual $\tau(p)$ in time
$O(p^{1/2+\eps})$ for all $\eps>0$. A third very deep method, developed
by Edixhoven, Couveignes, et al., is a generalization of Schoof's algorithm.
While in principle polynomial time in $\log(p)$, it is not yet practical
compared to the preceding method.

For those who want to see the formula using the trace formula explicitly, we
let $H(N)$ be the
\emph{Hurwitz class number} $H(N)$ (essentially the class number of imaginary
quadratic orders counted with suitable multiplicity): if we set
$H_3(N)=H(4N)+2H(N)$ (note that $H(4N)$ can be computed in terms of $H(N)$),
then for $p$ prime
\begin{align*}\tau(p)&=28p^6-28p^5-90p^4-35p^3-1\\
  &\phantom{=}-128\sum_{1\le t<p^{1/2}}t^6(4t^4-9pt^2+7p^2)H_3(p-t^2)\;,
\end{align*}
which is the fastest \emph{practical} formula that I know for computing
$\tau(p)$.

\smallskip

On the contrary, consider
$$\Delta_{26}(\tau)=\eta^{26}(\tau)=q^{13/12}\prod_{m\ge1}(1-q^m)^{26}:=q^{13/12}\sum_{n\ge1}\tau_{26}(n)q^n\;.$$
This is what is called a modular form with complex multiplication. Whatever
the definition, this means that the coefficients $\tau_{26}(p)$ can be
computed in time polynomial in $\log(p)$ using a generalization of
Cornacchia's algorithm, hence very fast.

\begin{exercise} (You need some extra knowledge for this.) In the literature
find an exact formula for $\tau_{26}(p)$ in terms of values of Hecke 
\emph{Gr\"ossencharacters}, and program this formula. Use it to compute
some values of $\tau_{26}(p)$ for $p$ prime as large as you can go.
\end{exercise}

\subsection{Computing $|V(\F_q)|$ for Quasi-diagonal Hypersurfaces}

We now consider a completely different situation where $|V(\F_q)|$ can
be computed without too much difficulty.

As we have seen, in the case of elliptic curves $V$ defined over $\Q$, the
corresponding $L$-function is of \emph{degree $2$}, in other words is of
the form $\prod_p1/(1-a(p)p^{-s}+b(p)p^{-2s})$, where $b(p)\ne0$ for all but
a finite number of $p$. $L$-functions of degree $1$ such as the Riemann
zeta function are essentially $L$-functions of Dirichlet characters, in other
words simple ``twists'' of the Riemann zeta function. $L$-functions of degree
$2$ are believed to be always $L$-functions attached to modular forms,
and $b(p)=\chi(p)p^{k-1}$ for a suitable integer $k$ ($k=2$ for elliptic 
curves), the \emph{weight} (note that this is \emph{one more} than the
so-called \emph{motivic weight}). Even though many unsolved questions remain,
this case is also quite well understood. Much more mysterious are $L$-functions
of higher degree, such as $3$ or $4$, and it is interesting to study natural
mathematical objects leading to such functions. A case where this can be done
reasonably easily is the case of diagonal or \emph{quasi-diagonal hypersurfaces}. We study a special case:

\begin{definition} Let $m\ge2$, for $1\le i\le m$ let $a_i\in\F_q^*$ be
nonzero, and let $b\in\F_q$. The quasi-diagonal hypersurface defined by this
data is the hypersurface in $\P^{m-1}$ defined by the projective equation
$$\sum_{1\le i\le m}a_ix_i^m-b\prod_{1\le i\le m}x_i=0\;.$$
When $b=0$, it is a diagonal hypersurface.
\end{definition}

Of course, we could study more general equations, for instance where the
degree is not equal to the number of variables, but we stick to this
special case.

To compute the number of (projective) points on this hypersurface, we need
an additional definition:

\begin{definition} We let $\om$ be a generator of the group of characters
of $\F_q^*$, either with values in $\C$, or in the $p$-adic field $\C_p$
(do not worry if you are not familiar with this).\end{definition}

Indeed, by a well-known theorem of elementary algebra, the multiplicative
group $\F_q^*$ of a finite field is \emph{cyclic}, so its group of
characters, which is \emph{non-canonically isomorphic} to $\F_q^*$, is also
cyclic, so $\om$ indeed exists.

It is not difficult to prove the following theorem:

\begin{theorem}\label{thmquasi} Assume that $\gcd(m,q-1)=1$ and $b\ne0$, and 
set $B=\prod_{1\le i\le m}(a_i/b)$. If $V$ is the above quasi-diagonal
hypersurface, the number $|V(\F_q)|$ of \emph{affine} points on $V$ is given by
$$|V(\F_q)|=q^{m-1}+(-1)^{m-1}+\sum_{1\le n\le q-2}\om^{-n}(B)J_m(\om^n,\dotsc,\om^n)\;,$$
where $J_m$ is the $m$-variable Jacobi sum.
\end{theorem}

We will study in great detail below the definition and properties of
$J_m$.

Note that the number of \emph{projective} points is simply 
$(|V(\F_q)|-1)/(q-1)$.

There also exists a more general theorem with no restriction on
$\gcd(m,q-1)$, which we do not give.

The occurrence of Jacobi sums is very natural and frequent in point counting
results. It is therefore important to look at efficient ways to compute them,
and this is what we do in the next section, where we also give complete
definitions and basic results.

\section{Gauss and Jacobi Sums}

In this long section, we study in great detail Gauss and Jacobi sums.
Most results are standard, and I would like to emphasize
that almost all of them can be proved with little difficulty
by easy algebraic manipulations.

\subsection{Gauss Sums over $\F_q$}\label{sec:gausssum}

We can define and study Gauss and Jacobi sums in two different contexts: first,
and most importantly, over finite fields $\F_q$, with $q=p^f$ a prime power
(note that from now on we write $q=p^f$ and not $q=p^n$).
Second, over the ring $\Z/N\Z$. The two notions coincide when $N=q=p$ is prime,
but the methods and applications are quite different.

To give the definitions over $\F_q$ we need to recall some fundamental (and
easy) results concerning finite fields.

\begin{proposition} Let $p$ be a prime, $f\ge1$, and $\F_q$ be the finite field
with $q=p^f$ elements, which exists and is unique up to isomorphism.
\begin{enumerate}\item The map $\phi$ such that $\phi(x)=x^p$ is a field 
isomorphism from $\F_q$ to itself leaving $\F_p$ fixed. It is called the
\emph{Frobenius map}.
\item The extension $\F_q/\F_p$ is a \emph{normal} (i.e., separable and Galois)
field extension, with Galois group which is cyclic of order $f$ generated 
by~$\phi$.
\end{enumerate}\end{proposition}

In particular, we can define the \emph{trace} $\Tr_{\F_q/\F_p}$ and the
\emph{norm} $\N_{\F_q/\F_p}$, and we have the formulas (where from now on we
omit $\F_q/\F_p$ for simplicity):
$$\Tr(x)=\sum_{0\le j\le f-1}x^{p^j}\text{\quad and\quad}
\N(x)=\prod_{0\le j\le f-1}x^{p^j}=x^{(p^f-1)/(p-1)}=x^{(q-1)/(p-1)}\;.$$

\begin{definition} Let $\chi$ be a character from $\F_q^*$ to an
algebraically closed field $C$ of characteristic $0$. For $a\in\F_q$
we define the \emph{Gauss sum} $\g(\chi,a)$ by
$$\g(\chi,a)=\sum_{x\in\F_q^*}\chi(x)\z_p^{\Tr(ax)}\;,$$
where $\z_p$ is a fixed primitive $p$th root of unity in $C$.
We also set $\g(\chi)=\g(\chi,1)$.
\end{definition} 

Note that strictly speaking this definition depends on the choice
of $\z_p$. However, if $\z'_p$ is some other primitive $p$th root of
unity we have $\z'_p=\z_p^k$ for some $k\in\F_p^*$, so
$$\sum_{x\in\F_q^*}\chi(x){\z'_p}^{\Tr(ax)}=\g(\chi,ka)\;.$$
In fact it is trivial to see (this follows from the next proposition)
that $\g(\chi,ka)=\chi^{-1}(k)\g(\chi,a)$.

\begin{definition}\label{defeps} We define $\eps$ to be the trivial character,
i.e., such that $\eps(x)=1$ for all $x\in\F_q^*$. We extend characters $\chi$
to the whole of $\F_q$ by setting $\chi(0)=0$ if $\chi\ne\eps$ and $\eps(0)=1$.
\end{definition}

Note that this apparently innocuous definition of $\eps(0)$ is \emph{crucial}
because it simplifies many formulas. Note also that the definition of
$\g(\chi,a)$ is a sum over $x\in\F_q^*$ and not $x\in\F_q$, while for
Jacobi sums we will use all of $\F_q$.

\begin{exercise}\label{exoorth}\begin{enumerate}\item
Show that $\g(\eps,a)=-1$ if $a\in\F_q^*$ and $\g(\eps,0)=q-1$. 
\item If $\chi\ne\eps$, show that $\g(\chi,0)=0$, in other words that
$$\sum_{x\in\F_q}\chi(x)=0$$
(here it does not matter if we sum over $\F_q$ or $\F_q^*$).
\item Deduce that if $\chi_1\ne\chi_2$ then
$$\sum_{x\in\F_q^*}\chi_1(x)\chi_2^{-1}(x)=0\;.$$
This relation is called for evident reasons \emph{orthogonality of
characters}.
\item Dually, show that if $x\ne0,1$ we have $\sum_{\chi}\chi(x)=0$, where
the sum is over all characters of $\F_q^*$.
\end{enumerate}
\end{exercise}

Because of this exercise, if necessary we may assume that $\chi\ne\eps$ 
and/or that $a\ne0$.

\begin{exercise} Let $\chi$ be a character of $\F_q^*$ of exact order $n$.
\begin{enumerate}\item Show that $n\mid(q-1)$ and that
$\chi(-1)=(-1)^{(q-1)/n}$. In particular, if $n$ is odd and $p>2$ we have
$\chi(-1)=1$.
\item Show that $\g(\chi,a)\in\Z[\z_n,\z_p]$, where as usual $\z_m$ denotes
a primitive $m$th root of unity.
\end{enumerate}
\end{exercise}

\begin{proposition}\begin{enumerate}\item If $a\ne0$ we have
$$\g(\chi,a)=\chi^{-1}(a)\g(\chi)\;.$$
\item We have
$$\g(\chi^{-1})=\chi(-1)\ov{\g(\chi)}\;.$$
\item We have
$$\g(\chi^p,a)=\chi^{1-p}(a)\g(\chi,a)\;.$$
\item If $\chi\ne\eps$ we have
$$|\g(\chi)|=q^{1/2}\;.$$
\end{enumerate}\end{proposition}

\subsection{Jacobi Sums over $\F_q$}

Recall that we have extended characters of $\F_q^*$ by setting $\chi(0)=0$
if $\chi\ne\eps$ and $\eps(0)=1$.

\begin{definition} For $1\le j\le k$ let $\chi_j$ be characters of $\F_q^*$.
We define the Jacobi sum
$$J_k(\chi_1,\dotsc,\chi_k;a)=\sum_{x_1+\cdots+x_k=a}\chi_1(x_1)\cdots\chi_k(x_k)$$
and $J_k(\chi_1,\dotsc,\chi_k)=J_k(\chi_1,\dotsc,\chi_k;1)$.
\end{definition}

Note that, as mentioned above, we do not exclude the cases where some
$x_i=0$, using the convention of Definition \ref{defeps} for $\chi(0)$.

The following easy lemma shows that it is only necessary to study 
$J_k(\chi_1,\dotsc,\chi_k)$:

\begin{lemma}\label{lemjactriv} Set $\chi=\chi_1\cdots\chi_k$.
\begin{enumerate}\item If $a\ne0$ we have
$$J_k(\chi_1,\dotsc,\chi_k;a)=\chi(a)J_k(\chi_1,\dotsc,\chi_k)\;.$$
\item If $a=0$, abbreviating $J_k(\chi_1,\dotsc,\chi_k;0)$ to $J_k(0)$ we have
$$J_k(0)=\begin{cases} q^{k-1}&\text{\quad if $\chi_j=\eps$ for all $j$\;,}\\
0&\text{\quad if $\chi\ne\eps$\;,}\\
\chi_k(-1)(q-1)J_{k-1}(\chi_1,\dotsc,\chi_{k-1})&\text{\quad if $\chi=\eps$ and $\chi_k\ne\eps$\;.}\end{cases}$$
\end{enumerate}
\end{lemma}

As we have seen, a Gauss sum $\g(\chi)$ belongs to the rather large ring
$\Z[\z_{q-1},\z_p]$ (and in general not to a smaller ring). The advantage of
Jacobi sums is that they belong to the smaller ring $\Z[\z_{q-1}]$, and as
we are going to see, that they are closely related to Gauss sums. Thus, when
working \emph{algebraically}, it is almost always better to use Jacobi sums 
instead of Gauss sums. On the other hand, when working \emph{analytically}
(for instance in $\C$ or $\C_p$), it may be better to work with Gauss sums:
we will see below the use of root numbers (suggested by Louboutin), and of the
Gross--Koblitz formula.

Note that $J_1(\chi_1)=1$. Outside of this trivial case, the close link between
Gauss and Jacobi sums is given by the following easy proposition, whose
apparently technical statement is only due to the trivial character $\eps$:
if none of the $\chi_j$ nor their product is trivial, we have the simple formula
given by (3).

\begin{proposition}\label{jacgaufq} Denote by $t$ the number of $\chi_j$ equal
to the trivial character $\eps$, and as above set $\chi=\chi_1\dotsc\chi_k$.
\begin{enumerate}\item If $t=k$ then $J_k(\chi_1,\dots,\chi_k)=q^{k-1}$.
\item If $1\le t\le k-1$ then $J_k(\chi_1,\dots,\chi_k)=0$.
\item If $t=0$ and $\chi\ne\eps$ then
$$J_k(\chi_1,\dotsc,\chi_k)=\dfrac{\g(\chi_1)\cdots\g(\chi_k)}{\g(\chi_1\cdots\chi_k)}=\dfrac{\g(\chi_1)\cdots\g(\chi_k)}{\g(\chi)}\;.$$
\item If $t=0$ and $\chi=\eps$ then
\begin{align*}J_k(\chi_1,\dotsc,\chi_k)&=-\dfrac{\g(\chi_1)\cdots\g(\chi_k)}{q}\\
&=-\chi_k(-1)\dfrac{\g(\chi_1)\cdots\g(\chi_{k-1})}{\g(\chi_1\cdots\chi_{k-1})}=-\chi_k(-1)J_{k-1}(\chi_1,\dotsc,\chi_{k-1})\;.\end{align*}
In particular, in this case we have
$$\g(\chi_1)\cdots\g(\chi_k)=\chi_k(-1)qJ_{k-1}(\chi_1,\dotsc,\chi_{k-1})\;.$$
\end{enumerate}
\end{proposition}

\begin{corollary}\label{corjacrecur} With the same notation, assume that 
$k\ge2$ and all the $\chi_j$ are nontrivial. Setting 
$\psi=\chi_1\cdots\chi_{k-1}$, we have the following recursive formula:
$$J_k(\chi_1,\dotsc,\chi_k)=\begin{cases}J_{k-1}(\chi_1,\dotsc,\chi_{k-1})J_2(\psi,\chi_k)&\text{\quad if $\psi\ne\eps$\;,}\\
\chi_{k-1}(-1)qJ_{k-2}(\chi_1,\dotsc,\chi_{k-2})&\text{\quad if $\psi=\eps$\;.}\end{cases}$$
\end{corollary}

The point of this recursion is that the definition of a $k$-fold Jacobi sum 
$J_k$ involves a sum over $q^{k-1}$ values for $x_1,\dotsc,x_{k-1}$, the last
variable $x_k$ being
determined by $x_k=1-x_1-\cdots-x_{k-1}$, so neglecting the time to compute
the $\chi_j(x_j)$ and their product (which is a reasonable assumption), using
the definition takes time $O(q^{k-1})$. On the other hand, using the above
recursion boils down at worst to computing $k-1$ Jacobi sums $J_2$, for a
total time of $O((k-1)q)$. Nonetheless, we will see that in some cases it is
still better to use directly Gauss sums and formula (3) of the proposition.

Since Jacobi sums $J_2$ are the simplest and the above recursion in fact shows
that one can reduce to $J_2$, we will drop the subscript $2$ and simply write
$J(\chi_1,\chi_2)$. Note that
$$J(\chi_1,\chi_2)=\sum_{x\in\F_q}\chi_1(x)\chi_2(1-x)\;,$$
where the sum is over the whole of $\F_q$ and \emph{not} $\F_q\setminus\{0,1\}$
(which makes a difference only if one of the $\chi_i$ is trivial). More
precisely it is clear that $J(\eps,\eps)=q^2$, and that if $\chi\ne\eps$
we have $J(\chi,\eps)=\sum_{x\in\F_q}\chi(x)=0$, which are special cases of
Proposition \ref{jacgaufq}.

\begin{exercise} Let $n\mid(q-1)$ be the order of $\chi$. Prove that
$\g(\chi)^n\in\Z[\z_n]$.
\end{exercise}

\begin{exercise} Assume that none of the $\chi_j$ is equal to $\eps$, but that
their product $\chi$ is equal to $\eps$. Prove that (using the same notation
as in Lemma \ref{lemjactriv}):
$$J_k(0)=\left(1-\dfrac{1}{q}\right)\g(\chi_1)\cdots\g(\chi_k)\;.$$
\end{exercise}

\begin{exercise} Prove the following reciprocity formula for Jacobi sums:
if the $\chi_j$ are all nontrivial and $\chi=\chi_1\cdots\chi_k$, we have
$$J_k(\chi_1^{-1},\dotsc,\chi_k^{-1})=\dfrac{q^{k-1-\delta}}{J_k(\chi_1,\dotsc,\chi_k)}\;,$$
where $\delta=1$ if $\chi=\eps$, and otherwise $\delta=0$.
\end{exercise}

\subsection{Applications of $J(\chi,\chi)$}

In this short subsection we give without proof a couple of 
applications of the special Jacobi sums $J(\chi,\chi)$. Once again
the proofs are not difficult. We begin by the following result,
which is a special case of the Hasse--Davenport relations that we
will give below.

\begin{lemma} Assume that $q$ is odd, and let $\rho$ be the unique
character of order $2$ on $\F_q^*$. For any nontrivial character 
$\chi$ we have
$$\chi(4)J(\chi,\chi)=J(\chi,\rho)\;.$$
Equivalently, if $\chi\ne\rho$ we have
$$\g(\chi)\g(\chi\rho)=\chi^{-1}(4)\g(\rho)\g(\chi^2)\;.$$
\end{lemma}

\begin{exercise}\begin{enumerate}
\item Prove this lemma.
\item Show that $\g(\rho)^2=(-1)^{(q-1)/2}q$.
\end{enumerate}
\end{exercise}

\begin{proposition}\label{propjac34}\begin{enumerate}
\item Assume that $q\equiv1\pmod4$, let $\chi$ be one of the two
characters of order $4$ on $\F_q^*$, and write $J(\chi,\chi)=a+bi$.
Then $q=a^2+b^2$, $2\mid b$, and $a\equiv-1\pmod4$.
\item Assume that $q\equiv1\pmod3$, let $\chi$ be one of the two
characters of order $3$ on $\F_q^*$, and write $J(\chi,\chi)=a+b\rho$,
where $\rho=\z_3$ is a primitive cube root of unity.
Then $q=a^2-ab+b^2$, $3\mid b$, $a\equiv-1\pmod3$, and
$a+b\equiv q-2\pmod{9}$.
\item Let $p\equiv2\pmod3$, $q=p^{2m}\equiv1\pmod3$, and let $\chi$
be one of the two characters of order $3$ on $\F_q^*$. We have
$$J(\chi,\chi)=(-1)^{m-1}p^m=(-1)^{m-1}q^{1/2}\;.$$
\end{enumerate}\end{proposition}

\begin{corollary}\begin{enumerate}
\item (Fermat.) Any prime $p\equiv1\pmod4$ is a sum of two squares.
\item Any prime $p\equiv1\pmod3$ is of the form $a^2-ab+b^2$ with
$3\mid b$, or equivalently $4p=(2a-b)^2+27(b/3)^2$ is of the form
$c^2+27d^2$.
\item (Gauss.) $p\equiv1\pmod3$ is itself of the form $p=u^2+27v^2$
if and only if $2$ is a cube in $\F_p^*$.\end{enumerate}\end{corollary}

\begin{exercise} Assuming the proposition, prove the corollary.
\end{exercise}

\subsection{The Hasse--Davenport Relations}

All the results that we have given up to now on Gauss and Jacobi sums
have rather simple proofs, which is one of the reasons we have not
given them. Perhaps surprisingly, there exist other important
relations which are considerably more difficult to prove. Before
giving them, it is instructive to explain how one can ``guess''
their existence, if one knows the classical theory of the gamma
function $\G(s)$ (of course skip this part if you do not know it,
since it would only confuse you, or read the appendix).

Recall that $\G(s)$ is defined (at least for $\Re(s)>0$) by
$$\G(s)=\int_0^\infty e^{-t}t^s dt/t\;,$$ and the beta function
$B(a,b)$ by $B(a,b)=\int_0^1 t^{a-1}(1-t)^{b-1}\,dt$.
The function $e^{-t}$ transforms sums into products, so is an
\emph{additive} character, analogous to $\z_p^t$. The function
$t^s$ transforms products into products, so is a multiplicative
character, analogous to $\chi(t)$ ($dt/t$ is simply the Haar
invariant measure on $\R_{>0}$). Thus $\G(s)$ is a continuous
analogue of the Gauss sum $\g(\chi)$.

Similarly, since $J(\chi_1,\chi_2)=\sum_t\chi_1(t)\chi_2(1-t)$, we
see the similarity with the function $B$. Thus, it does not come
too much as a surprise that analogous formulas are valid on both
sides. To begin with, it is not difficult to show that
$B(a,b)=\G(a)\G(b)/\G(a+b)$, exactly analogous to
$J(\chi_1,\chi_2)=\g(\chi_1)\g(\chi_2)/\g(\chi_1\chi_2)$.
The analogue of $\G(s)\G(-s)=-\pi/(s\sin(s\pi))$ is
$$\g(\chi)\g(\chi^{-1})=\chi(-1)q\;.$$
But it is well-known that the gamma function has a duplication formula 
$\G(s)\G(s+1/2)=2^{1-2s}\G(1/2)\G(2s)$, and more generally
a multiplication (or distribution) formula. This duplication
formula is clearly the analogue of the formula
$$\g(\chi)\g(\chi\rho)=\chi^{-1}(4)\g(\rho)\g(\chi^2)$$
given above. The \emph{Hasse--Davenport product relation} is
the analogue of the distribution formula for the gamma function.

\begin{theorem} Let $\rho$ be a character of exact order $m$ dividing
$q-1$. For any character $\chi$ of $\F_q^*$ we have
$$\prod_{0\le a<m}\g(\chi\rho^a)=\chi^{-m}(m)k(p,f,m)q^{(m-1)/2}\g(\chi^m)\;,$$
where $k(p,f,m)$ is the fourth root of unity given by
$$k(p,f,m)=\begin{cases}
\leg{p}{m}^f&\text{ if $m$ is odd,}\\
(-1)^{f+1}\leg{(-1)^{m/2+1}m/2}{p}^f\leg{-1}{p}^{f/2}&\text{ if $m$ is even,}\\
\end{cases}$$
where $(-1)^{f/2}$ is to be understood as $i^f$ when $f$ is odd.
\end{theorem}

\begin{remark} For some reason, in the literature this formula is usually 
stated in the weaker form where the constant $k(p,f,m)$ is not
given explicitly.\end{remark}

Contrary to the proof of the distribution formula for the gamma
function, the proof of this theorem is quite long. There are 
essentially two completely different proofs: one using classical
algebraic number theory, and one using $p$-adic analysis. The latter
is simpler and gives directly the value of $k(p,f,m)$. See
Section 3.7.2 of \cite{Coh3} and Section 11.7.4 of \cite{Coh4} for both
detailed proofs.

\smallskip

Gauss sums satisfy another type of nontrivial relation, also due to
Hasse--Davenport, the so-called \emph{lifting relation}, as follows:

\begin{theorem} Let $\F_{q^n}/\F_q$ be an extension of finite fields,
let $\chi$ be a character of $\F_q^*$, and define the \emph{lift}
of $\chi$ to $\F_{q^n}$ by the formula
$\chi^{(n)}=\chi\circ\N_{\F_{q^n}/\F_q}$. We have
$$\g(\chi^{(n)})=(-1)^{n-1}\g(\chi)^n\;.$$
\end{theorem}

This relation is essential in the initial proof of the Weil conjectures
for diagonal hypersurfaces done by Weil himself. This is not surprising,
since we have seen in Theorem \ref{thmquasi} that $|V(\F_q)|$ is closely
related to Jacobi sums, hence also to Gauss sums.

\section{Practical Computations of Gauss and Jacobi Sums}

As above, let $\om$ be a character of order exactly $q-1$, so that
$\om$ is a generator of the group of characters of $\F_q^*$.
For notational simplicity, we will write $J(r_1,\dotsc,r_k)$ instead of
$J(\om^{r_1},\dotsc,\om^{r_k})$. Let us consider the specific example of 
efficient computation of the quantity
$$S(q;z)=\sum_{0\le n\le q-2}\om^{-n}(z)J_5(n,n,n,n,n)\;,$$
which occurs in the computation of the Hasse--Weil zeta function of
a quasi-diagonal threefold, see Theorem \ref{thmquasi}.

\subsection{Elementary Methods}

By the recursion of Corollary \ref{corjacrecur}, we have \emph{generically}
(i.e., except for special values of $n$ which will be considered separately):
$$J_5(n,n,n,n,n)=J(n,n)J(2n,n)J(3n,n)J(4n,n)\;.$$
Since $J(n,an)=\sum_{x}\om^n(x)\om^{an}(1-x)$, the cost of
computing $J_5$ as written is $\Os(q)$, where here and after we
write $\Os(q^\al)$ to mean $O(q^{\al+\eps})$ for all $\eps>0$
(soft-$O$ notation). Thus computing $S(q;z)$ by this direct
method requires time $\Os(q^2)$.

We can however do much better. Since the values of the characters are all in
$\Z[\z_{q-1}]$, we work in this ring. In fact, even better, we work in the 
ring with zero divisors $R=\Z[X]/(X^{q-1}-1)$, together with the natural
surjective map sending the class of $X$ in $R$ to $\z_{q-1}$. Indeed, let $g$ 
be the generator of $\F_q^*$ such that $\om(g)=\z_{q-1}$. We have,
again \emph{generically}:
$$J(n,an)=\sum_{1\le u\le q-2}\om^n(g^u)\om^{an}(1-g^u)
=\sum_{1\le u\le q-2}\z_{q-1}^{nu+an\log_g(1-g^u)}\;,$$
where $\log_g$ is the \emph{discrete logarithm} to base $g$ defined modulo
$q-1$, i.e., such that $g^{\log_g(x)}=x$. If $(q-1)\nmid n$ but $(q-1)\mid an$
we have $\om^{an}=\eps$ so we must add the contribution of $u=0$, which is $1$,
and if $(q-1)\mid n$ we must add the contribution of $u=0$ \emph{and} of
$x=0$, which is $2$ (recall the \emph{essential} convention that
$\chi(0)=0$ if $\chi\ne\eps$ and $\eps(0)=1$, see Definition \ref{defeps}). 

In other words, if we set
$$P_a(X)=\sum_{1\le u\le q-2}X^{(u+a\log_g(1-g^u))\bmod{(q-1)}}\in R\;,$$
we have 
$$J(n,an)=P_a(\z_{q-1}^n)+\begin{cases}
0&\text{\quad if $(q-1)\nmid an$\;,}\\
1&\text{\quad if $(q-1)\mid an$ but $(q-1)\nmid n$\;, and}\\
2&\text{\quad if $(q-1)\mid n$\;.}\end{cases}$$
Thus, if we set finally
$$P(X)=P_1(X)P_2(X)P_3(X)P_4(X)\bmod{X^{q-1}}\in R\;,$$
we have (still generically) $J_5(n,n,n,n,n)=P(\z_{q-1}^n)$.
Assume for the moment that this is true for all $n$ (we will correct this
below), let $\ell=\log_g(z)$, so that $\om(z)=\om(g^\ell)=\z_{q-1}^\ell$,
and write
$$P(X)=\sum_{0\le j\le q-2}a_jX^j\;.$$
We thus have
$$\om^{-n}(z)J_5(n,n,n,n,n)=\z_{q-1}^{-n\ell}\sum_{0\le j\le q-2}a_j\z_{q-1}^{nj}
=\sum_{0\le j\le q-2}a_j\z_{q-1}^{n(j-\ell)}\;,$$
hence
\begin{align*}S(q;z)&=\sum_{0\le n\le q-2}\om^{-n}(z)J_5(n,n,n,n,n)
=\sum_{0\le j\le q-2}a_j\sum_{0\le n\le q-2}\z_{q-1}^{n(j-\ell)}\\
&=(q-1)\sum_{0\le j\le q-2,\ j\equiv\ell\pmod{q-1}}a_j=(q-1)a_{\ell}\;.\end{align*}
The result is thus immediate as soon as we know the coefficients of the
polynomial $P$. Since there exist fast methods for computing discrete
logarithms, this leads to a $\Os(q)$ method for computing $S(q;z)$.

\smallskip

To obtain the correct formula, we need to adjust for the special $n$
for which $J_5(n,n,n,n,n)$ is not equal to $J(n,n)J(n,2n)J(n,3n)J(n,4n)$,
which are the same for which $(q-1)\mid an$ for some $a$ such that 
$2\le a\le 4$, together with $a=5$. This is easy but boring, and should be 
skipped on first reading.

\begin{enumerate}\item For $n=0$ we have $J_5(n,n,n,n,n)=q^4$, and on the
other hand $P(1)=(J(0,0)-2)^4=(q-2)^4$, so the correction term is
$q^4-(q-2)^4=8(q-1)(q^2-2q+2)$.
\item For $n=(q-1)/2$ (if $q$ is odd) we have
$$J_5(n,n,n,n,n)=\g(\om^n)^5/\g(\om^{5n})=\g(\om^n)^4=\g(\rho)^4$$
since $5n\equiv n\pmod{q-1}$, where $\rho$ is the character of order $2$,
and we have $\g(\rho)^2=(-1)^{(q-1)/2}q$, so $J_5(n,n,n,n,n)=q^2$.
On the other hand
\begin{align*}P(\z_{q-1}^n)&=J(\rho,\rho)(J(\rho,2\rho)-1)J(\rho,\rho)(J(\rho,2\rho)-1)\\
&=J(\rho,\rho)^2=\g(\rho)^4/q^2=1\;,\end{align*}
so the correction term is $\rho(z)(q^2-1)$.
\item For $n=\pm(q-1)/3$ (if $q\equiv1\pmod3$), writing $\chi_3=\om^{(q-1)/3}$,
which is one of the two cubic characters, we have
\begin{align*}J_5(n,n,n,n,n)&=\g(\om^n)^5/\g(\om^{5n})=\g(\om^n)^5/\g(\om^{-n})\\
&=\g(\om^n)^6/(\g(\om^{-n})\g(\om^n))=\g(\om^n)^6/q\\
&=qJ(n,n)^2\end{align*}
(check all this). On the other hand
\begin{align*}P(\z_{q-1}^n)&=J(n,n)J(n,2n)(J(n,3n)-1)J(n,4n)\\
&=\dfrac{\g(\om^n)^2}{\g(\om^{2n})}\dfrac{\g(\om^n)\g(\om^{2n})}{q}\dfrac{\g(\om^n)^2}{\g(\om^{2n})}\\
&=\dfrac{\g(\om^n)^5}{q\g(\om^{-n})}=\dfrac{\g(\om^n)^6}{q^2}=J(n,n)^2\;,\end{align*}
so the correction term is
$2(q-1)\Re(\chi_3^{-1}(z)J(\chi_3,\chi_3)^2)$.
\item For $n=\pm(q-1)/4$ (if $q\equiv1\pmod4$), writing $\chi_4=\om^{(q-1)/4}$,
which is one of the two quartic characters, we have
$$J_5(n,n,n,n,n)=\g(\om^n)^5/\g(\om^{5n})=\g(\om^n)^4
=\om^n(-1)qJ_3(n,n,n)\;.$$
In addition, we have
$$J_3(n,n,n)=J(n,n)J(n,2n)=\om^n(4)J(n,n)^2=\rho(2)J(n,n)^2\;,$$
so
$$J_5(n,n,n,n,n)=\g(\om^n)^4=\om^n(-1)q\rho(2)J(n,n)^2\;.$$
Note that 
$$\chi_4(-1)=\chi_4^{-1}(-1)=\rho(2)=(-1)^{(q-1)/4}\;,$$ 
(Exercise: prove it!), so that $\om^n(-1)\rho(2)=1$ and the above
simplifies to $J_5(n,n,n,n,n)=qJ(n,n)^2$.

On the other hand,
\begin{align*}P(\z_{q-1}^n)&=J(n,n)J(n,2n)J(n,3n)(J(n,4n)-1)\\
&=\dfrac{\g(\om^n)^2}{\g(\om^{2n})}\dfrac{\g(\om^n)\g(\om^{2n})}{\g(\om^{3n})}
\dfrac{\g(\om^n)\g(\om^{3n})}{q}\\
&=\dfrac{\g(\om^n)^4}{q}=\om^n(-1)\rho(2)J(n,n)^2=J(n,n)^2\end{align*}
as above, so the correction term is
$2(q-1)\Re(\chi_4^{-1}(z)J(\chi_4,\chi_4)^2)$.
\item For $n=a(q-1)/5$ with $1\le a\le 4$ (if $q\equiv1\pmod5$), writing
$\chi_5=\om^{(q-1)/5}$ we have
$J_5(n,n,n,n,n)=-\g(\chi_5^a)^5/q$, while abbreviating
$\g(\chi_5^{am})$ to $g(m)$ we have
\begin{align*}P(\z_{q-1}^n)&=J(n,n)J(n,2n)J(n,3n)J(n,4n)\\
&=-\dfrac{g(n)^2}{g(2n)}\dfrac{g(n)g(2n)}{g(3n)}\dfrac{g(n)g(3n)}{g(4n)}\dfrac{g(n)g(4n)}{q}\\
&=-\dfrac{g(n)^5}{q}\;,\end{align*}
so there is no correction term.\end{enumerate}

Summarizing, we have shown the following:

\begin{proposition} Let $S(q;z)=\sum_{0\le n\le q-2}\om^{-n}(z)J_5(n,n,n,n,n)$.
Let $\ell=\log_g(z)$ and let $P(X)=\sum_{0\le j\le q-2}a_jX^j$ be the polynomial
defined above. We have
$$S(q;z)=(q-1)(T_1+T_2+T_3+T_4+a_{\ell})\;,$$
where $T_m=0$ if $m\nmid(q-1)$ and otherwise
\begin{align*}T_1&=8(q^2-2q+2)\;,\quad T_2=\rho(z)(q+1)\;,\\
T_3&=2\Re(\chi_3^{-1}(z)J(\chi_3,\chi_3)^2)\;,\text{\quad and\quad}T_4=2\Re(\chi_4^{-1}(z)J(\chi_4,\chi_4)^2)\;,\end{align*}
with the above notation.
\end{proposition}

Note that thanks to Proposition \ref{propjac34}, these supplementary Jacobi 
sums $J(\chi_3,\chi_3)$ and $J(\chi_4,\chi_4)$ can be computed in logarithmic
time using Cornacchia's algorithm (this is not quite true, one needs an
additional slight computation, do you see why?).

Note also for future reference that the above proposition \emph{proves} that 
$(q-1)\mid S(q,z)$, which is not clear from the definition.

\subsection{Sample Implementations}

For simplicity, assume that $q=p$ is prime. I have written simple 
implementations of the computation of $S(q;z)$. In the first implementation,
I use the na\"\i ve formula expressing $J_5$ in terms of $J(n,an)$ and sum on
$n$, except that I use the reciprocity formula which gives
$J_5(-n,-n,-n,-n,-n)$ in terms of $J_5(n,n,n,n,n)$ to sum only over $(p-1)/2$
terms instead of $p-1$. Of course to avoid recomputation, I precompute
a discrete logarithm table.

The timings for $p\approx 10^k$ for $k=2$, $3$, and $4$ are 
$0.03$, $1.56$, and $149$ seconds respectively, compatible with $\Os(q^2)$
time.

\smallskip

On the other hand, implementing in a straightforward manner the algorithm 
given by the above proposition gives timings for $p\approx 10^k$ for
$k=2$, $3$, $4$, $5$, $6$, and $7$ of $0$, $0.02$, $0.08$, $0.85$, $9.90$,
and $123$ seconds respectively, of course much faster and compatible with 
$\Os(q)$ time. 

The main drawback of this method is that it requires $O(q)$ storage: it is
thus applicable only for $q\le 10^8$, say, which is more than sufficient
for many applications, but of course not for all. For instance, the case
$p\approx 10^7$ mentioned above already required a few gigabytes of storage.

\subsection{Using Theta Functions}

A completely different way of computing Gauss and Jacobi sums has been
suggested by S.~Louboutin. It is related to the theory of $L$-functions of
Dirichlet characters that we study below, and in our context is valid only
for $q=p$ prime, not for prime powers, but in the context of Dirichlet
characters it is valid in general (simply replace $p$ by $N$ and $\F_p$
by $\Z/N\Z$ in the following formulas when $\chi$ is a primitive character
of conductor $N$, see below for definitions):

\begin{definition} Let $\chi$ be a character on $\F_p$, and let $e=0$ or $1$
be such that $\chi(-1)=(-1)^e$. The \emph{theta function} associated to
$\chi$ is the function defined on the upper half-plane by
$$\Th(\chi,\tau)=2\sum_{m\ge1}m^e\chi(m)e^{i\pi m^2\tau/p}\;.$$
\end{definition}

The main property of this function, which is a direct consequence of the
\emph{Poisson summation formula}, and is equivalent to the functional
equation of Dirichlet $L$-functions, is as follows:

\begin{proposition} We have the functional equation
$$\Th(\chi,-1/\tau)=\om(\chi)(\tau/i)^{(2e+1)/2}\Th(\chi^{-1},\tau)\;,$$
with the principal determination of the square root, and where
$\om(\chi)=\g(\chi)/(i^ep^{1/2})$ is the so-called \emph{root number}.
\end{proposition}

\begin{corollary} If $\chi(-1)=1$ we have
$$\g(\chi)=p^{1/2}\dfrac{\sum_{m\ge1}\chi(m)\exp(-\pi m^2/pt)}{t^{1/2}\sum_{m\ge1}\chi^{-1}(m)\exp(-\pi m^2t/p)}$$
and if $\chi(-1)=-1$ we have
$$\g(\chi)=p^{1/2}i\dfrac{\sum_{m\ge1}\chi(m)m\exp(-\pi m^2/pt)}{t^{3/2}\sum_{m\ge1}\chi^{-1}(m)m\exp(-\pi n^2t/p)}$$
for any $t$ such that the denominator does not vanish.
\end{corollary}

Note that the optimal choice of $t$ is $t=1$, and (at least for $p$ prime)
it seems that the denominator never vanishes (there are counterexamples
when $p$ is not prime, but apparently only four, see \cite{Coh-Zag}).

It follows from this corollary that $\g(\chi)$ can be computed numerically as
a complex number in $\Os(p^{1/2})$ operations. Thus,
if $\chi_1$ and $\chi_2$ are nontrivial characters such that 
$\chi_1\chi_2\ne\eps$ (otherwise $J(\chi_1,\chi_2)$ is trivial to compute),
the formula $J(\chi_1,\chi_2)=\g(\chi_1)\g(\chi_2)/\g(\chi_1\chi_2)$ allows
the computation of $J_2$ \emph{numerically} as a complex number in
$\Os(p^{1/2})$ operations. 

To recover $J$ itself as an algebraic number we could either compute all its
conjugates, but this would require more time than the direct computation of
$J$, or possibly use the LLL algorithm, which although fast, would also
require some time. In practice, to perform computations such as that of
the sum $S(q;z)$ above, we only need
$J$ to sufficient accuracy: we perform all the elementary operations in
$\C$, and since we know that at the end the result will be an integer
for which we know an upper bound, we thus obtain a proven exact result.

More generally, we have generically $J_5(n,n,n,n,n)=\g(\om^n)^5/\g(\om^{5n})$,
which can thus be computed in $\Os(p^{1/2})$ operations. It follows that
$S(p;z)$ can be computed in $\Os(p^{3/2})$ operations, which is slower than
the elementary method seen above. The main advantage is that we do not need
much storage: more precisely, we want to compute $S(p;z)$ to sufficiently
small accuracy that we can recognize it as an integer, so a priori up to
an absolute error of $0.5$. However, we have seen that $(p-1)\mid S(p;z)$:
it is thus sufficient to have an absolute error less than $(p-1)/2$
thus at worse each of the $p-1$ terms in the sum to an absolute error less 
than $1/2$. Since generically $|J_5(n,n,n,n,n)|=p^2$, we need a relative
error less than $1/(2p^2)$, so less than $1/(10p^2)$ on each Gauss sum.
In practice of course this is overly pessimistic, but it does not matter.
For $p\le 10^9$, this means that $19$ decimal digits suffice.

The main term in the theta function computation (with $t=1$) is 
$\exp(-\pi m^2/p)$, so we need $\exp(-\pi m^2/p)\le 1/(100p^2)$, say, in other
words $\pi m^2/p\ge 4.7+2\log(p)$, so $m^2\ge p(1.5+0.7\log(p))$.

This means that we will need the values of $\om(m)$ only up to this limit,
of the order of $O((p\log(p))^{1/2})$, considerably smaller than $p$.
Thus, instead of computing a full discrete logarithm table, which takes
some time but more importantly a lot of memory, we compute only discrete
logarithms up to that limit, using specific algorithms for doing so
which exist in the literature, some of which being quite easy.

A straightforward implementation of this method gives timings for
$k=2$, $3$, $4$, and $5$ of $0.02$, $0.40$, $16.2$, and $663$ seconds
respectively, compatible with $\Os(p^{3/2})$ time. This is faster than
the completely na\"\i ve method, but slower than the method explained above.
Its advantage is that it requires much less memory. For $p$ around $10^7$,
however, it is much too slow so this method is rather useless. We will see
that its usefulness is mainly in the context where it was invented, i.e.,
for $L$-functions of Dirichlet characters.

\subsection{Using the Gross--Koblitz Formula}

This section is of a higher mathematical level than the
preceding ones, but is very important since it gives the best method for
computing Gauss (and Jacobi) sums. We refer to Sections 11.6 and 11.7 of
\cite{Coh4} for complete details, and urge the reader to try to understand
what follows.

In the preceding sections, we have considered Gauss sums as belonging to a 
number of different rings: the ring $\Z[\z_{q-1},\z_p]$ or the field $\C$ of 
complex numbers, and for Jacobi sums the ring $\Z[\z_{q-1}]$, but also the
ring $\Z[X]/(X^{q-1}-1)$, and again the field $\C$.

In number theory there exist other algebraically closed fields which are
useful in many contexts, the fields $\C_\ell$ of $\ell$-adic numbers, one
for each prime number $\ell$. These fields come with a topology and analysis
which are rather special: one of
the main things to remember is that a sequence of elements tends to $0$
if and only the $\ell$-adic valuation of the elements (the largest exponent
of $\ell$ dividing them) tends to infinity. For instance $2^m$ tends to $0$
in $\C_2$, but in no other $\C_{\ell}$, and $15^m$ tends to $0$ in
$\C_3$ and in $\C_5$.

The most important subrings of $\C_{\ell}$ are the ring $\Z_{\ell}$
of $\ell$-adic integers, the elements of which can be written as
$x=a_0+a_1\ell+\cdots+a_k\ell^k+\cdots$ with $a_j\in[0,\ell-1]$, and its field
of fractions $\Q_{\ell}$, which contains $\Q$, whose elements can be
represented in a similar way as $x=a_{-m}\ell^{-m}+a_{-(m-1)}\ell^{-(m-1)}+\cdots+a_{-1}\ell^{-1}+a_0+a_1\ell+\cdots.$

In dealing with Gauss and Jacobi sums over $\F_q$ with $q=p^f$,
the only $\C_{\ell}$ which is of use for us is the one with $\ell=p$
(in highbrow language, we are going to use implicitly \emph{crystalline}
$p$-adic methods, while for $\ell\ne p$ it would be \emph{\'etale} $\ell$-adic
methods).

Apart from this relatively strange topology, many definitions and results 
valid on $\C$ have analogues in $\C_p$. The main object that we will
need in our context is the analogue of the gamma function, naturally called
the $p$-adic gamma function, in the present case due to Morita (there is
another one, see Section 11.5 of \cite{Coh4}), and denoted $\G_p$.
Its definition is in fact quite simple:

\begin{definition} For $s\in\Z_p$ we define
$$\G_p(s)=\lim_{m\to s}(-1)^m\prod_{\substack{0\le k<m\\p\nmid k}}k\;,$$
where the limit is taken over any sequence of positive integers $m$
tending to $s$ for the $p$-adic topology.\end{definition}

It is of course necessary to show that this definition makes sense,
but this is not difficult, and most of the important properties
of $\G_p(s)$, analogous to those of $\G(s)$, can be deduced from it.

\begin{exercise} Choose $p=5$ and $s=-1/4$, so that $p$-adically
$s=1/(1-5)=1+5+5^2+5^3+\cdots$. 
\begin{enumerate}\item Compute the right hand side of
the above definition with small $5$-adic accuracy for $m=1$, $1+5$,
and $1+5+5^2$.
\item It is in fact easy to compute that
$$\G_5(-1/4)=4 + 4\cdot5 + 5^3 + 3\cdot5^4 + 2\cdot5^5 + 2\cdot5^6 + 2\cdot5^7 + 4\cdot5^8+\cdots$$
Using this, show that $\G_5(-1/4)^2/16$ seems to be a $5$-adic root of
the polynomial $5X^2+4X+1$. This is in fact true, see the Gross--Koblitz
formula below.\end{enumerate}
\end{exercise}

We need a much deeper property of $\G_p(s)$ known as the
Gross--Koblitz formula: it is in fact an analogue of a formula for
$\G(s)$ known as the Chowla--Selberg formula, and it is also closely
related to the Davenport--Hasse relations that we have seen above.

The proof of the Gross--Koblitz formula was initially given using tools of
crystalline cohomology, but an elementary proof due to A.~Robert now 
exists, see for instance Section 11.7 of \cite{Coh4} once again.

The Gross--Koblitz formula tells us that certain products of $p$-adic gamma
functions at \emph{rational} arguments are in fact \emph{algebraic
numbers}, more precisely \emph{Gauss sums} (explaining their
importance for us). This is quite surprising since usually 
transcendental functions such as $\G_p$ take transcendental values.

To give a specific example, we have $\G_5(1/4)^2=-2+\sqrt{-1}$,
where $\sqrt{-1}$ is the square root in $\Z_5$ congruent to
$3$ modulo $5$. In view of the elementary properties of the
$p$-adic gamma function, this is equivalent to the result stated
in the above exercise as $\G_5(-1/4)^2=-(16/5)(2+\sqrt{-1})$.

\smallskip

Before stating the formula we need to collect a number of facts,
both on classical algebraic number theory and on $p$-adic analysis.
None are difficult to prove, see Chapter 4 of \cite{Coh3}. Recall that 
$q=p^f$.

\medskip

$\bullet$ We let $K=\Q(\z_p)$ and $L=K(\z_{q-1})=\Q(\z_{q-1},\z_p)=\Q(\z_{p(q-1)})$, so that $L/K$ is an extension of degree $\phi(q-1)$.
There exists a unique prime ideal $\p$ of $K$ above $p$, and we have
$\p=(1-\z_p)\Z_K$ and $\p^{p-1}=p\Z_K$, and $\Z_K/\p\isom\F_p$. The prime 
ideal $\p$ splits into a product of $g=\phi(q-1)/f$ prime ideals 
$\GP_j$ of degree $f$ in the extension $L/K$, i.e., $\p\Z_L=\GP_1\cdots\GP_g$,
and for any prime ideal $\GP=\GP_j$ we have $\Z_L/\GP\isom\F_q$. 

\begin{exercise} Prove directly that for any $f$ we have $f\mid\phi(p^f-1)$.
\end{exercise}

$\bullet$ Fix one of the prime ideals $\GP$ as above. There exists a unique
group isomorphism $\om=\om_{\GP}$ from $(\Z_L/\GP)^*$ to the group of 
$(q-1)$st roots of unity in $L$, such that for all $x\in(\Z_L/\GP)^*$ we have
$\om(x)\equiv x\pmod{\GP}$. It is called the \emph{Teichm\"uller character},
and it can be considered as a character of order $q-1$ on 
$\F_q^*\isom(\Z_L/\GP)^*$. We can thus \emph{instantiate} the definition of
a Gauss sum over $\F_q$ by defining it as $\g(\om_{\GP}^{-r})\in L$.

\smallskip

$\bullet$ Let $\z_p$ be a primitive $p$th root of unity in $\C_p$,
fixed once and for all. There exists a unique $\pi\in\Z[\z_p]$
satisfying $\pi^{p-1}=-p$, $\pi\equiv1-\z_p\pmod{\pi^2}$, and
we set $K_{\p}=\Q_p(\pi)=\Q_p(\z_p)$, and $L_{\GP}$ the \emph{completion}
of $L$ at $\GP$. The field extension $L_{\GP}/K_{\p}$ is Galois, with Galois 
group isomorphic to $\Z/f\Z$ (which is the same as the Galois group of 
$\F_q/\F_p$, where $\F_p$ (resp., $\F_q$) is the so-called 
\emph{residue field} of $K$ (resp., $L$)).

\smallskip

$\bullet$ We set the following:

\begin{definition} We define the \emph{$p$-adic Gauss sum} by
$$\g_q(r)=\sum_{x\in L_{\GP},\ x^{q-1}=1}x^{-r}\z_p^{\Tr_{L_{\GP}/K_{\p}}(x)}\in L_{\GP}\;.$$
\end{definition}

Note that this depends on the choice of $\z_p$, or equivalently of $\pi$.
Since $\g_q(r)$ and $\g(\om_{\GP}^{-r})$ are algebraic numbers, it is
clear that they are equal, although viewed in fields having different 
topologies. Thus, results about $\g_q(r)$ translate immediately into results
about $\g(\om_{\GP}^{-r})$, hence about general Gauss sums over finite fields.

\smallskip

The Gross--Koblitz formula is as follows:

\begin{theorem}[Gross--Koblitz] Denote by $s(r)$ the sum of digits in base $p$
of the integer $r\bmod{(q-1)}$, i.e., of the unique integer $r'$ such that
$r'\equiv r\pmod{q-1}$ and $0\le r'<q-1$. We have
$$\g_q(r)=-\pi^{s(r)}\prod_{0\le i<f}\G_p\left(\left\{\dfrac{p^{f-i}r}{q-1}\right\}\right)\;,$$
where $\{x\}$ denotes the fractional part of $x$.\end{theorem}

Let us show how this can be used to compute Gauss or Jacobi sums, and in
particular our sum $S(q;z)$. Assume for simplicity that $f=1$, in other
words that $q=p$: the right hand
side is thus equal to $-\pi^{s(r)}\G_p(\{pr/(p-1)\})$. Since we can always
choose $r$ such that $0\le r<p-1$, we have $s(r)=r$ and
$\{pr/(p-1)\}=\{r+r/(p-1)\}=r/(p-1)$, so the RHS is $-\pi^r\G_p(r/(p-1))$.
Now an easy property of $\G_p$ is that it is differentiable: recall that $p$
is ``small'' in the $p$-adic topology, so $r/(p-1)$ is close to $-r$, more
precisely $r/(p-1)=-r+pr/(p-1)$ (this is how we obtained it in the first 
place!). Thus in particular, if $p>2$ we have the Taylor expansion 

\begin{align*}\G_p(r/(p-1))&=\G_p(-r)+(pr/(p-1))\G'_p(-r)+O(p^2)\\
&=\G_p(-r)-pr\G'_p(-r)+O(p^2)\;.\end{align*}

Since $\g_q(r)$ depends only on $r$ modulo $p-1$, we will assume that
$0\le r<p-1$. In that case it is easy to show from the definition that 
$$\G_p(-r)=1/r!\text{\quad and\quad}\G'_p(-r)=(-\ga_p+H_r)/r!\;,$$
where  $H_r=\sum_{1\le n\le r}1/n$ is the harmonic sum, and $\ga_p=-\G'_p(0)$
is the $p$-adic analogue of Euler's constant. 

\begin{exercise} Prove these formulas, as well as the congruence for
$\ga_p$ given below.
\end{exercise}

There exist infinite ($p$-adic)
series enabling accurate computation of $\ga_p$, but since we only need it 
modulo $p$, we use the easily proved congruence
$\ga_p\equiv((p-1)!+1)/p=W_p\pmod{p}$, the so-called \emph{Wilson quotient}.

\smallskip

We will see below that, as a consequence of the Weil conjectures proved
by Deligne, it is sufficient to compute $S(p;z)$ modulo $p^2$. Thus, in the
following $p$-adic computation we only work modulo $p^2$.

The Gross--Koblitz formula tells us that for $0\le r<p-1$ we have
$$\g_q(r)=-\dfrac{\pi^r}{r!}(1-pr(H_r-W_p)+O(p^2))\;.$$
It follows that for $(p-1)\nmid 5r$ we have
$$J(-r,-r,-r,-r,-r)=\dfrac{\g(\om_{\GP})^5}{\g(\om_{\GP}^5)}=\dfrac{\g_q(r)^5}{\g_q(5r)}=\pi^{f(r)}(a+bp+O(p^2))\;,$$
where $a$ and $b$ will be computed below and
\begin{align*}f(r)&=5r-(5r\bmod{p-1})=5r-(5r-(p-1)\lfloor5r/(p-1)\rfloor)\\
&=(p-1)\lfloor 5r/(p-1)\rfloor\;,\end{align*}
so that $\pi^{f(r)}=(-p)^{\lfloor 5r/(p-1)\rfloor}$ since $\pi^{p-1}=-p$.
Since we want the result modulo $p^2$, we consider three intervals together
with special cases:

\begin{enumerate}\item If $r>2(p-1)/5$ but $(p-1)\nmid 5r$, we have
$$J(-r,-r,-r,-r,-r)\equiv0\pmod{p^2}\;.$$
\item If $(p-1)/5<r<2(p-1)/5$ we have
$$J(-r,-r,-r,-r,-r)\equiv(-p)\dfrac{(5r-(p-1))!}{r!^5}\pmod{p^2}\;.$$
\item If $0<r<(p-1)/5$ we have $f(r)=0$ and $0\le 5r<(p-1)$ hence
\begin{align*}J(-r,-r,-r,-r,-r)&=\dfrac{(5r)!}{r!^5}(1-5pr(H_r-W_p)+O(p^2))\cdot\\
&\phantom{=}\cdot(1+5pr(H_{5r}-W_p)+O(p^2))\\
&\equiv\dfrac{(5r)!}{r!^5}(1+5pr(H_{5r}-H_r))\pmod{p^2}\;.\end{align*}
\item Finally, if $r=j(p-1)/5$ we have $J(-r,-r,-r,-r,-r)=p^4\equiv0\pmod{p^2}$
if $j=0$, and otherwise $J(-r,-r,-r,-r,-r)=-\g_q(r)^5/p$, and since the
$p$-adic valuation of $\g_q(r)$ is equal to $r/(p-1)=j/5$, that of
$J(-r,-r,-r,-r,-r)$ is equal to $j-1$, which is greater or equal to $2$
as soon as $j\ge3$. For $j=2$, i.e., $r=2(p-1)/5$, we thus have
$$J(-r,-r,-r,-r,-r)\equiv p\dfrac{1}{r!^5}\equiv(-p)\dfrac{(5r-(p-1))!}{r!^5}\pmod{p^2}\;,$$
which is the same formula as for $(p-1)/5<r\le 2(p-1)/5$.
For $j=1$, i.e., $r=(p-1)/5$, we thus have
$$J(-r,-r,-r,-r,-r)\equiv-\dfrac{1}{r!^5}(1-5pr(H_r-W_p))\pmod{p^2}\;,$$
while on the other hand 
$$(5r)!=(p-1)!=-1+pW_p\equiv-1-p(p-1)W_p\equiv-1-5prW_p\;,$$ and 
$H_{5r}=H_{p-1}\equiv0\pmod{p}$ (Wolstenholme's congruence, easy), so
\begin{align*}\dfrac{(5r)!}{r!^5}(1+5pr(H_{5r}-H_r))&\equiv-\dfrac{1}{r!^5}(1-5prH_r)(1+5prW_p)\\
&\equiv-\dfrac{1}{r!^5}(1-5pr(H_r-W_p))\pmod{p^2}\;,\end{align*}
which is the same formula as for $0<r<(p-1)/5$.
\end{enumerate}

An important point to note is that we are working $p$-adically, but the
final result $S(p;z)$ being an integer, it does not matter at the end.
There is one small additional detail to take care of: we have
\begin{align*}S(p;z)&=\sum_{0\le r\le p-2}\om^{-r}(z)J(r,r,r,r,r)\\
&=\sum_{0\le r\le p-2}\om^r(z)J(-r,-r,-r,-r,-r)\;,\end{align*}
so we must express $\om^r(z)$ in the $p$-adic setting. Since
$\om=\om_{\GP}$ is the \emph{Teichm\"uller character}, in the $p$-adic
setting it is easy to show that $\om(z)$ is the $p$-adic limit of
$z^{p^k}$ as $k\to\infty$. in particular $\om(z)\equiv z\pmod{p}$, but more
precisely $\om(z)\equiv z^p\pmod{p^2}$. 

\begin{exercise} Let $p\ge3$. Assume that $z\in\Z_p\setminus p\Z_p$ (for 
instance that $z\in\Z\setminus p\Z$). Prove that $z^{p^k}$ has a $p$-adic 
limit $\om(z)$ when $k\to\infty$, that $\om^{p-1}(z)=1$, that 
$\om(z)\equiv z\pmod{p}$, and $\om(z)\equiv z^p\pmod{p^2}$.
\end{exercise}

We have thus proved the following

\begin{proposition} We have
\begin{align*}S(p;z)&\equiv\sum_{0<r\le(p-1)/5}\dfrac{(5r)!}{r!^5}(1+5pr(H_{5r}-H_r))z^{pr}\\
&\phantom{=}-p\sum_{(p-1)/5<r\le2(p-1)/5}\dfrac{(5r-(p-1))!}{r!^5}z^r\pmod{p^2}\;.\end{align*}
In particular
$$S(p;z)\equiv\sum_{0<r\le(p-1)/5}\dfrac{(5r)!}{r!^5}z^r\pmod{p}\;.$$
\end{proposition}

\begin{remarks}{\rm \begin{enumerate}
\item Note that, as must be the case, all mention of $p$-adic numbers has 
disappeared from this formula. We used the $p$-adic setting only in the proof.
It can be proved ``directly'', but with some difficulty.
\item We used the Taylor expansion only to order $2$. It is of course possible
to use it to any order, thus giving a generalization of the above proposition
to any power of $p$.\end{enumerate}}
\end{remarks}

The point of giving all these details is as follows: it is easy to show that
$(p-1)\mid S(p;z)$ (in fact we have seen this in the elementary method above).
We can thus easily compute $S(p;z)$ modulo $p^2(p-1)$. On the other hand,
it is possible to prove (but not easy, it is part of the Weil conjectures
proved by Deligne), that $|S(p;z)-p^4|<4p^{5/2}$. It follows that as soon
as $8p^{5/2}<p^2(p-1)$, in other words $p\ge67$, the computation that we 
perform modulo $p^2$ is sufficient to determine $S(p;z)$ exactly. It is
clear that the time to perform this computation is $\Os(p)$, and in fact
much faster than any that we have seen.

\smallskip

In fact, implementing in a reasonable way the algorithm 
given by the above proposition gives timings for $p\approx 10^k$ for
$k=2$, $3$, $4$, $5$, $6$, $7$, and $8$ of $0$, $0.01$, $0.03$, $0.21$, $2.13$,
$21.92$, and $229.6$ seconds respectively, of course much faster and 
compatible with $\Os(p)$ time. The great additional advantage is that we 
use very small memory. This is therefore the best known method.

\smallskip

{\bf Numerical example:} Choose $p=10^6+3$ and $z=2$. In $2.13$ seconds we find
that $S(p;z)\equiv a\pmod{p^2}$ with $a=356022712041$. Using the Chinese
remainder formula
$$S(p;z)=p^4+((a-(1+a)p^2)\bmod((p-1)p^2))\;,$$
we immediately deduce that
$$S(p;z)=1000012000056356142712140\;.$$

\smallskip

Here is a summary of the timings (in seconds) that we have mentioned:

\bigskip

\centerline{
\begin{tabular}{|c||c|c|c|c|c|c|c|}
\hline
$k$ 	  & $2$ & $3$ & $4$ & $5$ & $6$ & $7$ & $8$ \\
\hline\hline
Na\"\i ve & $0.03$ & $1.56$ & $149$ & $*$ & $*$ & $*$ & $*$\\
\hline
Theta & $0.02$ & $0.40$ & $16.2$ & $663$ & $*$ & $*$ & $*$\\
\hline
Mod $X^{q-1}-1$ & $0$ & $0.02$ & $0.08$ & $0.85$ & $9.90$ & $123$ & $*$\\
\hline
Gross--Koblitz & $0$ & $0.01$ & $0.03$ & $0.21$ & $2.13$ & $21.92$ & $229.6$\\
\hline
\end{tabular}}

\medskip

\centerline{Time for computing $S(p;z)$ for $p\approx10^k$}

\section{Gauss and Jacobi Sums over $\Z/N\Z$}

Another context in which one encounters Gauss sums is over finite rings
such as $\Z/N\Z$. The theory coincides with that over $\F_q$ when
$q=p=N$ is prime, but is rather different otherwise. These other Gauss sums
enter in the important theory of \emph{Dirichlet characters}.

\subsection{Definitions}

We recall the following definition:

\begin{definition} Let $\chi$ be a (multiplicative) character from the
multiplicative group $(\Z/N\Z)^*$ of invertible elements of $\Z/N\Z$ to
the complex numbers $\C$.
We denote by abuse of notation again by $\chi$ the map from $\Z$ to $\C$
defined by $\chi(x)=\chi(x\bmod N)$ when $x$ is coprime to $N$, and
$\chi(x)=0$ if $x$ is not coprime to $N$, and call it the Dirichlet character
modulo $N$ associated to $\chi$.\end{definition}

It is clear that a Dirichlet character satisfies $\chi(xy)=\chi(x)\chi(y)$
for all $x$ and $y$, that $\chi(x+N)=\chi(x)$, and that $\chi(x)=0$
if and only if $x$ is not coprime with $N$. Conversely, it immediate that
these properties characterize Dirichlet characters.

A crucial notion (which has no equivalent in the context of characters of
$\F_q^*$) is that of \emph{primitivity}:

Assume that $M\mid N$. If $\chi$ is a Dirichlet character modulo $M$, we can 
transform it into a character $\chi_N$ modulo $N$ by setting 
$\chi_N(x)=\chi(x)$ if $x$ is coprime to $N$, and $\chi_N(x)=0$ otherwise.
We say that the characters $\chi$ and $\chi_N$ are \emph{equivalent}.
Conversely, if $\psi$ is a character modulo $N$, it is not always true that
one can find $\chi$ modulo $M$ such that $\psi=\chi_N$. If it is possible,
we say that $\psi$ \emph{can be defined modulo $M$}.

\begin{definition} Let $\chi$ be a character modulo $N$. We say that
$\chi$ is a \emph{primitive character} if $\chi$ cannot be defined modulo
$M$ for any proper divisor $M$ of $N$, i.e., for any $M\mid N$ such that 
$M\ne N$.\end{definition}

\begin{exercise} Assume that $N\equiv2\pmod4$. Show that there do not exist
any primitive characters modulo $N$.
\end{exercise}

\begin{exercise} Assume that $p^a\mid N$ with $p$ prime. Show that if $\chi$
is a primitive character modulo $N$, the \emph{order} of $\chi$ (the smallest
$k$ such that $\chi^k$ is a trivial character) is \emph{divisible}
by $p^{a-1}$.
\end{exercise}

As we will see, questions about general Dirichlet characters can always be
reduced to questions about primitive characters, and the latter have much
nicer properties.

\begin{proposition} Let $\chi$ be a character modulo $N$. There exists
a divisor $f$ of $N$ called the \emph{conductor} of $\chi$ (this $f$ has 
nothing to do with the $f$ used above such that $q=p^f$), having the following
properties:
\begin{enumerate}\item The character $\chi$ can be defined modulo $f$,
in other words there exists a character $\psi$ modulo $f$ such that
$\chi=\psi_N$ using the notation above.
\item $f$ is the smallest divisor of $N$ having this property.
\item The character $\psi$ is a primitive character modulo $f$.
\end{enumerate}\end{proposition}

There is also the notion of \emph{trivial character modulo $N$}: however
we must be careful here, and we set the following:

\begin{definition} The trivial character modulo $N$ is the Dirichlet
character associated with the trivial character of $(\Z/N\Z)^*$. It is
usually denoted by $\chi_0$ (but be careful, the index $N$ is implicit, so
$\chi_0$ may represent different characters), and its values are as follows:
$\chi_0(x)=1$ if $x$ is coprime to $N$, and $\chi_0(x)=0$ if $x$ is not
coprime to $N$.\end{definition}

In particular, $\chi_0(0)=0$ if $N\ne1$. The character $\chi_0$ can also be
characterized as the only character modulo $N$ of conductor $1$. 

\begin{definition} Let $\chi$ be a character modulo $N$. The \emph{Gauss sum}
associated to $\chi$ and $a\in\Z$ is
$$\g(\chi,a)=\sum_{x\bmod N}\chi(x)\z_N^{ax}\;,$$
and we write simply $\g(\chi)$ instead of $\g(\chi,1)$.
\end{definition}

The most important results concerning these Gauss sums is the following:

\begin{proposition} Let $\chi$ be a character modulo $N$.\begin{enumerate}
\item If $a$ is coprime to $N$ we have
$$\g(\chi,a)=\chi^{-1}(a)\g(\chi)=\ov{\chi(a)}\g(\chi)\;,$$
and more generally
$\g(\chi,ab)=\chi^{-1}(a)\g(\chi,b)=\ov{\chi(a)}\g(\chi,b)$.
\item If $\chi$ is a \emph{primitive} character, we have
$$\g(\chi,a)=\ov{\chi(a)}\g(\chi)$$ 
for \emph{all} $a$, in other words, in addition to (1), we have
$\g(\chi,a)=0$ if $a$ is not coprime to $N$.
\item If $\chi$ is a \emph{primitive} character, we have
$|\g(\chi)|^2=N$.\end{enumerate}\end{proposition}

Note that (1) is trivial, and that since $\chi(a)$ has modulus $1$ when
$a$ is coprime to $N$, we can write indifferently $\chi^{-1}(a)$ or
$\ov{\chi(a)}$. On the other hand, (2) is not completely trivial.

\smallskip

We leave to the reader the easy task of defining Jacobi sums and of proving
the easy relations between Gauss and Jacobi sums.

\subsection{Reduction to Prime Gauss Sums}

A fundamental and little-known fact is that in the context of Gauss
sums over $\Z/N\Z$ (as opposed to $\F_q$), one can in fact always reduce
to prime $N$. First note (with proof) the following easy result:

\begin{proposition} Let $N=N_1N_2$ with $N_1$ and $N_2$ coprime, and
let $\chi$ be a character modulo $N$.\begin{enumerate}
\item There exist unique characters $\chi_i$ modulo $N_i$ such that
$\chi=\chi_1\chi_2$ in an evident sense, and if $\chi$ is primitive,
the $\chi_i$ will also be primitive.
\item We have the identity (valid even if $\chi$ is not primitive):
$$\g(\chi)=\chi_1(N_2)\chi_2(N_1)\g(\chi_1)\g(\chi_2)\;.$$
\end{enumerate}\end{proposition}

\begin{proof} (1). Since $N_1$ and $N_2$ are coprime there exist $u_1$ and $u_2$ 
such that $u_1N_1+u_2N_2=1$. We define $\chi_1(x)=\chi(xu_2N_2+u_1N_1)$ and 
$\chi_2(x)=\chi(xu_1N_1+u_2N_2)$. We leave to the reader to check (1)
using these definitions.

\smallskip

(2). When $x_i$ ranges modulo $N_i$, $x=x_1u_2N_2+x_2u_1N_1$ ranges
modulo $N$ (check it, in particular that the values are distinct!),
and $\chi(x)=\chi_1(x)\chi_2(x)=\chi_1(x_1)\chi_2(x_2)$. Furthermore,
$$\z_N=\exp(2\pi i/N)=\exp(2\pi i(u_1/N_2+u_2/N_1))=\z_{N_1}^{u_2}\z_{N_2}^{u_1}\;,$$
hence
\begin{align*}\g(\chi)&=\sum_{x\bmod N}\chi(x)\z_N^x\\
&=\sum_{x_1\bmod N_1,\ x_2\bmod N_2}\chi_1(x_1)\chi_2(x_2)\z_{N_1}^{u_2x_1}\z_{N_2}^{u_1x_2}\\
&=\g(\chi_1;u_2)\g(\chi_2;u_1)=\chi_1^{-1}(u_2)\chi_2^{-1}(u_1)\g(\chi_1)\g(\chi_2)\;,\end{align*}
so the result follows since $N_2u_2\equiv1\pmod{N_1}$ and 
$N_1u_1\equiv1\pmod{N_2}$.\fp\end{proof}

Thanks to the above result, the computation of Gauss sums modulo $N$ can be
reduced to the computation of Gauss sums modulo prime powers. 

Here a remarkable simplification occurs, due to Odoni: Gauss sums modulo
$p^a$ for $a\ge2$ can be ``explicitly computed'', in the sense that there
is a direct formula not involving a sum over $p^a$ terms for computing 
them. Although the proof is not difficult, we do not give it, and refer
instead to \cite{Coh5} which can be obtained from the author. We use the 
classical notation $\e(x)$ to mean $e^{2\pi i x}$. Furthermore, we use
the $p$-adic logarithm $\log_p(m)$, but in a totally elementary manner
since we will always have $m\equiv1\pmod p$ and the standard expansion
$-\log_p(1-x)=\sum_{k\ge1}x^k/k$ which we stop as soon as all the terms
are divisible by $p^n$:

\begin{theorem}[Odoni et al.]\label{thmodoni} Let $\chi$ be a \emph{primitive}
character modulo $p^n$.
\begin{enumerate}\item Assume that $p\ge3$ is prime and $n\ge2$. Write
$\chi(1+p)=\e(-b/p^{n-1})$ with $p\nmid b$. Define
$$A(p)=\dfrac{p}{\log_p(1+p)}\text{\quad and\quad}B(p)=A(p)(1-\log_p(A(p)))\;,$$
except when $p^n=3^3$, in which case we define $B(p)=10$. Then
$$\g(\chi)=p^{n/2}\e\left(\dfrac{bB(p)}{p^n}\right)\chi(b)\cdot\begin{cases}
1&\text{\quad if $n\ge2$ is even,}\\
\leg{b}{p}i^{p(p-1)/2}&\text{\quad if $n\ge3$ is odd.}
\end{cases}$$
\item Let $p=2$ and assume that $n\ge4$. Write
$\chi(1+p^2)=\e(b/p^{n-2})$ with $p\nmid b$. Define
$$A(p)=-\dfrac{p^2}{\log_p(1+p^2)}\text{\quad and\quad}B(p)=A(p)(1-\log_p(A(p)))\;,$$
except when $p^n=2^4$, in which case we define $B(p)=13$. Then
$$\g(\chi)=p^{n/2}\e\left(\dfrac{bB(p)}{p^n}\right)\chi(b)\cdot\begin{cases}
\e\left(\dfrac{b}{8}\right)&\text{\quad if $n\ge4$ is even,}\\
\e\left(\dfrac{(b^2-1)/2+b}{8}\right)&\text{\quad if $n\ge5$ is odd.}
\end{cases}$$
\item If $p^n=2^2$, or $p^n=2^3$ and $\chi(-1)=1$, we have $\g(\chi)=p^{n/2}$, 
and if $p^n=2^3$ and $\chi(-1)=-1$ we have $\g(\chi)=p^{n/2}i$.
\end{enumerate}\end{theorem}

Thanks to this theorem, we see that the computation of Gauss sums in the
context of Dirichlet characters can be reduced to the computation of Gauss
sums modulo $p$ for prime $p$. This is of course the same as the
computation of a Gauss sum for a character of $\F_p^*$.

We recall the available methods for computing a single Gauss sum of this
type:

\begin{enumerate}\item The na\"\i ve method, time $\Os(p)$ (applicable in
general, time $\Os(N)$).
\item Using the Gross--Koblitz formula, also time $\Os(p)$, but the implicit
constant is much smaller, and also computations can be done modulo $p$ or
$p^2$ for instance, if desired (applicable only to $N=p$, or in the
context of finite fields).
\item Using theta functions, time $\Os(p^{1/2})$ (applicable in general,
time $\Os(N^{1/2})$).\end{enumerate}

\subsection{General Complete Exponential Sums over $\Z/N\Z$}

We have just seen the (perhaps surprising) fact that Gauss sums modulo
$p^a$ for $a\ge2$ can be ``explicitly computed''. This is in fact
a completely general fact. Let $\chi$ be a Dirichlet character modulo $N$,
and let $F\in \Q[X]$ be integer-valued. Consider the following
\emph{complete exponential sum}:
$$S(F,N)=\sum_{x\bmod N}\chi(x)e^{2\pi i F(x)/N}\;.$$
For this to make sense we must of course assume that $x\equiv y\pmod N$
implies $F(x)\equiv F(y)\pmod{N}$, which is for instance the case if
$F\in\Z[X]$. As we did for Gauss sums, using Chinese remaindering we can
reduce the computation to the case where $N=p^a$ is a prime power. But
the essential point is that if $a\ge2$, $S(F,p^a)$ can be ``explicitly
computed'', see \cite{Coh5} for the detailed statement and proof, so
we are again reduced to the computation of $S(F,p)$.

A simplified version and incomplete version of the result when $\chi$ is the
trivial character is as follows:

\begin{theorem} Let $S=\sum_{x\bmod{p^a}}e^{2\pi iF(x)/p^a}$, and
  assume that $a\ge2$ and $p>2$. Then under suitable assumptions on $F$ we
  have the following:
  \begin{enumerate}
  \item If there does not exist $y$ such that $F'(y)\equiv0\pmod p$ then $S=0$.
  \item Otherwise, there exists $u\in\Z_p$ such that
    $F'(u)=0$ and $v_p(F''(u))=0$, $u$ is unique, and we have
    $$S=p^{a/2}e^{2\pi iF(u)/p^a}g(u,p,a)\;,$$
    where $g(u,p,a)=1$ if $a$ is even and otherwise
    $$g(u,p,a)=\leg{F''(u)}{p}i^{p(p-1)/2}\;.$$
  \end{enumerate}
\end{theorem}

\begin{exercise} Let $F(x)=cx^3+dx$ with $c$ and $d$ integers, and let $p$
  be a prime number such that $p\nmid 6cd$. The assumptions of the theorem
  will then be satisfied. Compute explicitly
  $\sum_{x\bmod{p^a}}e^{2\pi iF(x)/p^a}$ for $a\ge2$. You will need to
  introduce a square root of $-3cd$ modulo $p^a$.\end{exercise}

For instance, using a variant of the above theorem, it is immediate to prove
the following result due to Sali\'e:

\begin{proposition} The \emph{Kloosterman sum} $K(m,n,N)$ is defined by
  $$K(m,n,N)=\sum_{x\in(\Z/N\Z)^*}e^{2\pi i(mx+nx^{-1})/N}\;,$$
  where $x$ runs over the invertible elements of $\Z/N\Z$. If $p>2$
  is a prime such that $p\nmid n$ and $a\ge2$ we have
  $$K(n,n,p^a)=\begin{cases}
  2p^{a/2}\cos(4\pi n/p^a)&\text{ if $2\mid a$,}\\
  2p^{a/2}\leg{n}{p}\cos(4\pi n/p^a)&\text{ if $2\nmid a$ and $p\equiv1\pmod4$,}\\
  -2p^{a/2}\leg{n}{p}\sin(4\pi n/p^a)&\text{ if $2\nmid a$ and $p\equiv3\pmod4$.}\end{cases}$$
\end{proposition}

Note that it is immediate to reduce general $K(m,n,N)$ to the case $m=n$
and $N=p^a$, and to give formulas also for the case $p=2$. As usual the
case $N=p$ is \emph{not} explicit, and, contrary to the case of Gauss sums
where it is easy to show that $|\gg(\chi)|=\sqrt{p}$ for a primitive character
$\chi$, the bound $|K(m,n,p)|\le 2\sqrt{p}$ for $p\nmid nm$ due to Weil is
much more difficult to prove, and in fact follows from his proof of the
Riemann hypothesis for curves.

\section{Numerical Computation of $L$-Functions}

\subsection{Computational Issues}

Let $L(s)$ be a general $L$-function as defined in Section \ref{sec:one},
and let $N$ be its conductor. There are several computational problems that we
want to solve. The first, but not necessarily the most important, is the
numerical computation of $L(s)$ for given complex values of $s$. This problem
is of very varying difficulty depending on the size of $N$ and of the
imaginary part of $s$ (note that if the \emph{real part} of $s$ is quite
large, the defining series for $L(s)$ converges quite well, if not
exponentially fast, so there is no problem in that range, and by the
functional equation the same is true if the real part of $1-s$ is quite large).

The problems for $\Im(s)$ large are quite specific, and are already crucial
in the case of the Riemann zeta function $\z(s)$. It is by an efficient
management of this problem (for instance by using the so-called
\emph{Riemann--Siegel formula}) that one is able to compute billions of
nontrivial zeros of $\z(s)$. We will not consider these problems here, but
concentrate on reasonable ranges of $s$.

The second problem is specific to general $L$-functions as opposed to
$L$-functions attached to Dirichlet characters for instance: in the general
situation, we are given an $L$-function by an Euler product known outside of
a finite and small number of ``bad primes''. Using recipes dating to the
late 1960's and well explained in a beautiful paper of Serre \cite{Ser}, one
can give the ``gamma factor'' $\ga(s)$, and some (but not all) the information
about the ``conductor'', which is the exponential factor, at least in the
case of $L$-functions of varieties, or more generally of motives.

\smallskip

We will ignore these problems and assume that we know all the bad primes,
gamma factor, conductor, and root number. Note that if we know the gamma
factor and the bad primes, using the formulas that we will give below for
different values of the argument it is easy to recover the conductor and the
root number. What is most difficult to obtain are the Euler factors at the
bad primes, and this is the object of current work.

\subsection{Dirichlet $L$-Functions}

Let $\chi$ be a Dirichlet character modulo $N$. We define the $L$-function
attached to $\chi$ as the complex function
$$L(\chi,s)=\sum_{n\ge1}\dfrac{\chi(n)}{n^s}\;.$$
Since $|\chi(n)|\le1$, it is clear that $L(\chi,s)$ converges absolutely
for $\Re(s)>1$. Furthermore, since $\chi$ is multiplicative, as for the
Riemann zeta function we have an \emph{Euler product}
$$L(\chi,s)=\prod_p\dfrac{1}{1-\chi(p)/p^s}\;.$$
The denominator of this product being generically of degree $1$, this is
also called an $L$-function of degree $1$, and conversely, with a suitable
definition of the notion of $L$-function, one can show that these are the
only $L$-functions of degree $1$.

If $f$ is the conductor of $\chi$ and $\chi_f$ is the character modulo $f$
equivalent to $\chi$, it is clear that
$$L(\chi,s)=\prod_{p\mid N, p\nmid f}(1-\chi_f(p)p^{-s})L(\chi_f,s)\;,$$
so if desired we can always reduce to primitive characters, and this is
what we will do from now on.

Dirichlet $L$-series have important analytic and arithmetic properties, some
of them conjectural (such as the Riemann Hypothesis), which should (again
conjecturally) be shared by all global $L$-functions, see the discussion
in the introduction. We first give the following:

\begin{theorem} Let $\chi$ be a \emph{primitive} character modulo $N$, and
let $e=0$ or $1$ be such that $\chi(-1)=(-1)^e$. 
\begin{enumerate}
\item (Analytic continuation.) 
The function $L(\chi,s)$ can be analytically continued to the whole
complex plane into a meromorphic function, which is in fact holomorphic
except in the special case $N=1$, $L(\chi,s)=\z(s)$, where it has a unique
pole, at $s=1$, which is simple with residue $1$.
\item (Functional equation.) 
There exists a \emph{functional equation} of the following form:
letting $\ga_{\R}(s)=\pi^{-s/2}\G(s/2)$, we set 
$$\Lambda(\chi,s)=N^{(s+e)/2}\ga_{\R}(s+e)L(\chi,s)\;,$$
where $e$ is as above. Then
$$\Lambda(\chi,1-s)=\om(\chi)\Lambda(\ov{\chi},s)\;,$$
where $\om(\chi)$, the so-called \emph{root number}, is a complex
number of modulus $1$ given by the formula
$\om(\chi)=\g(\chi)/(i^eN^{1/2})$.
\item (Special values.)
For each integer $k\ge1$ we have the \emph{special values}
$$L(\chi,1-k)=-\dfrac{B_k(\chi)}{k}-\delta_{N,1}\delta_{k,1}\;,$$
where $\delta$ is the Kronecker symbol, and the \emph{generalized Bernoulli
numbers} $B_k(\chi)$ are easily computable algebraic numbers. In particular,
when $k\not\equiv e\pmod{2}$ we have $L(\chi,1-k)=0$ (except when $k=N=1$).

By the functional equation this is equivalent to the formula
for $k\equiv e\pmod{2}$, $k\ge1$:
$$L(\chi,k)=(-1)^{k-1+(k+e)/2}\om(\chi)\dfrac{2^{k-1}\pi^k\ov{B_k(\chi)}}{m^{k-1/2}k!}\;.$$
\end{enumerate}\end{theorem}

To state the next theorem, which for the moment we state for Dirichlet 
$L$-functions, we need still another important special function:

\begin{definition} For $x>0$ we define the \emph{incomplete gamma function}
$\G(s,x)$ by
$$\G(s,x)=\int_x^\infty t^se^{-t}\,\dfrac{dt}{t}\;.$$
\end{definition}

Note that this integral converges for \emph{all} $s\in\C$, and that it
tends to $0$ exponentially fast when $x\to\infty$, more precisely
$\G(s,x)\sim x^{s-1}e^{-x}$. In addition (but this would carry us too far here)
there are many efficient methods to compute it; see however the section
on inverse Mellin transforms below.

\begin{theorem} Let $\chi$ be a \emph{primitive} character modulo $N$. For all 
$A>0$ we have:
\begin{align*}\Gamma\left(\dfrac{s+e}{2}\right)L(\chi,s)&=\delta_{N,1}\pi^{s/2}\left(\dfrac{A^{(s-1)/2}}{s-1}-\dfrac{A^{s/2}}{s}\right)
+\sum_{n\ge 1}\dfrac{\chi(n)}{n^s}\Gamma\left(\dfrac{s+e}{2},\dfrac{\pi n^2 A}{N}\right)\\
&\phantom{=}+\om(\chi)\left(\dfrac{\pi}{N}\right)^{s-1/2}\sum_{n\ge 1}\dfrac{\ov{\chi}(n)}{n^{1-s}}\Gamma\left(\dfrac{1-s+e}{2},\dfrac{\pi n^2}{AN}\right)\;.\end{align*}
\end{theorem}

\begin{remarks}{\rm \begin{enumerate}
\item Thanks to this theorem, we can compute numerical values of $L(\chi,s)$
(for $s$ in a reasonable range) in time $\Os(N^{1/2})$.
\item The optimal value of $A$ is $A=1$, but the theorem is stated in this
form for several reasons, one of them being that by varying $A$ (for instance
taking $A=1.1$ and $A=0.9$) one can check the correctness of the 
implementation, or even compute the root number $\om(\chi)$ if it is not known.
\item To compute values of $L(\chi,s)$ when $\Im(s)$ is large, one
does not use the theorem as stated, but variants, see \cite{Rub}.
\item The above theorem, called the \emph{approximate functional equation}, 
evidently implies the functional equation itself, so it seems to be more
precise; however this is an illusion since one can show that under very mild
assumptions functional equations in a large class imply corresponding
approximate functional equations. 
\end{enumerate}}
\end{remarks}

\subsection{Approximate Functional Equations}

In fact, let us make this last statement completely precise. For the sake of
simplicity we will assume that the $L$-functions have no poles (this 
corresponds for Dirichlet $L$-functions to the requirement that $\chi$ not be 
the trivial character). We begin by the following (where we restrict to 
certain kinds of gamma products, but it is easy to generalize; incidentally 
recall the \emph{duplication formula} for the gamma function 
$\G(s/2)\G((s+1)/2)=2^{1-s}\pi^{1/2}\G(s)$, which allows the reduction of 
factors of the type $\G(s+a)$ to several of the type $\G(s/2+a')$ and 
conversely).

\begin{definition} Recall that we have defined $\G_{\R}(s)=\pi^{-s/2}\G(s/2)$,
which is the gamma factor attached to $L$-functions of even characters, for 
instance to $\z(s)$. A \emph{gamma product} is a function of the type
$$\ga(s)=f^{s/2}\prod_{1\le i\le d}\G_{\R}(s+b_j)\;,$$
where $f>0$ is a real number. The number $d$ of gamma factors is called the
\emph{degree} of $\ga(s)$.\end{definition}

Note that the $b_j$ may not be real numbers, but in the case of $L$-functions
attached to motives, they will always be, and in fact be integers.

\begin{proposition} Let $\ga$ be a gamma product.\begin{enumerate}
\item There exists a function
$W(t)$ called the \emph{inverse Mellin transform} of $\ga$ such that
$$\ga(s)=\int_0^\infty t^sW(t)\,dt/t$$
for $\Re(s)$ sufficiently large (greater than the real part of the rightmost
pole of $\ga(s)$ suffices).
\item $W(t)$ is given by the following \emph{Mellin inversion formula} for
$t>0$:
$$W(t)=\M^{-1}(\ga)(t)=\dfrac{1}{2\pi i}\int_{\sigma-i\infty}^{\sigma+i\infty}t^{-s}\ga(s)\,ds\;,$$
for any $\sigma$ larger than the real part of the poles of $\ga(s)$.
\item $W(t)$ tends to $0$ exponentially fast when $t\to+\infty$. More 
precisely, as $t\to\infty$ we have
$$W(t)\sim C\cdot(t/f^{1/2})^B\exp(-\pi d(t/f^{1/2})^{2/d})$$
with $B=(1-d+\sum_{1\le j\le d}b_j)/d$ and $C=2^{(d+1)/2}/d^{1/2}$.
\end{enumerate}\end{proposition}

\begin{definition} Let $\ga(s)$ be a gamma product and $W(t)$ its inverse
Mellin transform. The \emph{incomplete gamma product} $\ga(s,x)$ is defined 
for $x>0$ by
$$\ga(s,x)=\int_x^\infty t^sW(t)\,\dfrac{dt}{t}\;.$$
\end{definition}

Note that this integral always converges since $W(t)$ tends to $0$ 
exponentially fast when $t\to\infty$. In addition, thanks to the above 
proposition it is immediate to show the following:

\begin{corollary}\label{asympunsmooth}\begin{enumerate}
\item For any $\sigma$ larger than the real part of the poles of $\ga(s)$
we have
$$\ga(s,x)=\dfrac{x^s}{2\pi i}\int_{\sigma-i\infty}^{\sigma+i\infty}\dfrac{x^{-z}\ga(z)}{z-s}\,dz\;.$$
\item For $s$ fixed, as $x\to\infty$ we have with the same
constants $B$ and $C$ as above
$$\ga(s,x)\sim \dfrac{C}{2\pi}x^s(x/f^{1/2})^{B-2/d}\exp(-\pi d(x/f^{1/2})^{2/d})$$
so has essentially the same exponential decay as $W(x)$.
\end{enumerate}
\end{corollary}

The first theorem, essentially due to Lavrik, which is an exercise in complex
integration is as follows (recall that a function $f$ is of \emph{finite
order} $\al\ge0$ if for all $\eps>0$ and sufficiently large $|z|$ we have
$|f(z)|\le \exp(|z|^{\al+\eps})$):

\begin{theorem}\label{thmapprox} For $i=1$ and $i=2$, let 
$L_i(s)=\sum_{n\ge 1}a_i(n)n^{-s}$ be
Dirichlet series converging in some right half-plane $\Re(s)\ge\sigma_0$.
For $i=1$ and $i=2$, let $\ga_i(s)$ be gamma products having the same
degree $d$. Assume that the functions $\Lambda_i(s)=\ga_i(s)L_i(s)$
extend analytically to $\C$ into holomorphic functions of \emph{finite order},
and that we have the functional equation 
$$\Lambda_1(k-s)=w\cdot\Lambda_2(s)$$ for some constant
$w\in\C^*$ and some real number $k$. 

Then for all $A>0$, we have
$$\Lambda_1(s)=\sum_{n\ge1}\dfrac{a_1(n)}{n^s}\ga_1(s,nA)+
w\sum_{n\ge1}\dfrac{a_2(n)}{n^{k-s}}\ga_2\Bigl(k-s,\dfrac{n}{A}\Bigr)$$
and symmetrically
$$\Lambda_2(s)=\sum_{n\ge1}\dfrac{a_2(n)}{n^s}\ga_2\Bigl(s,\dfrac{n}{A}\Bigr)+
w^{-1}\sum_{n\ge1}\dfrac{a_1(n)}{n^{k-s}}\ga_1(k-s,nA)\;,$$
where $\ga_i(s,x)$ are the corresponding incomplete gamma products.
\end{theorem}

Note that, as already mentioned, it is immediate to modify this theorem
to take into account possible poles of $L_i(s)$.

Since the incomplete gamma products $\ga_i(s,x)$ tend to $0$ exponentially
fast when $x\to\infty$, the above formulas are rapidly
convergent series. We can make this more precise: if we write as above
$\ga_i(s,x)\sim C_ix^{B'_i}\exp(-\pi d(x/f_i^{1/2})^{2/d})$, since
the convergence of the series is dominated by the exponential term, choosing
$A=1$, to have the $n$th term of the series less than $e^{-D}$, say, we
need (approximately) $\pi d(n/f^{1/2})^{2/d}>D$, in other words
$n>(D/(\pi d))^{d/2}f^{1/2}$, with $f=\max(f_1,f_2)$. Thus, if the 
``conductor'' $f$ is large, we may have some trouble. But this stays 
reasonable for $f<10^8$, say.

\smallskip

The above argument leads to the belief that, apart from special values which
can be computed by other methods, the computation of values of $L$-functions
of conductor $f$ requires at least $C\cdot f^{1/2}$ operations. It has
however been shown by Hiary (see \cite{Hia}), that if $f$ is far from
squarefree (for instance if $f=m^3$ for Dirichlet $L$-functions), the
computation can be done faster (in $\Os(m)$ in the case $f=m^3$), at least in
the case of Dirichlet $L$-functions.

\medskip

For practical applications, it is very useful to introduce an additional
function as a parameter. We state the following version due to Rubinstein
(see \cite{Rub}), whose proof is essentially identical to that of the
preceding version. To simplify the exposition, we again assume that the
$L$ function has no poles (it is easy to generalize), but also that
$L_2=\ov{L_1}$.

\begin{theorem} Let $L(s)=\sum_{n\ge1}a(n)n^{-s}$ be an $L$-function as above
with functional equation $\Lambda(k-s)=w\ov{\Lambda}(s)$ with
$\Lambda(s)=\ga(s)L(s)$. For simplicity of exposition, assume that $L(s)$
has no poles in $\C$. Let $g(s)$ be an entire function such that for fixed
$s$ we have $|\Lambda(z+s)g(z+s)/z|\to0$ as $\Im(z)\to\infty$ in any bounded
strip $|\Re(z)|\le \alpha$. We have
$$\Lambda(s)g(s)=\sum_{n\ge1}\dfrac{a(n)}{n^s}f_1(s,n)
+\om\sum_{n\ge1}\dfrac{\ov{a(n)}}{n^{k-s}}f_2(k-s,n)\;,$$
where
$$f_1(s,x)=\dfrac{x^s}{2\pi i}\int_{\sigma-i\infty}^{\sigma+i\infty}\dfrac{\ga(z)g(z)x^{-z}}{z-s}\,dz\text{\quad and\quad}f_2(s,x)=\dfrac{x^s}{2\pi i}\int_{\sigma-i\infty}^{\sigma+i\infty}\dfrac{\ga(z)\ov{g(k-\ov{z})}x^{-z}}{z-s}\,dz\;,$$
where $\sigma$ is any real number greater than the real parts of all the
poles of $\ga(z)$ and than $\Re(s)$.
\end{theorem}

Several comments are in order concerning this theorem:

\begin{enumerate}\item As already mentioned, the proof is a technical but
elementary exercise in complex analysis. In particular, it is very easy to
modify the formula to take into account possible poles of $L(s)$, see
\cite{Rub} once again.
\item
As in the unsmoothed case, the functions $f_i(s,x)$ are exponentially 
decreasing as $x\to\infty$. Thus this gives fast formulas for computing values
of $L(s)$ for reasonable values of $s$. The very simplest case of this 
approximate functional equation, even simpler than the Riemann zeta function,
is for the computation of the value at $s=1$ of the $L$-function of an
\emph{elliptic curve} $E$: if the sign of its functional equation is equal
to $+1$ (otherwise $L(E,1)=0$), the (unsmoothed) formula reduces to
$$L(E,1)=2\sum_{n\ge1}\dfrac{a(n)}{n}e^{-2\pi n/N^{1/2}}\;,$$
where $N$ is the conductor of the curve.
\item It is not difficult to show that as $n\to\infty$ we have a similar
behavior for the functions $f_i(s,n)$ as in the unsmoothed case
(Corollary \ref{asympunsmooth}), i.e.,
$$f_i(s,n)\sim C_i\cdot n^{B'_i}e^{-\pi d(n/N^{1/2})^{2/d}}$$
for some explicit constants $C_i$ and $B'_i$ (in the preceding example $d=2$).
\item The theorem can be used with $g(s)=1$ to compute values of
$L(s)$ for ``reasonable'' values of $s$. When $s$ is unreasonable,
for instance when $s=1/2+iT$ with $T$ large (to check the Riemann
hypothesis for instance), one chooses other functions $g(s)$ adapted
to the computation to be done, such as $g(s)=e^{is\th}$ or
$g(s)=e^{-a(s-s_0)^2}$; I refer to Rubinstein's paper for detailed
examples.
\item By choosing two very simple functions $g(s)$ such as $a^s$ for
two different values of $a$ close to $1$, one can compute numerically
the value of the root number $\om$ if it is unknown. In a similar manner,
if the $a(n)$ are known but not $\om$ nor the conductor $N$, by
choosing a few easy functions $g(s)$ one can find them. But much more
surprisingly, if almost nothing is known apart from the gamma factors and $N$,
say, by cleverly choosing a number of functions $g(s)$ and applying techniques
from numerical analysis such as singular value decomposition and least
squares methods, one can prove or disprove (numerically of course)
the existence of an $L$-function having the given gamma factors and
conductor, and find its first few Fourier coefficients if they exist.
This method has been used extensively by D.~Farmer in his search for
$\GL_3(\Z)$ and $\GL_4(\Z)$ Maass forms, by Poor and Yuen in computations 
related to the paramodular conjecture of Brumer--Kramer and abelian surfaces,
and by A.~Mellit in the search of $L$-functions of degree $4$ with
integer coefficients and small conductor. Although a fascinating and
active subject, it would carry us too far afield to give more detailed 
explanations.\end{enumerate}

\subsection{Inverse Mellin Transforms}

We thus see that it is necessary to compute inverse Mellin transforms of some
common gamma factors. Note that the exponential factors (either involving the
conductor and/or $\pi$) are easily taken into account: if
$\ga(s)=\M(W)(s)=\int_0^\infty W(t)t^s\,dt/t$ is the Mellin transform of $W(t)$,
we have for $a>0$, setting $u=at$:
$$\int_0^\infty W(at)t^s\,dt/t=\int_0^\infty W(u)u^sa^{-s}\,du/u=a^{-s}\ga(s)\;,$$
so the inverse Mellin transform of $a^{-s}\ga(s)$ is simply $W(at)$.

As we have seen, there exists an explicit formula for the inverse Mellin 
transform, which is immediate from the Fourier inversion formula.
We will see that although this looks quite technical, it is in practice very 
useful for computing inverse Mellin transforms.

Let us look at the simplest examples (omitting the exponential factor $f^{s/2}$
thanks to the above remark):

\begin{enumerate}
\item $\M^{-1}(\G_{\R}(s))=2e^{-\pi x^2}$ (this occurs for $L$-functions of
even characters, and in particular for $\z(s)$).
\item $\M^{-1}(\G_{\R}(s+1))=2xe^{-\pi x^2}$ (this occurs for $L$-functions of
odd characters).
\item $\M^{-1}(\G_{\C}(s))=2e^{-2\pi x}$ (this occurs for $L$-functions
attached to modular forms and to elliptic curves).
\item $\M^{-1}(\G_{\R}(s)^2)=4K_0(2\pi x)$ (this occurs for instance for
Dedekind zeta functions of real quadratic fields). Here $K_0(z)$ is a
well-known special function called a $K$-Bessel function. Of course this is
just a name, but it can be computed quite efficiently and can be found in
all computer algebra packages.
\item $\M^{-1}(\G_{\C}(s)^2)=8K_0(4\pi x^{1/2})$.
\item $\M^{-1}(\G_{\C}(s)\G_{\C}(s-1))=8K_1(4\pi x^{1/2})/x^{1/2}$, where
$K_1(z)$ is another $K$-Bessel function which can be defined by 
$K_1(z)=-K_0'(z)$.
\end{enumerate}

\begin{exercise} Prove all these formulas.
\end{exercise}

It is clear however that when the gamma factor is more complicated, we
cannot write such ``explicit'' formulas, for instance what must be done for
$\ga(s)=\G_{\C}(s)\G_{\R}(s)$ or $\ga(s)=\G_{\R}(s)^3$ ? In fact all of
the above formulas involving $K$-Bessel functions are ``cheats'' in the sense
that we have simply given a \emph{name} to these inverse Mellin transform,
without explaining how to compute them.

\smallskip

However the Mellin inversion formula does provide such a method. The
main point to remember (apart of course from the crucial use of the Cauchy
residue formula and contour integration), is that the gamma function
\emph{tends to zero exponentially fast} on vertical lines, uniformly in the
real part (this may seem surprising if you have never seen it since
the gamma function grows so fast on the real axis, see appendix).
This exponential decrease implies that in the Mellin inversion
formula we can \emph{shift} the line of integration without changing
the value of the integral, as long as we take into account the residues
of the poles which are encountered along the way.

The line $\Re(s)=\sigma$ has been chosen so that $\sigma$ is larger than
the real part of any pole of $\ga(s)$, so shifting to the right does not
bring anything. On the other hand, shifting towards the left shows that
for any $r<0$ not a pole of $\ga(s)$ we have
$$W(t)=\sum_{\substack{s_0\text{ pole of $\ga(s)$}\\\Re(s_0)>r}}\Res_{s=s_0}(t^{-s}\ga(s))+\dfrac{1}{2\pi i}\int_{r-i\infty}^{r+i\infty}t^{-s}\ga(s)\,ds\;.$$
Using the reflection formula for the gamma function 
$\G(s)\G(1-s)=\pi/\sin(s\pi)$, it is easy to show that if $r$ stays say
half-way between the real part of two consecutive poles of $\ga(s)$ then
$\ga(s)$ will tend to $0$ exponentially fast on $\Re(s)=r$ as $r\to-\infty$,
in other words that the integral tends to $0$ (exponentially fast). We thus
have the \emph{exact formula}
$$W(t)=\sum_{s_0\text{ pole of $\ga(s)$}}\Res_{s=s_0}(t^{-s}\ga(s))\;.$$
Let us see the simplest examples of this, taken from those given above.
\begin{enumerate}
\item For $\ga(s)=\G_{\C}(s)=2\cdot(2\pi)^{-s}\G(s)$ the poles of $\ga(s)$
are for $s_0=-n$, $n$ a positive or zero integer, and since
$\G(s)=\G(s+n+1)/((s+n)(s+n-1)\cdots s)$, the residue at $s_0=-n$ is equal to 
$$2\cdot (2\pi t)^n\G(1)/((-1)(-2)\cdots(-n))=(-1)^n(2\pi t)^n/n!\;,$$
so we obtain $W(t)=2\sum_{n\ge0}(-1)^n(2\pi t)^n/n!=2\cdot e^{-2\pi t}$.
Of course we knew that!
\item For $\ga(s)=\G_{\C}(s)^2=4(2\pi)^{-2s}\G(s)^2$, the inverse Mellin
transform is $8K_0(4\pi x^{1/2})$ whose expansion we do \emph{not} yet know.
The poles of $\ga(s)$ are again for $s_0=-n$, but here all the poles are
double poles, so the computation is slightly more complicated. More precisely
we have $$\G(s)^2=\G(s+n+1)^2/((s+n)^2(s+n-1)^2\cdots s^2)\;,$$ so setting
$s=-n+\eps$ with $\eps$ small this gives
\begin{align*}\G(-n+\eps)^2&=\dfrac{\G(1+\eps)^2}{\eps^2}\dfrac{1}{(1-\eps)^2\cdots(n-\eps)^2}\\
&=\dfrac{1+2\G'(1)\eps+O(\eps^2)}{n!^2\eps^2}(1+2\eps/1)(1+2\eps/2)\cdots(1+2\eps/n)\\
&=\dfrac{1+2\G'(1)\eps+O(\eps^2)}{n!^2\eps^2}(1+2H_n\eps)\;,\end{align*}
where we recall that $H_n=\sum_{1\le j\le n}1/j$ is the harmonic sum.
Since $(4\pi^2t)^{-(-n+\eps)}=(4\pi^2t)^{n-\eps}=(4\pi^2t)^n(1-\eps\log(4\pi^2t)+O(\eps^2))$, it follows
that
$$(4\pi^2t)^{-(-n+\eps)}\G(-n+\eps)^2=\dfrac{(4\pi^2t)^n}{n!^2\eps^2}(1+\eps(2H_n+2\G'(1)-\log(4\pi^2t)))\;,$$
so that the residue of $\ga(s)$ at $s=-n$ is equal to
$4((4\pi^2t)^n/n!^2)(2H_n+2\G'(1)-\log(4\pi^2t))$.
We thus have
$2K_0(4\pi t^{1/2})=\sum_{n\ge0}((4\pi^2t)^n/n!^2)(2H_n+2\G'(1)-\log(4\pi^2t))$,
hence using the easily proven fact that $\G'(1)=-\ga$, where
$$\ga=\lim_{n\to\infty}(H_n-\log(n))=0.57721566490\dots$$
is Euler's constant, this gives finally the expansion
$$K_0(t)=\sum_{n\ge0}\dfrac{(t/2)^{2n}}{n!^2}(H_n-\ga-\log(t/2))\;.$$
\end{enumerate}

\begin{exercise} In a similar manner, or directly from this formula, find the
expansion of $K_1(t)$.
\end{exercise}

\begin{exercise}\label{exga}
Like all inverse Mellin transforms of gamma factors, the
function $K_0(x)$ tends to $0$ exponentially fast as $x\to\infty$
(more precisely $K_0(x)\sim(2x/\pi)^{-1/2}e^{-x}$). Note that this is
absolutely not ``visible'' on the expansion given above. Use this remark
and the above expansion to write an algorithm which computes Euler's
constant $\ga$ \emph{very efficiently} to a given accuracy. 
\end{exercise}

It must be remarked that even though the series defining the inverse Mellin
transform converge for \emph{all} $x>0$, one need a large number of terms
before the terms become very small when $x$ is large. For instance, we have
seen that for $\ga(s)=\G(s)$ we have
$W(t)=\M^{-1}(\ga)(t)=\sum_{n\ge0}(-1)^nt^n/n!=e^{-t}$,
but this series is not very good for computing $e^{-t}$.

\begin{exercise} Show that for $t>0$, to compute $e^{-t}$ to any reasonable
accuracy (even to $1$ decimal) we must take at least $n>3.6\cdot t$ 
($e=2.718...$), and work to accuracy at most $e^{-2t}$ in an evident sense.
\end{exercise}

The reason that this is not a good way is that there is catastrophic
cancellation in the series. One way to circumvent this problem is to
compute $e^{-t}$ as
$$e^{-t}=1/e^t=1/\sum_{n\ge0}t^n/n!\;,$$
and the cancellation problem disappears. However this is very special to
the exponential function, and is not applicable for instance to the
$K$-Bessel function.

Nonetheless, an important result is that for any inverse Mellin transform
as above, or more importantly for the corresponding incomplete gamma
product, there exist \emph{asymptotic expansions} as $x\to\infty$, in other
words nonconvergent series which however give a good approximation if limited
to a few terms.

Let us take the simplest example of the incomplete gamma function
$\G(s,x)=\int_x^\infty t^se^{-t}\,dt/t$. The \emph{power series} expansion
is easily seen to be (at least for $s$ not a negative or zero integer,
otherwise the formula must be slightly modified):
$$\G(s,x)=\G(s)-\sum_{n\ge0}(-1)^n\dfrac{x^{n+s}}{n!(s+n)}\;,$$
which has the same type of (bad when $x$ is large) convergence behavior as
$e^{-x}$. On the other hand, it is immediate to prove by integration by parts
that
\begin{align*}\G(s,x)&=e^{-x}x^{s-1}\left(1+\dfrac{s-1}{x}+\dfrac{(s-1)(s-2)}{x^2}+\cdots\right.\\
&\phantom{=}\left.+\dfrac{(s-1)(s-2)\cdots(s-n)}{x^n}+R_n(s,x)\right)\;,\end{align*}
and one can show that in reasonable ranges of $s$ and $x$ the modulus of
$R_n(s,x)$ is smaller than the first ``neglected term'' in an evident sense.
This is therefore quite a practical method for computing these functions
when $x$ is rather large.

\begin{exercise} Explain why the asymptotic series above terminates when
$s$ is a strictly positive integer.
\end{exercise}

\subsection{Hadamard Products and Explicit Formulas}

This could be the subject of a course in itself, so we will be quite
brief. I refer to Mestre's paper \cite{Mes} for a precise and
general statement (note that there are quite a number of evident
misprints in the paper).

\smallskip

In Theorem \ref{thmapprox} we assume that the $L$-series that we consider
satisfy a functional equation, together with some mild growth conditions,
in particular that they are of finite order. According to a well-known
theorem of complex analysis, this implies that they have a so-called
\emph{Hadamard product}, see Appendix. For instance, in the case of the
Riemann zeta function, which is of order $1$, we have
$$\z(s)=\dfrac{e^{bs}}{s(s-1)\G(s/2)}\prod_{\rho}\left(1-\dfrac{s}{\rho}\right)e^{s/\rho}\;,$$
where the product is over all nontrivial zeros of $\z(s)$ (i.e., such that
$0\le\Re(\rho)\le1$), and $b=\log(2\pi)-1-\ga$. In fact, this can be written
in a much nicer way as follows: recall that 
$\Lambda(s)=\pi^{-s/2}\G(s/2)\z(s)$ satisfies $\Lambda(1-s)=\Lambda(s)$. Then
$$s(s-1)\Lambda(s)=\prod_{\rho}\left(1-\dfrac{s}{\rho}\right)\;,$$
where it is now understood that the product is taken as the limit as 
$T\to\infty$ of $\prod_{|\Im(\rho)|\le T}(1-s/\rho)$.

\smallskip

However, almost all $L$-functions that are used in number theory not only
have the above properties, but have also \emph{Euler products}. Taking again
the example of $\z(s)$, we have for $\Re(s)>1$ the Euler product
$\z(s)=\prod_p(1-1/p^s)^{-1}$. It follows that (in a suitable range of $s$)
we have equality between two products, hence taking logarithms, equality
between two \emph{sums}. In our case the Hadamard product gives
$$\log(\Lambda(s))=-\log(s(s-1))+\sum_{\rho}\log(1-s/\rho)\;,$$
while the Euler product gives
\begin{align*}\log(\Lambda(s))&=-(s/2)\log(\pi)+\log(\G(s/2))-\sum_p\log(1-1/p^s)\\
&=-(s/2)\log(\pi)+\log(\G(s/2))+\sum_{p,k\ge1}1/(kp^{ks})\;,\end{align*}
Equating the two sides gives a relation between on the one hand a
sum over the nontrivial zeros of $\z(s)$, and on the other hand a
sum over prime powers.

In itself, this is not very useful. The crucial idea is to introduce
a test function $F$ which we will choose to the best of our interests,
and obtain a formula depending on $F$ and some transforms of it.

This is in fact quite easy to do, and even though not very useful in this
case, let us perform the computation for Dirichlet $L$-function of 
even primitive characters.

\begin{theorem} Let $\chi$ be an even primitive Dirichlet character of 
conductor $N$, and let $F$ be a real function satisfying a number of easy
technical conditions (see \cite{Mes}). We have the \emph{explicit formula}:
\begin{align*}\sum_{\rho}\Phi(\rho)&-2\delta_{N,1}\int_{-\infty}^\infty F(x)\cosh(x/2)\,dx\\
&=-\sum_{p,k\ge1}\dfrac{\log(p)}{p^{k/2}}(\chi^k(p)F(k\log(p))+\ov{\chi^k(p)}F(-k\log(p)))\\
&\phantom{=}+F(0)\log(N/\pi)\\
&\phantom{=}+\int_0^\infty\left(\dfrac{e^{-x}}{x}F(0)-\dfrac{e^{-x/4}}{1-e^{-x}}\dfrac{F(x/2)+F(-x/2)}{2}\right)\,dx\;,\end{align*}
where we set
$$\Phi(s)=\int_{-\infty}^\infty F(x)e^{(s-1/2)x}\,dx\;,$$
and as above the sum on $\rho$ is a sum over all the nontrivial zeros of 
$L(\chi,s)$ taken symmetrically 
($\sum_{\rho}=\lim_{T\to\infty}\sum_{|\Im(\rho)|\le T}$).
\end{theorem}

\begin{remarks}{\rm \begin{enumerate}
\item Write $\rho=1/2+i\ga$ (if the GRH is true all $\ga$ are real,
but even without GRH we can always write this). Then
$$\Phi(\rho)=\int_{-\infty}^\infty F(x)e^{i\ga x}\,dx=\widehat{F}(\ga)$$
is simply the value at $\ga$ of the \emph{Fourier transform}
$\widehat{F}$ of $F$.
\item It is immediate to generalize to odd $\chi$ or more general
$L$-functions:

\begin{exercise} After studying the proof, generalize to an arbitrary pair
of $L$-functions as in Theorem \ref{thmapprox}.
\end{exercise}
\end{enumerate}}
\end{remarks}

\begin{proof} The proof is not difficult, but involves a number of integral
transform computations. We will omit some detailed justifications which
are in fact easy but boring.

As in the theorem, we set
$$\Phi(s)=\int_{-\infty}^\infty F(x)e^{(s-1/2)x}\,dx\;,$$
and we first prove some lemmas.

\begin{lemma} 
We have the inversion formulas valid for any $c>1$:
$$F(x)=e^{x/2}\int_{c-i\infty}^{c+i\infty}\Phi(s)e^{-sx}\,ds\;.$$
$$F(-x)=e^{x/2}\int_{c-i\infty}^{c+i\infty}\Phi(1-s)e^{-sx}\,ds\;.$$
\end{lemma}

\begin{proof} This is in fact a hidden version of the Mellin inversion formula:
setting $t=e^x$ in the definition of $\Phi(s)$, we deduce that
$\Phi(s)=\int_0^\infty F(\log(t))t^{s-1/2}\,dt/t$, so that
$\Phi(s+1/2)$ is the Mellin transform of $F(\log(t))$. By Mellin inversion
we thus have for sufficiently large $\sigma$:
$$F(\log(t))=\dfrac{1}{2\pi i}\int_{\sigma-i\infty}^{\sigma+i\infty}\Phi(s+1/2)t^{-s}\,ds\;,$$
so changing $s$ into $s-1/2$ and $t$ into $e^x$ gives the first formula
for $c=\sigma+1/2$ sufficiently large, and the assumptions on $F$ (which
we have not given) imply that we can shift the line of integration to any 
$c>1$ without changing the integral. 

For the second formula, we simply note that
$$\Phi(1-s)=\int_{-\infty}^\infty F(x)e^{-(s-1/2)x}\,dx
=\int_{-\infty}^\infty F(-x)e^{(s-1/2)x}\,dx\;,$$
so we simply apply the first formula to $F(-x)$.\fp\end{proof}

\begin{corollary} For any $c>1$ and any $p\ge1$ we have
\begin{align*}
\int_{c-i\infty}^{c+i\infty}\Phi(s)p^{-ks}\,ds&=F(k\log(p))p^{-k/2}\text{\quad and}\\
\int_{c-i\infty}^{c+i\infty}\Phi(1-s)p^{-ks}\,ds&=F(-k\log(p))p^{-k/2}\;.
\end{align*}
\end{corollary}

\begin{proof} Simply apply the lemma to $x=k\log(p)$.\fp\end{proof}

Note that we will also use this corollary for $p=1$.

\begin{lemma} Denote as usual by $\psi(s)$ the logarithmic derivative
$\G'(s)/\G(s)$ of the gamma function. We have
\begin{align*}
\int_{c-i\infty}^{c+i\infty}\Phi(s)\psi(s/2)&=\int_0^\infty\left(\dfrac{e^{-x}}{x}F(0)-\dfrac{e^{-x/4}}{1-e^{-x}}F(x/2)\right)\,dx\text{\quad and}\\
\int_{c-i\infty}^{c+i\infty}\Phi(1-s)\psi(s/2)&=\int_0^\infty\left(\dfrac{e^{-x}}{x}F(0)-\dfrac{e^{-x/4}}{1-e^{-x}}F(-x/2)\right)\,dx\;.\end{align*}
\end{lemma}

\begin{proof} We use one of the most common integral representations of
$\psi$, see Proposition 9.6.43 of \cite{Coh4}: we have
$$\psi(s)=\int_0^\infty\left(\dfrac{e^{-x}}{x}-\dfrac{e^{-sx}}{1-e^{-x}}\right)\,dx\;.$$
Thus, assuming that we can interchange integrals (which is easy to justify),
we have, using the preceding lemma:

\begin{align*}\int_{c-i\infty}^{c+i\infty}\Phi(s)\psi(s/2)\,ds&=\int_0^\infty\left(\dfrac{e^{-x}}{x}\int_{c-i\infty}^{c+i\infty}\Phi(s)\,ds\right.\\
&\phantom{=}\left.-\dfrac{1}{1-e^{-x}}\int_{c-i\infty}^{c+i\infty}\Phi(s)e^{-(s/2)x}\,ds\right)\,dx\\
&=\int_0^\infty\left(\dfrac{e^{-x}}{x}F(0)-\dfrac{e^{-x/4}}{1-e^{-x}}F(x/2)\right)\,dx\;,\end{align*}
proving the first formula, and the second follows by changing $F(x)$ into
$F(-x)$.\fp\end{proof}

\smallskip

\noindent
{\it Proof of the theorem.\/} Recall from above that if we set
$\L(s)=N^{s/2}\pi^{-s/2}\G(s/2)L(\chi,s)$ we have the functional
equation $\L(1-s)=\om(\chi)\L(\ov{\chi},s)$ for some $\om(\chi)$ of modulus
$1$.

For $c>1$, consider the following integral
$$J=\dfrac{1}{2i\pi}\int_{c-i\infty}^{c+i\infty}\Phi(s)\dfrac{\L'(s)}{\L(s)}\,ds\;,$$
which by our assumptions does not depend on $c>1$. We shift the line of
integration to the left (it is easily seen that this is allowed) to the
line $\Re(s)=1-c$, so by the residue theorem we obtain
$$J=S+\dfrac{1}{2i\pi}\int_{1-c-i\infty}^{1-c+i\infty}\Phi(s)\dfrac{\L'(s)}{\L(s)}\,ds\;,$$
where $S$ is the sum of the residues in the rectangle $[1-c,c]\times\R$.
We first have possible poles at $s=0$ and $s=1$, which occur only for $N=1$,
and they contribute to $S$
$$-\delta_{N,1}(\Phi(0)+\Phi(1))=-2\delta_{N,1}\int_{-\infty}^\infty F(x)\cosh(x/2)\,dx\;,$$
and of course second we have the contributions from the nontrivial zeros 
$\rho$, which contribute $\sum_{\rho}\Phi(\rho)$, where it is understood that
zeros are counted with multiplicity, so that
$$S=-2\delta_{N,1}\int_{-\infty}^\infty F(x)\cosh(x/2)\,dx+\sum_{\rho}\Phi(\rho)\;.$$
On the other hand, by the functional equation we have
$\L'(1-s)/\L(1-s)=-\oL'(s)/\oL(s)$ (note that this does not involve 
$\om(\chi)$), where we write $\oL(s)$ for $\L(\ov{\chi},s)$, so that
\begin{align*}\int_{1-c-i\infty}^{1-c+i\infty}\Phi(s)\dfrac{\L'(s)}{\L(s)}\,ds
&=\int_{c-i\infty}^{c+i\infty}\Phi(1-s)\dfrac{\L'(1-s)}{\L(1-s)}\,ds\\
&=-\int_{c-i\infty}^{c+i\infty}\Phi(1-s)\dfrac{\oL'(s)}{\oL(s)}\,ds\;.\end{align*}
Thus,
\begin{align*}S&=J-\dfrac{1}{2i\pi}\int_{1-c-i\infty}^{1-c+i\infty}\Phi(s)\dfrac{\L'(s)}{\L(s)}\,ds\\
&=\dfrac{1}{2i\pi}\int_{c-i\infty}^{c+i\infty}\left(\Phi(s)\dfrac{\L'(s)}{\L(s)}+\Phi(1-s)\dfrac{\oL'(s)}{\oL(s)}\right)\,ds\;.\end{align*}
Now by definition we have as above
$$\log(\L(s))=\dfrac{s}{2}\log(N/\pi)+\log\left(\G\left(\dfrac{s}{2}\right)\right)+\sum_{p,k\ge1}\dfrac{\chi^k(p)}{kp^{ks}}$$
(where the double sum is over primes and integers $k\ge1$), so
$$\dfrac{\L'(s)}{\L(s)}=\dfrac{1}{2}\log(N/\pi)+\dfrac{1}{2}\psi(s/2)
-\sum_{p,k\ge1}\chi^k(p)\log(p)p^{-ks}\;,$$
and similarly for $\oL'(s)/\oL(s)$. Thus, by the above lemmas and corollaries,
we have
$$S=\log(N/\pi)F(0)+J_1-\sum_{p,k\ge1}\dfrac{\log(p)}{p^{k/2}}(\chi^k(p)F(k\log(p))+\ov{\chi^k(p)}F(-k\log(p)))\;,$$
where 
$$J_1=\int_0^\infty\left(\dfrac{e^{-x}}{x}F(0)-\dfrac{e^{-x/4}}{1-e^{-x}}\dfrac{F(x/2)+F(-x/2)}{2}\right)\,dx\;,$$
proving the theorem.\fp\end{proof}

\smallskip

This theorem can be used in several different directions, and has
been an extremely valuable tool in analytic number theory. Just to
mention a few:

\begin{enumerate}\item Since the conductor $N$ occurs, we can obtain
\emph{bounds} on $N$, assuming certain conjectures such as the
generalized Riemann hypothesis. For instance, this is how
Stark--Odlyzko--Poitou--Serre find \emph{lower bounds for discriminants}
of number fields. This is also how Mestre finds lower bounds
for conductors of abelian varieties, and so on.
\item When the $L$-function has a zero at its central point (here of
course it usually does not, but for more general $L$-functions
it is important), this can give good upper bounds for the order
of the zero.
\item More generally, suitable choices of the test functions
can give information on the nontrivial zeros $\rho$ of small
imaginary part.
\end{enumerate}

\section{Some Useful Analytic Computational Tools}

We finish this course by giving a number of little-known numerical methods
which are not always directly related to the computation of $L$-functions, but
which are often very useful.

\subsection{The Euler--MacLaurin Summation Formula}

This numerical method is \emph{very} well-known (there is in fact even a
whole chapter in Bourbaki devoted to it!), and is as old as Taylor's
formula, but deserves to be mentioned since it is very useful. We will be
vague on purpose, and refer to \cite{Bou} or Section 9.2 of \cite{Coh4} for
details. Recall that the \emph{Bernoulli numbers} are defined by the formal
power series
$$\dfrac{T}{e^T-1}=\sum_{n\ge0}\dfrac{B_n}{n!}T^n\;.$$
We have $B_0=0$, $B_1=-1/2$, $B_2=1/6$, $B_3=0$, $B_4=-1/30$, and
$B_{2k+1}=0$ for $k\ge1$. 

Let $f$ be a $C^\infty$ function defined on $\R>0$. The basic statement of
the Euler--MacLaurin formula is that there exists a constant $z=z(f)$ such that
$$\sum_{n=1}^Nf(n)=\int_1^N f(t)\,dt+z(f)+\dfrac{f(N)}{2}+\sum_{1\le k\le p}
\dfrac{B_{2k}}{(2k)!}f^{(2k-1)}(N)+R_p(N)\;,$$
where $R_p(N)$ is ``small'', in general smaller than the first neglected term,
as in most asymptotic series.

The above formula can be slightly modified at will, first by changing the
lower bound of summation and/or of integration (which simply changes the
constant $z(f)$), and second by writing
$\int_1^Nf(t)\,dt+z(f)=z'(f)-\int_N^\infty f(t)\,dt$ (when $f$ tends to
$0$ sufficiently fast for the integral to converge), where
$z'(f)=z(f)+\int_1^\infty f(t)\,dt$. 

\medskip

The Euler--MacLaurin summation formula can be used in many contexts, but we
mention the two most important ones.

$\bullet$ First, to have some idea of the size of $\sum_{n=1}^Nf(n)$.
Let us take an example. Consider $S_2(N)=\sum_{n=1}^N n^2\log(n)$. Note
incidentally that
$$\exp(S_2(N))=\prod_{n=1}^N n^{n^2}=1^{1^2}2^{2^2}\cdots N^{N^2}\;.$$
What is the size of this generalized kind of factorial? Euler--MacLaurin
tells us that there exists a constant $z$ such that
\begin{align*}S_2(N)&=\int_1^N t^2\log(t)\,dt+z+\dfrac{N^2\log(N)}{2}\\
&\phantom{=}+\dfrac{B_2}{2!}(N^2\log(N))'+\dfrac{B_4}{4!}(N^2\log(N))'''+\cdots\;.\end{align*}
We have $\int_1^N t^2\log(t)\,dt=(N^3/3)\log(N)-(N^3-1)/9$, 
$(N^2\log(N))'=2N\log(N)+N$, $(N^2\log(N))''=2\log(N)+3$, and
$(N^2\log(N))'''=2/N$, so using $B_2=1/6$ we obtain for some other constant 
$z'$:
$$S_2(N)=\dfrac{N^3\log(N)}{3}-\dfrac{N^3}{9}+\dfrac{N^2\log(N)}{2}+\dfrac{N\log(N)}{6}+\dfrac{N}{12}+z'+O\left(\dfrac{1}{N}\right)\;,$$
which essentially answers our question, up to the determination of the constant
$z'$. Thus we obtain a generalized Stirling's formula:
$$\exp(S_2(N))=N^{N^3/3+N^2/2+N/6}e^{-(N^3/9-N/12)}C\;,$$
where $C=\exp(z')$ is an a priori unknown constant. In the case of the usual
Stirling's formula we have $C=(2\pi)^{1/2}$, so we can ask for a similar 
formula here. And indeed, such a formula exists: we have
$$C=\exp(\zeta(3)/(4\pi^2))\;.$$

\begin{exercise} Do a similar (but simpler) computation for 
$S_1(N)=\sum_{1\le n\le N}n\log(n)$. The corresponding constant is explicit
but more difficult (it involves $\z'(-1)$; more generally the constant
in $S_r(N)$ involves $\z'(-r)$).
\end{exercise}

$\bullet$ The second use of the Euler--MacLaurin formula is to increase
considerably the speed of convergence of slowly convergent series.
For instance, if you want to compute $\z(3)$ directly using the series
$\z(3)=\sum_{n\ge1}1/n^3$, since the remainder term after $N$ terms is
asymptotic to $1/(2N^2)$ you will never get more than $15$ or $20$ decimals
of accuracy. On the other hand, it is immediate to use Euler--MacLaurin:

\begin{exercise} Write a computer program implementing the computation
of $\z(3)$ (and more generally of $\z(s)$ for reasonable $s$) using
Euler--MacLaurin, and compute it to $100$ decimals.
\end{exercise}

A variant of the method is to compute limits: a typical example is the
computation of Euler's constant
$$\ga=\lim_{N\to\infty}\left(\sum_{n=1}^N\dfrac{1}{n}-\log(N)\right)\;.$$
Using Euler--MacLaurin, it is immediate to find the \emph{asymptotic expansion}
$$\sum_{n=1}^N\dfrac{1}{n}=\log(N)+\ga+\dfrac{1}{2N}-\sum_{k\ge1}\dfrac{B_{2k}}{2kN^{2k}}$$
(note that this is not a misprint, the last denominator is $2kN^{2k}$, not
$(2k)!N^{2k}$).

\begin{exercise} Implement the above, and compute $\ga$ to $100$ decimal
digits.\end{exercise}

Note that this is \emph{not} the fastest way to compute Euler's constant,
the method using Bessel functions given in Exercise \ref{exga} is better.

\subsection{Variant: Discrete Euler--MacLaurin}

One problem with the Euler--MacLaurin method is that we need to compute
the derivatives $f^{(2k-1)}(N)$. When $k$ is tiny, say $k=2$ or $k=3$ this
can be done explicitly. When $f(x)$ has a special form, such as
$f(x)=1/x^{\al}$, it is very easy to compute all derivatives. In fact, this
is more generally the case when the expansion of $f(1/x)$ around $x=0$ is
known explicitly. But in general none of this is available.

One way around this is to use finite differences instead of derivatives:
we can easily compute
$$\Delta_{\delta}(f)(x)=(f(x+\delta)-f(x-\delta))/(2\delta)$$
and iterates of this, where $\delta$ is some fixed and nonzero number.
The choice of $\delta$ is essential: it should not be too large, otherwise
$\Delta_{\delta}(f)$ would be too far away from the true derivative
(which will be reflected in the speed of convergence of the asymptotic
formula), and it should not be too small, otherwise catastrophic cancellation
errors will occur. After numerous trials, the value $\delta=1/4$ seems
reasonable.

One last thing must be done: find the analogue of the Bernoulli numbers.
This is a very instructive exercise which we leave to the reader.

\subsection{Zagier's Extrapolation Method}

The following nice trick is due to D.~Zagier. Assume that you have
a sequence $u_n$ that you suspect of converging to some limit $a_0$ when
$n\to\infty$ in a regular manner. How do you give a reasonable numerical
estimate of $a_0$ ?

Assume for instance that as $n\to\infty$ we have
$u_n=\sum_{0\le i\le p}a_i/n^i+O(n^{-p-1})$ for any $p$. One idea would be to
choosing for $n$ suitable values and solve a linear system. This would in
general be quite unstable and inaccurate. Zagier's trick is instead to
proceed as follows: choose some reasonable integer $k$, say $k=10$, set
$u'_n=n^ku_n$, and compute the $k$th \emph{forward difference} 
$\Delta^k(u'_n)$ of this sequence (the forward difference of a sequence $w_n$
is the sequence $\Delta(w)_n=w_{n+1}-w_n$). Note that
$$u'_n=a_0n^k+\sum_{1\le i\le k}a_in^{k-i}+O(1/n)\;.$$
The two crucial points are the following:

\begin{itemize}\item The $k$th forward difference of a polynomial of degree 
less than or equal to $k-1$ vanishes, and that of $n^k$ is equal to
$k!$. 
\item Assuming reasonable regularity conditions, the $k$th forward difference
of an asymptotic expansion beginning at $1/n$ will begin at $1/n^{k+1}$.
\end{itemize}

Thus, under reasonable assumptions we have
$$a_0=\Delta^k(v)_n/k!+O(1/n^{k+1})\;,$$
so choosing $n$ large enough can give a good estimate for $a_0$.

A number of remarks concerning this basic method:

\begin{remarks}{\rm \begin{enumerate} 
\item It is usually preferable to apply this not to the sequence $u_n$
itself, but for instance to the sequence $u_{n+100}$, if it is not too
expensive to compute, since the first terms of $u_n$ are usually far from
the asymptotic expansion.
\item It is immediate to modify the method to compute further coefficients
$a_1$, $a_2$, etc.
\item If the asymptotic expansion of $u_n$ is (for instance) in powers of
$1/n^{1/2}$, it is not difficult to modify this method, see below.
\end{enumerate}}
\end{remarks}

\medskip

{\bf Example.} Let us compute numerically the constant occurring in
the first example of the use of Euler--MacLaurin that we have given. We
set
$$u_N=\sum_{1\le n\le N}n^2\log(n)-(N^3/3+N^2/2+N/6)\log(N)+N^3/9-N/12\;.$$
We compute for instance that $u_{1000}=0.0304456\cdots$, which has only
$4$ correct decimal digits. On the other hand, if we apply the above
trick with $k=12$ and $N=100$, we find
$$a_0=\lim_{N\to\infty}u_N=0.0304484570583932707802515304696767\cdots$$
with $28$ correct decimal digits: recall that the exact value is
$$\z(3)/(4\pi^2)=0.03044845705839327078025153047115477\cdots\;.$$

\medskip

Assume now that $u_n$ has an asymptotic expansion in integral powers of
$1/n^{1/2}$, i.e.,
$u_n=\sum_{0\le i\le p}a_i/n^{i/2}+O(n^{-(p+1)/2})$ for any $p$. We can modify
the above method as follows. First write
$u_n=v_n+w_n/n^{1/2}$, where $v_n=\sum_{0\le i\le q}a_{2i}/n^i+O(n^{-q-1})$
and $w_n=\sum_{0\le i\le q}a_{2i+1}/n^i+O(n^{-q-1})$ are two sequences as
above. Once again we choose some reasonable integer $k$ such as $k=10$, and
we now multiply the sequence $u_n$ by $n^{k-1/2}$, so we set
$u'_n=n^{k-1/2}u_n=n^{k-1/2}v_n+n^{k-1}w_n$. Thus, when we compute
the $k$th forward difference we will have
$$\Delta^k(n^{k-1/2}v_n)=\dfrac{(k-1/2)(k-3/2)\cdots 1/2}{n^{1/2}}\left(a_0+\sum_{0\le i\le q+k}b_{k,i}/n^i\right)$$
for certain coefficients $b_{k,i}$, while as above since
$n^{k-1}w_n=P_{k-1}(n)+O(1/n)$ for some polynomial $P_{k-1}(n)$ of degree
$k-1$, we have $\Delta^k(n^{k-1}w_n)=O(1/n^k)$. Thus we have essentially
eliminated the sequence $w_n$, so we now apply the usual method to
$v'_n=n^{1/2}\Delta^k(n^{k-1/2}v_n)$, which has an expansion in integral
powers of $1/n$: we will thus have
$$\Delta^k(v'_n)/k!=((k-1/2)(k-3/2)\cdots (1/2))a(0)+O(1/n^k)$$
(in fact we do not even have to take the same $k$ for this last step).

This method can immediately be generalized to sequences $u_n$ having an
asymptotic expansion in integral powers of $n^{1/q}$ for small integers $q$.

\subsection{Computation of Euler Sums and Euler Products}

Assume that we want to compute numerically
$$S_1=\prod_p\left(1+\dfrac{1}{p^2}\right)\;,$$
where here and elsewhere, the expression $\prod_p$ always means the product
over all prime numbers. Trying to compute it using a large table of prime
numbers will not give much accuracy: if we use primes up to $X$, we will
make an error of the order of $1/X$, so it will be next to impossible to
have more than $8$ or $9$ decimal digits.

On the other hand, if we simply notice that $1+1/p^2=(1-1/p^4)/(1-1/p^2)$,
by definition of the Euler product for the Riemann zeta function this implies
that
$$S_1=\dfrac{\z(2)}{\z(4)}=\dfrac{\pi^2/6}{\pi^4/90}=\dfrac{15}{\pi^2}=1.519817754635066571658\cdots$$

Unfortunately this is based on a special identity. What if we wanted instead
to compute $S_2=\prod_p(1+2/p^2)$ ? There is no special identity to help us
here.

The way around this problem is to approximate the function of which we want
to take the product (here $1+2/p^2$) by \emph{infinite products} of values
of the Riemann zeta function. Let us do it step by step before giving the
general formula.

When $p$ is large, $1+2/p^2$ is close to $1/(1-1/p^2)^2$, which is the
Euler factor for $\z(2)^2$. More precisely,
$(1+2/p^2)(1-1/p^2)^2=1-3/p^4+2/p^6$, so we deduce that
$$S_2=\z(2)^2\prod_p(1-3/p^4+2/p^6)=(\pi^4/36)\prod_p(1-3/p^4+2/p^6)\;.$$
Even though this looks more complicated, what we have gained is that the
new Euler product converges \emph{much} faster. Once again, if we compute it
for $p$ up to $10^8$, say, instead of having $8$ decimal digits we now
have approximately $24$ decimal digits (convergence in $1/X^3$ instead
of $1/X$). But there is no reason to stop there: we have
$(1-3/p^4+2/p^6)/(1-1/p^4)^3=1+O(1/p^6)$ with evident notation and explicit
formulas if desired, so we get an even better approximation by writing
$S_2=\z(2)^2/\z(4)^3\prod_p(1+O(1/p^6))$, with convergence in $1/X^5$.
More generally, it is easy to compute by induction exponents $a_n\in\Z$ such
that $S_2=\prod_{2\le n\le N}\z(n)^{a_n}\prod_p(1+O(1/p^{N+1}))$
(in our case $a_n=0$ for $n$ odd but this will not be true in general).
It can be shown in essentially all examples that one can pass to the limit,
and for instance here write $S_2=\prod_{n\ge2}\z(n)^{a_n}$.

\begin{exercise}\begin{enumerate}
\item Compute explicitly the recursion for the $a_n$ in the example of $S_2$.
\item More generally, if $S=\prod_pf(p)$, where $f(p)$ has a convergent
series expansion in $1/p$ starting with $f(p)=1+1/p^b+o(1/p^b)$ with $b>1$
(not necessarily integral), express $S$ as a product of zeta values raised
to suitable exponents, and find the recursion for these exponents.
\end{enumerate}\end{exercise}

An important remark needs to be made here: even though the product
$\prod_{n\ge2}\z(n)^{a_n}$ may be convergent, it may converge rather slowly:
remember that when $n$ is large we have $\z(n)-1\sim1/2^n$, so that in fact
if the $a_n$ grow like $3^n$ the product will not even converge.
The way around this, which must be used even when the product converges, is
as follows: choose a reasonable integer $N$, for instance $N=50$, and
compute $\prod_{p\le 50}f(p)$, which is of course very fast. Then
the tail $\prod_{p>50}f(p)$ of the Euler product will be equal to
$\prod_{n\ge2}\z_{>50}(n)^{a_n}$, where $\z_{>N}(n)$ is the zeta function
without its Euler factors up to $N$, in other words
$\z_{>N}(n)=\z(n)\prod_{p\le N}(1-1/p^n)$ (I am assuming here that we have
zeta values at integers as in the $S_2$ example above, but it is immediate
to generalize). Since $\z_{>N}(n)-1\sim1/(N+1)^n$,
the convergence of our zeta product will of course be considerably faster.

\smallskip

Note that by using the power series expansion of the logarithm
together with \emph{M\"obius inversion}, it is immediate to do the same for
Euler \emph{sums}, for instance to compute $\sum_p1/p^2$ and the like,
see Section 10.3.6 of \cite{Coh4} for details. Using \emph{derivatives} of the
zeta function we can compute Euler sums of the type $\sum_p\log(p)/p^2$, and
using antiderivatives we can compute sums of the type $\sum_p1/(p^2\log(p))$.
We can even compute sums of the form $\sum_p\log(\log(p))/p^2$, but this
is slightly more subtle: it involves taking derivatives with respect to the
order of \emph{fractional derivation}.

We can also compute products and sums over primes
which involve Dirichlet characters, as long as their conductor is small,
as well as such products and sums where the primes are restricted to
certain congruence classes:

\begin{exercise} Compute to 100 decimal digits
  $$\prod_{p\equiv1\pmod{4}}(1-1/p^2)\quad\text{and}\quad\prod_{p\equiv1\pmod4}(1+1/p^2)$$
  by using products of $\z(ns)$ and of $L(\chi_{-4},ns)$ as above, where
  as usual $\chi_{-4}$ is the character $\lgs{-4}{n}$.
\end{exercise}

\subsection{Summation of Alternating Series}

This is due to F.~Rodriguez--Villegas, D.~Zagier, and the author \cite{Coh-Vil-Zag}.

We have seen above the use of the Euler--MacLaurin summation formula to sum
quite general types of series. If the series is \emph{alternating} (the terms
alternate in sign), the method cannot be used as is, but it is trivial to
modify it: simply write
$$\sum_{n\ge1}(-1)^nf(n)=\sum_{n\ge1}f(2n)-\sum_{n\ge1}f(2n-1)$$
and apply Euler--MacLaurin to each sum. One can even do better and avoid this
double computation, but this is not what I want to mention here.

A completely different method which is much simpler since it avoids completely
the computation of derivatives and Bernoulli numbers, due to the above authors,
is as follows. The idea is to express (if possible) $f(n)$ as a \emph{moment}
$$f(n)=\int_0^1 x^nw(x)\,dx$$
for some \emph{weight function} $w(x)$. Then it is clear that
$$S=\sum_{n\ge0}(-1)^nf(n)=\int_0^1\dfrac{1}{1+x}w(x)\,dx\;.$$
Assume that $P_n(X)$ is a polynomial of degree $n$ such that $P_n(-1)\ne0$. 
Evidently 
$$\dfrac{P_n(X)-P_n(-1)}{X+1}=\sum_{k=0}^{n-1}c_{n,k}X^k$$
is still a polynomial (of degree $n-1$), and we note the trivial fact that
\begin{align*}S&=\dfrac{1}{P_n(-1)}\int_0^1\dfrac{P_n(-1)}{1+x}w(x)\,dx\\
&=\dfrac{1}{P_n(-1)}\left(\int_0^1\dfrac{P_n(-1)-P_n(x)}{1+x}w(x)\,dx
+\int_0^1\dfrac{P_n(x)}{1+x}w(x)\,dx\right)\\
&=\dfrac{1}{P_n(-1)}\sum_{k=0}^{n-1}c_{n,k}f(k)+R_n\;,\end{align*}
with
$$|R_n|\le\dfrac{M_n}{|P_n(-1)|}\int_0^1\dfrac{1}{1+x}w(x)\,dx
=\dfrac{M_n}{|P_n(-1)|}S\;,$$
and where $M_n=\sup_{x\in[0,1]}|P_n(x)|$.
Thus if we can manage to have $M_n/|P_n(-1)|$ small, we obtain a good
approximation to $S$.

It is a classical result that the best choice for $P_n$ are the shifted
Chebychev polynomials defined by $P_n(\sin^2(t))=\cos(2nt)$, but in any
case we can use these polynomials and ignore that they are the best.

\medskip

This leads to an incredibly simple algorithm which we write explicitly:

\medskip

$d\gets (3+\sqrt{8})^n$; $d\gets (d+1/d)/2$; $b\gets -1$; $c\gets -d$; $s\gets0$; For $k=0,\dotsc,n-1$ do: 

$c\gets b-c$; $s\gets s+c\cdot f(k)$; $b\gets(k+n)(k-n)b/((k+1/2)(k+1))$;

The result is $s/d$.

\medskip

The convergence is in $5.83^{-n}$. 

It is interesting to note that, even though this algorithm is designed to
work with functions $f$ of the form $f(n)=\int_0^1 x^nw(x)\,dx$ with
$w$ continuous and positive, it is in fact valid outside its proven region
of validity. For example:

\begin{exercise}
It is well-known that the Riemann zeta function $\z(s)$
can be extended analytically to the whole complex plane, and that we have
for instance $\z(-1)=-1/12$ and $\z(-2)=0$. Apply the above algorithm to the
\emph{alternating} zeta function
$$\beta(s)=\sum_{n\ge1}(-1)^{n-1}\dfrac{1}{n^s}=\left(1-\dfrac{1}{2^{s-1}}\right)\z(s)$$
(incidentally, prove this identity), and by using the above algorithm, show
the nonconvergent ``identities''
$$1-2+3-4+\cdots=1/4\text{\quad and\quad}1-2^2+3^2-4^2+\cdots=0\;.$$
\end{exercise}

\begin{exercise} (B.~Allombert.) Let $\chi$ be a periodic arithmetic function
of period $m$, say, and assume that $\sum_{0\le j<m}\chi(j)=0$ (for instance
$\chi(j)=(-1)^j$ with $m=2$).
\begin{enumerate}\item Using the same polynomials $P_n$ as above, write
a similar algorithm for computing $\sum_{n\ge0}\chi(n)f(n)$, and estimate
its rate of convergence.
\item Using this, compute to 100 decimals
$L(\chi_{-3},k)=1-1/2^k+1/4^k-1/5^k+\cdots$ for $k=1$, $2$, and $3$,
and recognize the exact value for $k=1$ and $k=3$.\end{enumerate}\end{exercise}

\subsection{Numerical Differentiation}

The problem is as follows: given a function $f$, say defined and $C^\infty$
on a real interval, compute $f'(x_0)$ for a given value of $x_0$. To be able
to analyze the problem, we will assume that $f'(x_0)$ is not too close to $0$,
and that we want to compute it to a given \emph{relative accuracy}, which
is what is usually required in numerical analysis.

The na\"\i ve, although reasonable, approach, is to choose a small $h>0$ and
compute $(f(x_0+h)-f(x_0))/h$. However, it is clear that (using the same
number of function evaluations) the formula $(f(x_0+h)-f(x_0-h))/(2h)$
will be better. Let us analyze this in detail. For simplicity we will
assume that all the derivatives of $f$ around $x_0$ that we consider are
neither too small nor too large in absolute value. It is easy to modify the
analysis to treat the general case.

Assume $f$ computed to a relative accuracy of $\eps$, in other words that
we know values $\tilde{f}(x)$ such that
$\tilde{f}(x)(1-\eps)<f(x)<\tilde{f}(x)(1+\eps)$
(the inequalities being reversed if $f(x)<0$). The absolute error
in computing $(f(x_0+h)-f(x_0-h))/(2h)$ is thus essentially equal to 
$\eps |f(x_0)|/h$. On the other hand, by Taylor's theorem we have
$(f(x_0+h)-f(x_0-h))/(2h)=f'(x_0)+(h^2/6)f'''(x)$ for some $x$ close to $x_0$,
so the absolute error made in computing $f'(x_0)$ as 
$(f(x_0+h)-f(x_0-h))/(2h)$ is close to $\eps |f(x_0)|/h+(h^2/6)|f'''(x_0)|$.
For a given value of $\eps$ (i.e., the accuracy to which we compute $f$)
the optimal value of $h$ is $(3\eps |f(x_0)/f'''(x_0)|)^{1/3}$ for an
absolute error of $(1/2)(3\eps |f(x_0)f'''(x_0)|)^{2/3}$ hence a relative
error of $(3\eps |f(x_0)f'''(x_0)|)^{2/3}/(2|f'(x_0)|)$. 

Since we have assumed that the derivatives have reasonable size,
the relative error is roughly $C\eps^{2/3}$,
so if we want this error to be less than $\eta$, say, we need $\eps$
of the order of $\eta^{3/2}$, and $h$ will be of the order of $\eta^{1/2}$.

\smallskip

Note that this result is not completely intuitive. For instance,
assume that we want to compute derivatives to $38$ decimal digits.
With our assumptions, we choose $h$ around $10^{-19}$, and perform
the computations with $57$ decimals of relative accuracy. If for some
reason or other we are limited to $38$ decimals in the computation of $f$,
the ``intuitive'' way would be also to choose $h=10^{-19}$, and the above
analysis shows that we would obtain only approximately $19$ decimals.
On the other hand, if we chose $h=10^{-13}$ for instance, close to
$10^{-38/3}$, we would obtain $25$ decimals.

\smallskip

There are of course many other formulas for computing $f'(x_0)$, or for
computing higher derivatives, which can all easily be analyzed as above.
For instance (exercise), one can look for approximations to $f'(x_0)$ of the
form $S=(\sum_{1\le i\le 3}\la_if(x_0+h/a_i))/h$, for any nonzero and pairwise
distinct $a_i$, and we find that this is possible as soon as
$\sum_{1\le i\le 3}a_i=0$ (for instance, if $(a_1,a_2,a_3)=(-3,1,2)$
we have $(\la_1,\la_2,\la_3)=(-27,-5,32)/20$), and the absolute error is then
of the form $C_1/h+C2h^3$, so the same analysis shows that we should
work with accuracy $\eps^{4/3}$ instead of $\eps^{3/2}$. Even though we
have $3/2$ times more evaluations of $f$, we require less accuracy:
for instance, if $f$ requires time $O(D^a)$ to be computed to $D$ decimals,
as soon as $(3/2)\cdot((4/3)D)^a<((3/2)D)^a$, i.e., $3/2<(9/8)^a$, hence
$a\ge3.45$, this new method will be faster.

Perhaps the best known method with more function evaluations is the
approximation 
$$f'(x_0)\approx(f(x-2h)-8f(x-h)+8f(x+h)-f(x+2h))/(12h)\;,$$
which requires accuracy $\eps^{5/4}$, and since this requires $4$ evaluations
of $f$, this is faster than the first method as soon as
$2\cdot(5/4)^a<(3/2)^a$, in other words $a>3.81$, and faster than the
second method as soon as $(4/3)\cdot(5/4)^a<(4/3)^a$, in other words
$a>4.46$. To summarize, use the first method if $a<3.45$, the second method
if $3.45\le a<4.46$, and the third if $a>4.46$. Of course this game can
be continued at will, but there is not much point in doing so. In practice
the first method is sufficient.

\medskip

\subsection{Double Exponential Numerical Integration}

A remarkable although little-known technique invented around 1970 deals with
\emph{numerical integration} (the numerical computation of a definite
integral $\int_a^b f(t)\,dt$, where $a$ and $b$ are allowed to be $\pm\infty$).
In usual numerical analysis courses one teaches very elementary techniques
such as the trapezoidal rule, Simpson's rule, or more sophisticated methods
such as Romberg or Gaussian integration. These methods apply to very general
classes of functions $f(t)$, but are unable to compute more than a few
decimal digits of the result, except for Gaussian integration which we will
mention below.

However, in most mathematical (as opposed for instance to physical) contexts,
the function $f(t)$ is \emph{extremely regular}, typically holomorphic or
meromorphic, at least in some domain of the complex plane. It was observed
in the late 1960's by H.~Takahashi and M.~Mori \cite{Tak-Mor} that
this property can be used to obtain a \emph{very simple} and 
\emph{incredibly accurate} method to compute definite integrals of such
functions. It is now instantaneous to compute $100$ decimal digits, and takes
only a few seconds to compute $500$ decimal digits, say.

In view of its importance it is essential to have some knowledge of this
method. It can of course be applied in a wide variety of contexts, but note
also that in his thesis \cite{Mol}, P.~Molin has applied it specifically to
the \emph{rigorous} and \emph{practical} computation of values of
$L$-functions, which brings us back to our main theme.

\medskip

There are two basic ideas behind this method. The first is in fact a theorem,
which I state in a vague form: If $F$ is a holomorphic function which tends to
$0$ ``sufficiently fast'' when $x\to\pm\infty$, $x$ real, then the most
efficient method to compute $\int_{\R}F(t)\,dt$ is indeed the trapezoidal
rule. Note that this is a \emph{theorem}, not so difficult but a little
surprising nonetheless. The definition of ``sufficiently fast'' can be
made precise. In practice, it means at least like $e^{-ax^2}$ ($e^{-a|x|}$
is not fast enough), but it can be shown that the best results are obtained
with functions tending to $0$ \emph{doubly exponentially fast} such as
$\exp(-\exp(a|x|))$. Note that it would be (very slightly) worse to choose
functions tending to $0$ even faster.

To be more precise, we have an estimate coming for instance from the
\emph{Euler--MacLaurin summation formula}:
$$\int_{-\infty}^{\infty}F(t)\,dt=h\sum_{n=-N}^NF(nh)+R_N(h)\;,$$
and under suitable holomorphy conditions on $F$, if we choose $h=a\log(N)/N$
for some constant $a$ close to $1$, the remainder term $R_N(h)$ will
satisfy $R_n(h)=O(e^{-bN/\log(N)})$ for some other (reasonable) constant $b$,
showing exponential convergence of the method.

\medskip

The second and of course crucial idea of the method is as follows: evidently
not all functions are doubly-exponentially tending to $0$ at $\pm\infty$,
and definite integrals are not all from $-\infty$ to $+\infty$. But it is
possible to reduce to this case by using clever \emph{changes of variable}
(the essential condition of holomorphy must of course be preserved).

Let us consider the simplest example, but others that we give below are
variations on the same idea. Assume that we want to compute
$$I=\int_{-1}^1f(x)\,dx\;.$$
We make the ``magical'' change of variable $x=\phi(t)=\tanh(\sinh(t))$, so that
if we set $F(t)=f(\phi(t))$ we have
$$I=\int_{-\infty}^{\infty}F(t)\phi'(t)\,dt\;.$$
Because of the elementary properties of the hyperbolic sine and tangent,
we have gained two things at once: first the integral from $-1$ to $1$ is
now from $-\infty$ to $\infty$, but most importantly the function
$\phi'(t)$ is easily seen to tend to $0$ doubly exponentially. We thus
obtain an \emph{exponentially good approximation}
$$\int_{-1}^1f(x)\,dx=h\sum_{n=-N}^Nf(\phi(nh))\phi'(nh)+R_N(h)\;.$$
To give an idea of the method, if one takes $h=1/200$ and $N=500$, hence
only $1000$ evaluations of the function $f$, one can compute $I$ to several
hundred decimal places!

\medskip

Before continuing, I would like to comment that in this theory many results
are not completely rigorous: the method works very well, but the proof that
it does is sometimes missing. Thus I cannot resist giving a \emph{proven and
precise} theorem due to P.~Molin (which is of course just an example).
We keep the above notation $\phi(t)=\tanh(\sinh(t))$, and note that
$\phi'(t)=\cosh(t)/\cosh^2(\sinh(t))$.

\begin{theorem}[P.~Molin] Let $f$ be holomorphic on the disc $D=D(0,2)$
centered at the origin and of radius $2$. Then for all $N\ge1$, if we choose
$h=\log(5N)/N$ we have
$$\int_{-1}^1f(x)\,dx=h\sum_{n=-N}^Nf(\phi(nh))\phi'(nh)+R_N\;,$$
where
$$|R_N|\le \left(e^4\sup_{D}|f|\right)\exp(-5N/\log(5N))\;.$$
\end{theorem}

\medskip

Coming back to the general situation, I briefly comment on the computation
of general definite integrals $\int_a^b f(t)\,dt$.

\begin{enumerate}\item If $a$ and $b$ are finite, we can reduce to $[-1,1]$
by affine changes of variable.
\item If $a$ (or $b$) is finite and the function has an algebraic singularity
at $a$ (or $b$), we remove the singularity by a polynomial change of variable.
\item If $a=0$ (say) and $b=\infty$, then if $f$ does \emph{not} tend to $0$
exponentially fast (for instance $f(x)\sim 1/x^k$), we use 
$x=\phi(t)=\exp(\sinh(t))$.
\item If $a=0$ (say) and $b=\infty$ and if $f$ does tend to $0$
exponentially fast (for instance $f(x)\sim e^{-ax}$ or $f(x)\sim e^{-ax^2}$), 
we use $x=\phi(t)=\exp(t-\exp(-t))$.
\item If $a=-\infty$ and $b=\infty$, use $x=\phi(t)=\sinh(\sinh(t))$ if
$f$ does not tend to $0$ exponentially fast, and $x=\phi(t)=\sinh(t)$
otherwise.
\end{enumerate}

The problem of \emph{oscillating} integrals such as
$\int_0^\infty f(x)\sin(x)\,dx$ is more subtle, but there does exist
similar methods when, as here, the oscillations are completely under control.

\begin{remark} The theorems are valid when the function is holomorphic in
a sufficiently large region compared to the path of integration. If the
function is only \emph{meromorphic}, with known poles, the direct application
of the formulas may give totally wrong answers. However, if we take into
account the poles, we can recover perfect agreement. Example of bad behavior:
$f(t)=1/(1+t^2)$ (poles $\pm i$). Integrating on the intervals
$[0,\infty]$, $[0,1000]$, or even $[-\infty,\infty]$, which involve different
changes of variables, give perfect results (the latter being somewhat 
surprising). On the other hand, integrating on $[-1000,1000]$ gives
a totally wrong answer because the poles are ``too close'', but it is easy
to take them into account if desired.
\end{remark}

Apart from the above pathological behavior, let us give a couple of examples
where we must slightly modify the direct use of doubly-exponential 
integration techniques.

\medskip

\newcommand{\hh}[1]{\^{}{#1}}
$\bullet$ Assume for instance that we want to compute
$$J=\int_1^\infty\left(\dfrac{1+e^{-x}}{x}\right)^2\,dx\;,$$
and that we use the built-in function {\tt intnum} of {\tt Pari/GP} for
doing so. The function tends to $0$ slowly at infinity, so we should compute
it using the {\tt GP} syntax {\tt oo} to represent $\infty$, so we write
{\tt f(x)=((1+exp(-x))/x)\hh{2};}, then {\tt intnum(x=1,oo,f(x))}.
This will give some sort of error, because the software will try to
evaluate $\exp(-x)$ for large values of $x$, which it cannot do since there
is exponent underflow. To compute the result, we need to split it into
its slow part and fast part: when a function tends exponentially fast to
$0$ like $exp(-ax)$, $\infty$ is represented as {\tt [oo,a]}, so we write
$J=J_1+J_2$, with $J_1$ and $J_2$ computed by:

{\tt J1=intnum(x=1,[oo,1],(exp(-2*x)+2*exp(-x))/x\hh{2});} and

{\tt J2=intnum(x=1,oo,1/x\hh{2});}
(which of course is equal to $1$), giving
$$J=1.3345252753723345485962398139190637\cdots\;.$$
Note that we could have tried to ``cheat'' and written directly

{\tt intnum(x=1,[oo,1],f(x))}, but the answer would
be wrong, because the software would have assumed that $f(x)$ tends to $0$
exponentially fast, which is not the case.

\medskip

$\bullet$ A second situation where we must be careful is when we have
``apparent singularities'' which are not real singularities.
Consider the function $f(x)=(\exp(x)-1-x)/x^2$. It has an apparent singularity
at $x=0$ but in fact it is completely regular. If you ask
{\tt J=intnum(x=0,1,f(x))}, you will get a result which is reasonably correct,
but never more than $19$ decimals, say. The reason is \emph{not} due to
a defect in the numerical integration routine, but more in the computation
of $f(x)$: if you simply write {\tt f(x)=(exp(x)-1-x)/x\hh{2};}, the results
will be bad for $x$ close to $0$.

Assuming that you want $38$ decimals, say, the solution is to write

\noindent
{\tt f(x)=if(x<10\hh{(-10)},1/2+x/6+x\hh{2}/24+x\hh{3}/120,(exp(x)-1-x)/x\hh{2});}
and now we obtain the value of our integral as
$$J=0.59962032299535865949972137289656934022\cdots$$

\subsection{The Use of Abel--Plana for Definite Summation}

We finish this course by describing an identity, which is first quite amusing
and second can be used efficiently for definite summation. Consider for
instance the following theorem:

\begin{theorem} Define by convention $\sin(n/10)/n$ as equal to its limit
$1/10$ when $n=0$, and define $\sum'_{n\ge0}f(n)$ as
$f(0)/2+\sum_{n\ge1}f(n)$. We have
$$\sump_{n\ge0}\left(\dfrac{\sin(n/10)}{n}\right)^k=\int_0^\infty\left(\dfrac{\sin(x/10)}{x}\right)^k$$
for $1\le k\le 62$, but not for $k\ge63$.
\end{theorem}

If you do not like all these conventions, replace the left-hand side by
$$\dfrac{1}{2\cdot 10^k}+\sum_{n\ge1}\left(\dfrac{\sin(n/10)}{n}\right)^k\;.$$

It is clear that something is going on: it is the Abel--Plana formula.
There are several forms of this formula, here is one of them:

\begin{theorem}[Abel--Plana] Assume that $f$ is an entire function
and that $f(z)=o(\exp(2\pi|\Im(z)|))$ as $|\Im(z)|\to\infty$ uniformly in 
vertical strips of bounded width, and a number of less important additional 
conditions which we omit. Then
\begin{align*}\sum_{m\ge1}f(m)&=\int_0^\infty f(t)\,dt-\dfrac{f(0)}{2}+i\int_0^\infty\dfrac{f(it)-f(-it)}{e^{2\pi t}-1}\,dt\\
&=\int_{1/2}^\infty f(t)\,dt-i\int_0^\infty\dfrac{f(1/2+it)-f(1/2-it)}{e^{2\pi t}+1}\,dt\;.\end{align*}

In particular, if the function $f$ is \emph{even}, we have
$$\dfrac{f(0)}{2}+\sum_{m\ge1}f(m)=\int_0^\infty f(t)\,dt\;.$$
\end{theorem}

Since we have seen above that using doubly-exponential techniques it is easy 
to compute numerically a definite \emph{integral}, the Abel--Plana formula can
be used to compute numerically a \emph{sum}. Note that in the first version
of the formula there is an apparent singularity (but which is not a
singularity) at $t=0$, and the second version avoids this problem.

In practice, this summation method is very competitive with other methods
if we use the doubly-exponential method to compute $\int_0^\infty f(t)\,dt$,
but most importantly if we use a variant of \emph{Gaussian integration} to
compute the complex integrals, since the nodes and weights for the
function $t/(e^{2\pi t}-1)$ can be computed once and for all by using
continued fractions, see Section \ref{sec:gauss}.

\section{The Use of Continued Fractions}

\subsection{Introduction}

The last idea that I would like to mention and that is applicable in quite
different situations is the use of continued fractions. Recall that a
continued fraction is an expression of the form
$$a_0+\dfrac{b_0}{a_1+\dfrac{b_1}{a_2+\dfrac{b_2}{a_3+\ddots}}}\;.$$
The problem of \emph{convergence} of such expressions (when they are unlimited)
is difficult and will not be considered here. We refer to any good textbook
on the elementary properties of continued fractions. In particular, recall that
if we denote by $p_n/q_n$ the $n$th \emph{partial quotient} (obtained by
stopping at $b_{n-1}/a_n$) then both $p_n$ and $q_n$ satisfy the same recursion
$u_n=a_nu_{n-1}+b_{n-1}u_{n-2}$.

We will mainly consider continued fractions representing \emph{functions}
as opposed to simply numbers. Whatever the context, the interest of continued
fractions (in addition to the fact that they are easy to evaluate) is that
they give essentially the \emph{best possible} approximations, both for
real numbers (this is the standard theory of \emph{regular} continued
fractions, where $b_n=1$ and $a_n\in\Z_{\ge1}$ for $n\ge1$), and for
functions (this is the theory of \emph{Pad\'e approximants}).

\subsection{The Two Basic Algorithms}

The first algorithm that we need is the following: assume that we want
to expand a (formal) power series $S(z)$ (without loss of generality
such that $S(0)=1$) into a continued fraction:
$$S(z)=1+c(1)z+c(2)z^2+\cdots = 1+\dfrac{b(0)z}{1+\dfrac{b(1)z}{1+\dfrac{b(2)z}{1+\ddots}}}\;.$$
The following method, called the \emph{quotient-difference} (QD) algorithm
does what is required:

We define two arrays $e(j,k)$ for $j\ge0$ and $q(j,k)$ for $j\ge1$ by
$e(0,k)=0$, $q(1,k)=c(k+2)/c(k+1)$ for $k\ge0$, and by induction for $j\ge1$
and $k\ge0$:
\begin{align*}e(j,k)&=e(j-1,k+1)+q(j,k+1)-q(j,k)\;,\\
q(j+1,k)&=q(j,k+1)e(j,k+1)/e(j,k)\;.\end{align*}

Then $b(0)=c(1)$ and $b(2n-1)=-q(n,0)$ and $b(2n)=-e(n,0)$ for $n\ge1$.

Three essential implementation remarks: first keeping the whole arrays is
costly, it is sufficient to keep the latest vectors of $e$ and $q$. Second,
even if the $c(n)$ are rational numbers it is essential to do the computation
with floating point approximations to avoid coefficient explosion. The
algorithm can become unstable, but this is corrected by increasing the working
accuracy. Third, it is of course possible that some division by $0$ occurs,
and this is in fact quite frequent. There are several ways to overcome this,
probably the simplest being to multiply or divide the power series by
something like $1-z/\pi$.

\medskip

The second algorithm is needed to \emph{evaluate} the continued fraction for
a given value of $z$. It is well-known that this can be done from bottom to
top (start at $b(n)z/1$, then $b(n-1)/(1+b(n)z/1)$, etc.), or from top
to bottom (start at $(p(-1),q(-1))=(1,0)$, $(p(0),q(0))=(1,1)$, and use
the recursion). It is in general better to evaluate from bottom to top, but
before doing this we can considerably improve on the speed by using an identity
due to Euler:
$$1+\dfrac{b(0)z}{1+\dfrac{b(1)z}{1+\dfrac{b(2)z}{1+\ddots}}}
=1+\dfrac{B(0)Z}{Z+A(1)+\dfrac{B(1)}{Z+A(2)+\dfrac{B(2)}{Z+A(3)+\ddots}}}\;,$$
where $Z=1/z$, $A(1)=b(1)$, $A(n)=b(2n-2)+b(2n-1)$ for $n\ge2$,
$B(0)=b(0)$, $B(n)=-b(2n)b(2n-1)$ for $n\ge1$.
The reason for which this is much faster is that we replace
$n$ multiplications ($b(j)*z$) plus $n$ divisions by
$1$ multiplication plus approximately $1+n/2$ divisions, counting as usual
additions as negligible.

This is still not the end of the story since we can ``compress'' any
continued fraction by taking, for instance, two steps at once instead
of one, which reduces the cost . In any case this leads to a very efficient method for evaluating
continued fractions.

\subsection{Using Continued Fractions for Inverse Mellin Transforms}

We have mentioned above that one can use asymptotic expansions to compute
the incomplete gamma function $\G(s,x)$ when $x$ is large. But this method
cannot give us great accuracy since we must stop the asymptotic expansion
at its smallest term. We can of course always use the power series expansion,
which has infinite radius of convergence, but when $x$ is large this is not
very efficient (remember the example of computing $e^{-x}$).

In the case of $\G(s,x)$, continued fractions save the day: indeed, one can
prove that
$$\G(s,x)=\dfrac{x^se^{-x}}{x+1-s-\dfrac{1(1-s)}{x+3-s-\dfrac{2(2-s)}{x+5-s-\ddots}}}\;,$$
with precisely known speed of convergence. This formula is the best method
for computing $\G(s,x)$ when $x$ is large (say $x>50$), and can give arbitrary
accuracy.

However here we were in luck: we had an ``explicit'' continued fraction
representing the function that we wanted to compute. Evidently, in general
this will not be the case.

It is a remarkable idea of T.~Dokchitser \cite{Dok} that it does not really
matter if the continued fraction is not explicit, at least in the context of
computing $L$-functions, for instance for inverse Mellin transforms. Simply do
the following:

\begin{enumerate}\item First compute sufficiently many terms of the asymptotic
expansion of the function to be computed. This is very easy because our
functions all satisfy a \emph{linear differential equation} with polynomial
coefficients, which gives a \emph{recursion} on the coefficients of the
asymptotic expansion.
\item Using the quotient-difference algorithm seen above, compute the
corresponding continued fraction, and write it in the form due to Euler
to evaluate it as efficiently as possible.
\item Compute the value of the function at all desired arguments by evaluating
the Euler continued fraction.\end{enumerate}

The first two steps are completely automatic and rigorous. The whole problem
lies in the third step, the evaluation of the continued fraction. In the case
of the incomplete gamma function, we had a theorem giving us the speed of
convergence. In the case of inverse Mellin transforms, not only do we not
have such a theorem, but we do not even know how to prove that the continued
fraction converges! However experimentation shows that not only does the
continued fraction converge, but rather fast, in fact at a similar speed to
that of the incomplete gamma function.

Even though this step is completely heuristic, since its introduction by
T.~Dokchitser it is used in all packages computing $L$-functions since it is
so useful. It would of course be nice to have a \emph{proof} of its validity,
but for now this seems completely out of reach, except for the simplest
examples where there are at most two gamma factors (for instance the problem
is completely open for the inverse Mellin transform of $\G(s)^3$).

\subsection{Using Continued Fractions for Gaussian Integration and Summation}\label{sec:gauss}

We have seen above the doubly-exponential method for numerical integration,
which is robust and quite generally applicable. However, an extremely classical
method is \emph{Gaussian integration}: it is orders of magnitude faster,
but note the crucial fact that it is much less robust, in that it works
much less frequently.

The setting of Gaussian
integration is the following: we have a measure $d\mu$ on a (compact or
infinite) interval $[a,b]$; you can of course think of $d\mu$ as $K(x)dx$ for
some fixed function $K(x)$. We want to compute $\int_a^bf(x)d\mu$ by
means of \emph{nodes} and \emph{weights}, i.e., for a given $n$ compute $x_i$
and $w_i$ for $1\le i\le n$ such that $\sum_{1\le i\le n}w_if(x_i)$
approximates as closely as possible the exact value of the integral.

Note that \emph{classical} Gaussian integration such as Gauss--Legendre
integration (integration of a continuous function on a compact interval)
is easy to perform because one can easily compute explicitly the necessary
nodes and weights using standard \emph{orthogonal polynomials}. What I want to
stress here is that \emph{general} Gaussian integration can be performed very
simply using continued fractions, as follows.

In general the measure $d\mu$ is (or can be) given through its \emph{moments}
$M_k=\int_a^bx^kd\mu$. The remarkably simple algorithm to compute
the $x_i$ and $w_i$ using continued fractions is as follows:

\begin{enumerate}
\item Set $\Phi(z)=\sum_{k\ge0}M_kz^{k+1}$, and using the
quotient-difference algorithm compute $c(m)$ such that
$\Phi(z)=c(0)z/(1+c(1)z/(1+c(2)z/(1+\cdots)))$ (see the
remark made above in case the algorithm has a division by $0$; it may also
happen that the odd or even moments vanish, so that the continued fraction
is only in powers of $z^2$, but this is also easily dealt with).
\item For any $m$, denote as usual by $p_m(z)/q_m(z)$ the $m$th
convergent obtained by stopping the continued fraction at
$c(m)z/1$, and denote by $N_n(z)$ the reciprocal polynomial of
$p_{2n-1}(z)/z$ (which has degree $n-1$) and by $D_n(z)$ the
reciprocal polynomial of $q_{2n-1}$ (which has degree $n$).
\item The $x_i$ are the $n$ roots of $D_n$ (which are all simple
and in the interval $]a,b[$), and the $w_i$ are given by the formula
$w_i=N_n(x_i)/D'_n(x_i)$.
\end{enumerate}

By construction, this Gaussian integration method will work when the
function $f(x)$ to be integrated is well approximated by polynomials,
but otherwise will fail miserably, and this is why we say that the method
is much less ``robust'' than doubly-exponential integration.

\medskip

The fact that Gaussian ``integration'' can also be used very efficiently for
numerical \emph{summation} was discovered quite recently by H.~Monien. We
explain the simplest case. Consider the measure on $]0,1]$ given by
$d\mu=\sum_{n\ge1}\delta_{1/n}/n^2$, where $\delta_x$ is the Dirac measure
centered at $x$. Thus by definition
$\int_0^1 f(x)d\mu=\sum_{n\ge1}f(1/n)/n^2$. Let us apply the recipe
given above: the $k$th moment $M_k$ is given by
$M_k=\sum_{n\ge1}(1/n)^k/n^2=\z(k+2)$, so that
$\Phi(z)=\sum_{k\ge1}\z(k+1)z^k$. Note that this is closely related to the
digamma function $\psi(z)$, but we do not need this. Applying the
quotient-difference algorithm, we write
$\Phi(z)=c(0)z/(1+c(1)z/(1+\cdots))$, and compute the $x_i$ and $w_i$ as
explained above. We will then have that $\sum_iw_if(x_i)$ is a very good
approximation to $\sum_{n\ge1}f(1/n)/n^2$, or equivalently (changing the
definition of $f$) that $\sum_iw_if(y_i)$ is a very good approximation
to $\sum_{n\ge1}f(n)$, with $y_i=1/x_i$.

To take essentially the simplest example, stopping the continued fraction
after two terms we find that
$y_1=1.0228086266\cdots$, $w_1=1.15343168\cdots$,
$y_2=4.371082834\cdots$, and $w_2=10.3627543\cdots$,
and (by definition) we have $\sum_{1\le i\le 2}w_if(y_i)=\sum_{n\ge1}f(n)$
for $f(n)=1/n^k$ with $k=2$, $3$, $4$, and $5$.

\section{{\tt Pari/GP} Commands}

In this section, we give some of the {\tt Pari/GP} commands related to the
subjects studied in this course, together with examples. Unless mentioned
otherwise, the commands assume that the current default accuracy is the
default, i.e., $38$ decimal digits.

{\tt zeta(s)}: Riemann zeta function at $s$.

\begin{verbatim}
? zeta(3)
% = 1.2020569031595942853997381615114499908
? zeta(1/2+14*I)
% = 0.022241142609993589246213199203968626387
  - 0.10325812326645005790236309555257383451*I 
\end{verbatim}

{\tt lfuncreate(obj)}: create $L$-function attached to mathematical object
{\tt obj}.

{\tt lfun(pol,s)}: Dedekind zeta function of the number field $K$ defined by
{\tt pol} at $s$. Identical to {\tt L=lfuncreate(pol); lfun(L,s)}.

\begin{verbatim}
? L = lfuncreate(x^3-x-1); lfunan(L,10)
% = [1, 0, 0, 0, 1, 0, 1, 1, 0, 0]
? lfun(L,1)
% = 0.36840932071582682111186846662888526986*x^-1 + O(x^0)
? lfun(L,2)
% = 1.1100010060250153929372222560595385375
\end{verbatim}

{\tt lfunlambda(pol,s)}: same, but for the completed function $\Lambda_K(s)$,
identical to {\tt lfunlambda(L,s)} where {\tt L} is as above.

\begin{verbatim}
? lfunlambda(L,2)
% = 0.41169121016707136240079852448689476625
\end{verbatim}

{\tt lfun(D,s)}: $L$-function of quadratic character $(D/.)$ at $s$.

\noindent
Identical to {\tt L=lfuncreate(D); lfun(L,s)}.

\begin{verbatim}
? lfun(-23,-2)
% = -48.000000000000000000000000000000000000
? lfun(5,-1)
% = -0.4000000000000000000000000000000000000
\end{verbatim}

{\tt L1=lfuncreate(pol); L2=lfuncreate(1); L=lfundiv(L1,L2)}: $L$ function
attached to $\z_K(s)/\z(s)$.

\begin{verbatim}
? L1 = lfuncreate(x^3-x-1); L2 = lfuncreate(1);
? L = lfundiv(L1,L2); lfunan(L,14)
% = [1, -1, -1, 0, 0, 1, 0, 1, 0, 0, 0, 0, -1, 0]
\end{verbatim}

{\tt lfunetaquo($[m_1,r_1;m_2,r_2]$)}: $L$-function of eta product
$\eta(m_1\tau)^{r_1}\eta(m_2\tau)^{r_2}$, for instance with
{\tt [1,1;23,1]} or {\tt [1,2;11,2]}.

\begin{verbatim}
? L1 = lfunetaquo([1,1;23,1]); lfunan(L1,14)
% = [1, -1, -1, 0, 0, 1, 0, 1, 0, 0, 0, 0, -1, 0]
? L2 = lfunetaquo([1,2;11,2]); lfunan(L2,14)
% = [1, -2, -1, 2, 1, 2, -2, 0, -2, -2, 1, -2, 4, 4]
\end{verbatim}

{\tt lfuncreate(ellinit(e))}: $L$-function of elliptic curve $e$, for
instance with $e=[0,-1,1,-10,-20]$.

\begin{verbatim}
? e = ellinit([0,-1,1,-10,-20]);
? L = lfuncreate(e); lfunan(L,14)
% = [1, -2, -1, 2, 1, 2, -2, 0, -2, -2, 1, -2, 4, 4]
\end{verbatim}

{\tt ellap(e,p)}: compute $a(p)$ for an elliptic curve $e$.

\begin{verbatim}
? ellap(e,nextprime(10^42))
% = -1294088699019102994696
\end{verbatim}

{\tt eta(q+O(q\^{}B))\^{}m}: compute the $m$th power of $\eta$ to $B$ terms.

\begin{verbatim}
? eta(q+O(q^5))^26
% = 1 - 26*q + 299*q^2 - 1950*q^3 + 7475*q^4 + O(q^5)
\end{verbatim}

{\tt D=mfDelta(); mfcoefs(D,B)}: compute $B+1$ terms of the Fourier expansion
of $\Delta$.

\begin{verbatim}
? D = mfDelta(); mfcoefs(D,7)
% = [0, 1, -24, 252, -1472, 4830, -6048, -16744]
\end{verbatim}

{\tt ramanujantau(n)}: compute Ramanujan's tau function $\tau(n)$ using
the trace formula.

\begin{verbatim}
? ramanujantau(nextprime(10^7))
% = 110949191154874445294730241687634133420
\end{verbatim}

{\tt qfbhclassno(n)}: Hurwitz class number $H(n)$.

\begin{verbatim}
? vector(13,n,qfbhclassno(n-1))
% = [-1/12, 0, 0, 1/3, 1/2, 0, 0, 1, 1, 0, 0, 1, 4/3]
\end{verbatim}

{\tt qfbsolve(Q,n)}: solve $Q(x,y)=n$ for a binary quadratic form $Q$
(contains in particular Cornacchia's algorithm).

\begin{verbatim}
? Q = Qfb(1,0,1); p = 10^16+61; qfbsolve(Q,p)
% = [86561206, 50071525]
\end{verbatim}

{\tt gamma(s)}: gamma function at $s$.

\begin{verbatim}
? gamma(1/4)*gamma(3/4)-Pi*sqrt(2)
% = 2.350988701644575016 E-38
\end{verbatim}

{\tt incgam(x,s)}: incomplete gamma function $\G(s,x)$.

\begin{verbatim}
? incgam(1,5/2)
% = 0.082084998623898795169528674467159807838
\end{verbatim} 

{\tt G=gammamellininvinit(A)}: initialize data for computing inverse Mellin
transforms of $\prod_{1\le i\le d}\G_{\R}(s+a_i)$, with $A=[a_1,\ldots,a_d]$.

{\tt gammamellininv(G,t)}: inverse Mellin transform at $t$ of $A$, with
$G$ initialized as above.

\begin{verbatim}
? G = gammamellininvinit([0,0]); gammamellininv(G,2)
% = 4.8848219774465217355974384319702281090 E-6
\end{verbatim}

{\tt K(nu,x)}: $K_{\nu}(x)$, $K$-Bessel function of (complex) index $\nu$ at
$x$.

\begin{verbatim}
? 4*besselk(0,4*Pi)
% = 4.8848219774465217355974384319702281090 E-6
\end{verbatim}

{\tt sumnum(n=a,f(n))}: numerical summation of $\sum_{n\ge a}f(n)$ using
discrete Euler--MacLaurin.

\begin{verbatim}
? sumnum(n=1,1/(n^2+n^(4/3)))
% = 0.95586324768586066988568837766973815238
\end{verbatim}

{\tt sumnumap(n=a,f(n))}: numerical summation of $\sum_{n\ge a}f(n)$ using
Abel--Plana.

{\tt sumnummonien(n=a,f(n))}: numerical summation using Monien's Gaussian
summation method,

(there also exists {\tt sumnumlagrange}, which can also be very useful).

{\tt limitnum(n->f(n))}: limit of $f(n)$ as $n\to\infty$ using a variant
of Zagier's method, assuming asymptotic expansion in integral powers of $1/n$
(also {\tt asympnum} to obtain more coefficients).

\begin{verbatim}
? limitnum(n->(1+1/n)^n)
% = 2.7182818284590452353602874713526624978
? asympnum(n->(1+1/n)^n*exp(-1))
% = [1, -1/2, 11/24, -7/16, 2447/5760, -959/2304,...]
\end{verbatim}

{\tt sumeulerrat(f(x))}: $\sum_{p\ge2}f(p)$, $p$ ranging over primes
(more general variant exists form $\sum_{p\ge a}f(p^s)$).

\begin{verbatim}
? sumeulerrat(1/(x^2+x))
% = 0.33022992626420324101509458808674476056
\end{verbatim}

{\tt prodeulerrat(f(x))}: $\prod_{p\ge2}f(p)$, $p$ ranging over primes, with
same variants.

\begin{verbatim}
? prodeulerrat((1-1/x)^2*(1+2/x))
% = 0.28674742843447873410789271278983844644
\end{verbatim}

{\tt sumalt(n=a,(-1)\^{}n*f(n))}: $\sum_{n\ge a}(-1)^nf(n)$, assuming $f$
positive.

\begin{verbatim}
? sumalt(n=1,(-1)^n/(n^2+n))
% = -0.38629436111989061883446424291635313615
\end{verbatim}

{\tt f'(x)} (or {\tt deriv(f)(x)}): numerical derivative of $f$ at $x$.

\begin{verbatim}
? -zeta'(-2)
% = 0.030448457058393270780251530471154776647
? zeta(3)/(4*Pi^2)
% = 0.030448457058393270780251530471154776647
\end{verbatim}

{\tt intnum(x=a,b,f(x))}: numerical computation of $\int_a^b f(x)\,dx$ using
general doubly-exponential integration.

{\tt intnumgauss(x=a,b,f(x))}: numerical integration using Gaussian
integration.

\begin{verbatim}
? intnum(t=0,1,lngamma(t+1))
% = -0.081061466795327258219670263594382360139
\end{verbatim}

  For instance, for $500$ decimal digits, after the initial computation of
  nodes and weights in both cases ({\tt intnuminit(0,1)} and
  {\tt intnumgaussinit()}) this examples requires $2.5$ seconds by
  doubly-exponential integration but only $0.25$ seconds by Gaussian
  integration.
  
\section{Three Pari/GP Scripts}

\subsection{The Birch--Swinnerton-Dyer Example}

Here is a list of commands which implements the explicit BSD example given
in Section \ref{sec:BSD}, again assuming the default accuracy of $38$ decimal
digits.

\begin{verbatim}
? E = ellinit([1,-1,0,-79,289]); /* initialize */
? N = ellglobalred(E)[1] /* compute conductor */
% = 234446
? /* define the integral $f(x)$ */
? f(x) = intnum(t=1,[oo,x],exp(-x*t)*log(t)^2);
? /* check that f(100) is small enough for 38D */
? f(100)
% = 7.2... E-50 
? A = ellan(E,8000); /* compute 8000 coefficients */
? /* Note that $2\pi 8000/sqrt(N) > 100$ */
? S = sum(n=1,8000,A[n]*f(2*Pi*n/sqrt(N)))
% = 9.02... E-35 /* almost 0 */
? /* compute APPARENT order of vanishing of L(E,s) */
? ellanalyticrank(E)[1]
% = 4
\end{verbatim}

Note that for illustrative purposes we use the {\tt intnum} command to compute
$f(x)$, corresponding to the use of doubly-exponential integration, but in the
present case there are methods which are orders of magnitude faster.
The last command, which is almost immediate, implements these methods.

\subsection{The Beilinson--Bloch Example}

The code for the explicit Beilinson--Bloch example seen in Section
\ref{sec:BB} is simpler (I have used the integral representation of $g(u)$,
but of course I could have used the series expansion instead):

\begin{verbatim}
? e(u) =
{
  my(E = ellinit([0,u^2+1,0,u^2,0]));
  lfun(E,2)*ellglobalred(E)[1];
}
? g(u) =
{
  my(S);
  S = 2*Pi*intnum(t=0,1,asin(t)/(t*sqrt(1-(t/u)^2)));
  S+Pi^2*acosh(u);
}
? e(5)/g(5)
% = 8.0000000000000000000000000000000000000
? /* we obtain perfect accuracy */
? /* for example: */
? for(u = 2,18,print1(bestappr(e(u)/g(u),10^6)," "))
% = 1 2 4/11 8 32 8 4/3 8 32 64 8 96 256 48 16 16 192
\end{verbatim}

\subsection{The Mahler Measure Example}

\begin{verbatim}
? L=lfunetaquo([2,1;4,1;6,1;12,1]);
\\ Equivalently L=lfuncreate(ellinit([0,-1,0,-4,4]));
? lfun(L,3)
% = 0.95050371329356644983179739940014855951
? (Pi^2/36)*(Catalan*Pi+intnum(t=0,1,asin(t)*asin(1-t)/t))
% = 0.95050371329356644983179739940014855950
\end{verbatim}

\section{Appendix: Selected Results}

\subsection{The Gamma Function}

The Gamma function, denoted by $\G(s)$, can be defined in several different
ways. My favorite is the one I give in Section 9.6.2 of \cite{Coh4}, but for
simplicity I will recall the classical definition. For $s\in\C$ we define
$$\G(s)=\int_0^\infty e^{-t}t^s\,\dfrac{dt}{t}\;.$$
It is immediate to see that this converges if and only if $\Re(s)>0$ (there is
no problem at $t=\infty$, the only problem is at $t=0$), and integration by
parts shows that $\G(s+1)=s\G(s)$, so that if $s=n$ is a positive integer,
we have $\G(n)=(n-1)!$. We can now \emph{define} $\G(s)$ for
all complex $s$ by using this recursion backwards, i.e., setting
$\G(s)=\G(s+1)/s$. It is then immediate to check that $\G(s)$ is a meromorphic
function on $\C$ having poles at $s=-n$ for $n=0$, $1$, $2$,\dots, which
are simple with residue $(-1)^n/n!$.

The gamma function has numerous additional properties, the most important
being recalled below:

\begin{enumerate}
\item (Stirling's formula for large $\Re(s)$): as $s\to\infty$, $s\in\R$ (say,
there is a more general formulation) we have
$\G(s)\sim s^{s-1/2}e^{-s}(2\pi)^{1/2}$.
\item (Stirling's formula for large $\Im(s)$): as $|T|\to\infty$, $\sigma\in\R$
being fixed (say, once again there is a more general formulation), we have
$|\G(\sigma+iT)|\sim |T|^{\sigma-1/2}e^{-\pi |T|/2}(2\pi)^{1/2}$. In particular,
it tends to $0$ exponentially fast on vertical strips.
\item (Reflection formula): we have $\G(s)\G(1-s)=\pi/\sin(\pi s)$.
\item (Duplication formula): we have $\G(s)\G(s+1/2)=2^{1-2s}\pi^{1/2}\G(2s)$
(there is also a more general distribution formula giving
$\prod_{0\le j<N}\G(s+j/N)$ which we do not need). Equivalently, if we set
$\G_{\R}(s)=\pi^{-s/2}\G(s/2)$ and $\G_{\C}(s)=2\cdot(2\pi)^{-s}\G(s)$, we have
$\G_{\R}(s)\G_{\R}(s+1)=\G_{\C}(s)$.
\item (Link with the beta function): let $a$ and $b$ in $\C$ with $\Re(a)>0$
and $\Re(b)>0$. We have
$$B(a,b):=\int_0^1t^{a-1}(1-t)^{b-1}\,dt=\dfrac{\G(a)\G(b)}{\G(a+b)}\;.$$
\end{enumerate}

\subsection{Order of a Function: Hadamard Factorization}

Let $F$ be a holomorphic function in the whole of $\C$ (it is immediate
to generalize to the case of meromorphic functions, but for simplicity we
stick to the holomorphic case). We say that $F$ has \emph{finite order} if
there exists $\al\ge0$ such that as $|s|\to\infty$ we have
$|F(s)|\le e^{|s|^{\al}}$. The infimum of such $\al$ is called the order of
$F$. It is an immediate consequence of Liouville's theorem that functions
of order $0$ are polynomials. Most functions occurring in number theory,
and in particular all $L$-functions occurring in this course, have order $1$.
The Selberg zeta function, which we do not consider, is also an interesting
function and has order $2$.

The Weierstrass--Hadamard factorization theorem is the following:

\begin{theorem} Let $F$ be a holomorphic function of order $\rho$, set
$p=\lfloor\rho\rfloor$, let $(a_n)_{n\ge1}$ be the non-zero zeros of $F$
repeated with multiplicity, and let $m$ be the order of the zero at $z=0$.
There exists a polynomial $P$ of degree at most $p$ such that for all
$z\in\C$ we have
$$F(z)=z^me^{P(z)}\prod_{n\ge1}\left(1-\dfrac{z}{a_n}\right)\exp\left(\dfrac{z/a_n}{1}+\dfrac{(z/a_n)^2}{2}+\cdots+\dfrac{(z/a_n)^p}{p}\right)\;.$$
\end{theorem}

In the case of order $1$ which is of interest to us, this reads
$$F(z)=B\cdot z^me^{Az}\prod_{n\ge1}\left(1-\dfrac{z}{a_n}\right)e^{z/a_n}\;.$$

For example, we have
$$\sin(\pi z)=\pi z\prod_{n\ge1}\left(1-\dfrac{z^2}{n^2}\right)\text{\quad and\quad}\dfrac{1}{\G(z+1)}=e^{\ga z}\prod_{n\ge1}\left(1+\dfrac{z}{n}\right)e^{-z/n}\;,$$
where as usual $\ga=0.57721\cdots$ is Euler's constant.

\begin{exercise}\begin{enumerate}
\item Using these expansions, prove the reflection formula and the duplication
formula for the gamma function, and find the distribution formula giving
$\prod_{0\le j<N}\G(s+j/N)$.
\item Show that the above expansion for the sine function is equivalent to
the formula expressing $\z(2k)$ in terms of Bernoulli numbers.
\item Show that the above expansion for the gamma function is equivalent to
the Taylor expansion
$$\log(\G(z+1))=-\ga z+\sum_{n\ge2}(-1)^n\dfrac{\z(n)}{n}z^n\;,$$
and prove the validity of this Taylor expansion for $|z|<1$, hence of
the above Hadamard product.
\end{enumerate}
\end{exercise}

\subsection{Elliptic Curves}

We will not need the abstract definition of an elliptic curve. For us, an
elliptic curve $E$ defined over a field $K$ will be a nonsingular projective
curve defined by the (affine) generalized Weierstrass equation with
coefficients in $K$:
$$y^2+a_1xy+a_3y=x^3+a_2x^2+a_4x+a_6\;.$$
This curve has a \emph{discriminant} (obtained essentially by completing the
square and computing the discriminant of the resulting cubic), and the
essential property of being nonsingular is equivalent to the discriminant being
nonzero.

This curve has a unique point ${\cal O}$ at infinity, with projective
coordinates $(0:1:0)$. Using chord and tangents one can define an addition
law on this curve, and the first essential (but rather easy) result is that
it is an \emph{abelian group law} with neutral element ${\cal O}$, making
$E$ into an algebraic group.

In the case where $K=\Q$ (or more generally a number field), a deeper theorem
due to Mordell states that the group $E(\Q)$ of rational points of $E$ is a
\emph{finitely generated abelian group}, i.e., is isomorphic to
$\Z^r\oplus E(\Q)_{\text{tors}}$, where $E(\Q)_{\text{tors}}$ (the torsion
subgroup) is a finite group, and the integer $r$ is called the (algebraic)
\emph{rank} of the curve.

Still in the case $K=\Q$, for all prime numbers $p$ except a finite number,
we can \emph{reduce} the equation modulo $p$, thus obtaining an elliptic curve
over the finite field $\F_p$. Using an algorithm due to J.~Tate, we can find
first a \emph{minimal Weierstrass equation} for $E$, second the behavior of
$E$ reduced at the ``bad'' primes in terms of so-called \emph{Kodaira symbols},
and third the algebraic \emph{conductor} $N$ of $E$, product of the bad primes
raised to suitable exponents (and other important quantities).

The deep theorem of Wiles et al. tells us that the $L$-function of $E$
(as defined in the main text) is equal to the $L$-function of a rational
Hecke eigenform in the modular form space $M_2(\G_0(N))$, where $N$ is
the conductor of $E$.

A weak form of the Birch and Swinnerton-Dyer conjecture says that the
algebraic rank $r$ is equal to the analytic rank defined as the order of
vanishing of the $L$-function of $E$ at $s=1$.


\bigskip

\begin{thebibliography}{14}
\bibitem{Bou} N.~Bourbaki, {\it D\'eveloppement tayloriens g\'en\'eralis\'es. Formule sommatoire d'Euler--MacLaurin\/}, Fonctions d'une variable r\'eelle,
Chap. 6.
\bibitem{Coh1} H.~Cohen, {\it A Course in Computational Algebraic Number Theory
(fourth corrected printing)\/}, Graduate Texts in Math.~{\bf 138}, Springer-Verlag, 2000.
\bibitem{Coh2} H.~Cohen, {\it Advanced Topics in Computational Number Theory\/}, Graduate Texts in Math.~{\bf 193}, Springer-Verlag, 2000.
\bibitem{Coh3} H.~Cohen, {\it Number Theory I, Tools and Diophantine Equations\/}, Graduate Texts in Math.~{\bf 239}, Springer-Verlag, 2007.
\bibitem{Coh4} H.~Cohen, {\it Number Theory II, Analytic and Modern Tools\/},
Graduate Texts in Math.~{\bf 240}, Springer-Verlag, 2007.
\bibitem{Coh5} H.~Cohen, {\it A $p$-adic stationary phase theorem and 
applications\/}, preprint.
\bibitem{Coh-Str} H.~Cohen and F.~Str\"omberg, {\it Modular Forms: A Classical
Approach\/}, Graduate Studies in Math.~{\bf 179}, American Math.~Soc.,
(2017).
\bibitem{Coh-Vil-Zag} H.~Cohen, F.~Rodriguez-Villegas, and 
D.~Zagier, {\it Convergence acceleration of alternating series\/}, 
Exp.~Math.~{\bf 9} (2000), 3--12.
\bibitem{Coh-Zag} H.~Cohen and D.~Zagier, {\it Vanishing and nonvanishing theta
values\/}, Ann.~Sci.~Math.~Quebec {\bf 37} (2013), pp 45--61.
\bibitem{Dok} T.~Dokchitser, {\it Computing special values of motivic
$L$-functions\/}, Exp.~Math.~{\bf 13} (2004), 137--149.
\bibitem{Hia} G.~Hiary, {\it Computing Dirichlet character sums to a power-full
modulus\/}, ArXiv preprint 1205.4687v2.
\bibitem{Mes} J.-F.~Mestre, {\it Formules explicites et minorations
de conducteurs de vari\'et\'es alg\'ebriques\/}, Compositio Math.~{\bf 58}
(1986). pp. 209--232.
\bibitem{Mol} P.~Molin, {\it Int\'egration num\'erique et calculs de
  fonctions $L$\/}, Th\`ese, Universit\'e Bordeaux I (2010).
\bibitem{Rub} M.~Rubinstein, {\it Computational methods and experiments
in analytic number theory\/}, In: Recent Perspectives in Random Matrix Theory 
and Number Theory, F.~Mezzadri and N.~Snaith, eds (2005), pp. 407--483.
\bibitem{Tak-Mor} H.~Takashi and M.~Mori, {\it Double exponential
formulas for numerical integration\/}, Publications of RIMS, Kyoto University
  (1974), 9:721--741.
\bibitem{Ser} J.-P.~Serre, {\it Facteurs locaux des fonctions z\^eta des
vari\'et\'es alg\'ebriques (d\'efinitions et conjectures)\/},
S\'eminaire Delange--Pisot--Poitou {\bf 11} (1969--1970), exp. 19, pp. 1--15.
\end{thebibliography}
\end{document}